\documentclass[twoside, 14pt]{article}
\usepackage{amsmath,amssymb,amsthm,mathrsfs}
\usepackage[margin=2.0cm]{geometry} 
\usepackage{lipsum}
\usepackage{titlesec,hyperref}
\usepackage{fancyhdr}
\usepackage[numbers,sort&compress]{natbib}
\usepackage{color}
\usepackage[titletoc]{appendix}

\fancypagestyle{plain}
{
\fancyhf{}

}

\titleformat{\subsection}{\it}{\thesubsection.\enspace}{1.5pt}{}
\titleformat{\subsubsection}{\it}{\thesubsubsection.\enspace}{1.5pt}{}

\newtheorem{theo}{Theorem}[section]
\newtheorem{lemm}[theo]{Lemma}

\newtheorem{coro}[theo]{Corollary}
\newtheorem{prop}[theo]{Proposition}
\newtheorem{rema}{Remark}[section]
\numberwithin{equation}{section}

\allowdisplaybreaks

\def\p{\partial}
\def\px{\partial_x}

\def\py{\partial_y}
\def\pyy{\partial_{yy}}

\def\y{\langle y \rangle}
\def\k{\langle k \rangle}
\def\la{\langle}
\def\ra{\rangle}
\def\f{\frac}
\def\d{\delta}
\def\iy{\int_0^\infty}
\def\ix{\int_0^L}
\def\beq{\begin{equation}}
\def\bal{\begin{aligned}}
\def\dal{\end{aligned}}
\def\deq{\end{equation}}
\def\beqq{\begin{equation*}}
\def\deqq{\end{equation*}}
\def\e{\mathcal{E}}
\def\D{\mathcal{D}}
\def\eps{\epsilon}
\def\u{u^\eps}
\def\v{v^\eps}
\def\q{q^\eps}
\def\x{X^k_l(x)}
\def\bu{\widehat{u}}
\def\bv{\widehat{v}}
\def\bq{\widehat{q}}
\def\b{\mathcal{B}}
\def\e{\mathcal{E}}
\def\us{\bar{u}}
\def\vs{\bar{v}}
\def\bs{\bar{b}}
\def\om{\omega}

\def\sg{\sigma}
\def\gm{\gamma}

\begin{document}
\title{Higher regularity and asymptotic behavior of 2D
magnetic Prandtl model in the Prandtl-Hartmann  regime\hspace{-4mm}}
\author{Jincheng Gao$^\dag$ \quad Minling Li$^\ddag$ \quad Zheng-an Yao$^\sharp$ \\[10pt]
\small {School of Mathematics, Sun Yat-Sen University,}\\
\small {510275, Guangzhou, P. R. China}\\[5pt]
}

\footnotetext{Email: \it $^\dag$gaojch5@mail.sysu.edu.cn,
\it $^\ddag$limling3@mail2.sysu.edu.cn,
\it $^\sharp$mcsyao@mail.sysu.edu.cn}
\date{}

\maketitle

\begin{abstract}

In this paper, we investigate the higher regularity and asymptotic behavior for
the 2-D magnetic Prandtl model in the Prandtl-Hartmann regime. Due to the degeneracy of horizontal velocity near boundary, the higher regularity of solution is a tricky problem.
By constructing suitable approximated system and establishing closed energy estimate
for a good quantity(called ``quotient" in \cite{Guo-Iyer-2021}), our first result is to
solve this higher regularity problem.
Furthermore, we show the global well-posedness
and global-in-$x$ asymptotic behavior when the initial data are small
perturbation of the classical Hartmann layer in Sobolev space.
By using the energy method to establish closed estimate for the quotient,
we overcome the difficulty arising from the degeneracy of horizontal velocity
near boundary. Due to the damping effect, we also point out that this global solution
will converge to the equilibrium state(called Hartmann layer) with exponent decay rate.

\vspace*{5pt}

\vspace*{5pt}
\end{abstract}

\tableofcontents

\section{Introduction}\label{introduction}
In this paper, we are concerned with the higher regularity and asymptotic
behavior for the mixed Prandtl-Hartmann boundary layer equations
that were derived in \cite{G-V-P} from the classical incompressible
magnetohydrodynamic(MHD) system in two-dimensional domain with flat boundary.
Specifically, we are concerned with the following Prandtl-Hartmann equation
in $[0, L]\times \mathbb{R}_+$:
\beq\label{o-MHD}
  	\left\{\begin{aligned}
  &u \px u +v \py u-\pyy u=-\px p_E(x, 0)+\py b,\\
  &\py u+\pyy b=0,\\
  &\px u+\py v=0,
  	\end{aligned}\right.
  \deq
where $(u, v)$ denotes the velocity field, and
$b$ is the corresponding tangential magnetic component.
The quantity $p_E(x, 0)$ above is considered prescribed,
and $\px p_E(x, 0)$ evidently acts as a forcing term
to the Prandtl equations. In this paper, we are concerned with the
homogeneous Prandtl equations, that is $\px p_E(x, 0)=0$.

The Prandtl-Hartmann Eq.\eqref{o-MHD} are though of as evolution equation,
with $x$ being a time-like variable, and $y$ being space-like.
Thus, the Eq.\eqref{o-MHD} are supplemented with boundary conditions at
$y=0, y=+\infty$, and initial data at $x=0$.
The quantity $L$ appearing in \eqref{o-MHD} is though of as the time
over which we are considering the evolution.
For the Eq.\eqref{o-MHD}, we will consider this system with the initial data
\beq\label{id}
u(x, y)|_{x=0}=u_0(y),
\deq
and the Dirichlet boundary conditions
\beq\label{bc1}
u|_{y=0}=v|_{y=0}=b|_{y=0}=0.
\deq
The far field is taken as a uniform constant state,
and hence, we also suppose
\beq\label{bc2}
\underset{y\rightarrow +\infty}{\lim}u(x, y)=1,~~
\underset{y\rightarrow +\infty}{\lim}b(x, y)=1,
{\rm ~uniformly~with~respect~to~} x.
\deq
Using the boundary conditions \eqref{bc1}, \eqref{bc2} and
the second equation in \eqref{o-MHD}, one can obtain the relationship
\beq\label{re-eq}
\py b=-(u-1).
\deq
Substituting this equation \eqref{re-eq} into the original Eq.\eqref{o-MHD}, we have
\beq\label{H-eq}
  	\left\{\begin{aligned}
  &u \px u +v \py u-\pyy u+u-1=0,\\
  &\px u+\py v=0,
  	\end{aligned}\right.
  \deq
with the boundary conditions
\beq\label{H-bc}
u|_{y=0}=v|_{y=0}=0,\
\underset{y\rightarrow +\infty}{\lim}u(x, y)=1.
\deq
Since the vertical velocity $v(x, y)=-\int_0^y \px u(x, y')dy'$
creates a loss of $x-$derivative,
the classical method employed by Oleinik is to pass to the following
change of coordinates, known as the well-known von-Mises transform:
\beq\label{Von-Mises}
(x,\psi)=(x, \int_0^y u(x, y')dy').
\deq
In terms of new variables, the original Eq.\eqref{H-eq} translate into a quasilinear, degenerate diffusion equation as follows
\beq\label{u2eq}
(u^2)_x-u(u^2)_{\psi\psi}=2(1-u).
\deq
The corresponding boundary and initial conditions are\beq\label{u2bc}
u|_{\psi=0}=0,\quad
\underset{\psi \rightarrow +\infty}{\lim}u(x,\psi)=1,\quad
u(0,\psi)=u_0(\psi).
\deq
In this paper, the words ``local" and ``global" refer to the $x-$direction.
Similar to the classical Prandtl equation, higher regularity is a tricky problem
that arises from the degeneracy of velocity near the boundary.
Thus, our first target is to establish the local well-posedness
with higher regularity for the equations \eqref{o-MHD}-\eqref{bc2}.

\begin{theo}[Higher Regularity Result]\label{high-regu}
Assume that $u_0(y)\in C^\infty(\mathbb{R}^+)$ satisfies the following conditions:
\beq
u_0(y) >0, \ y>0, \ u_0(0)=0, \ u_0(y) \rightarrow 1, \ y \rightarrow \infty,
\deq
and
\beq
u_0'(0)>0, \ u_0''(0)+1=0, \ u_0'''(0)-u_0'(0)=0.
\deq
Assume the initial data $|u_0(y)-1|$ and $\py^m u_0(y)$ decay exponentially for all $m\ge 1$.
Assume also the generic compatibility conditions at the corner $(0, 0)$ up to order $2k-1$.
Then there exists a positive $0<L<<1$ depending on the initial data
such that on $0\le x \le L$, the solution $(u, v, b)$ of
Prandtl-Hartmann Eqs.\eqref{o-MHD}-\eqref{bc2}
obeys the following estimates for $0<2\alpha+\beta\le 2k$ and for $\gamma \le k-1$ that
\beq\label{THM_estimate}
\|\px^\alpha \py^\beta u \y^l\|_{L^\infty_x L^\infty_y}
+\|\px^\gamma v \y^{l-1}\|_{L^\infty_x L^\infty_y}
+\|\px^\alpha \py^{\beta+1} b \y^l\|_{L^\infty_x L^\infty_y}
\le C_0,
\deq
where the weight constant $l \ge 1$
and $C_0$ depends on $k, l$ and $u_0$.
\end{theo}

Next, our second target is to study asymptotic behavior
of solution as it tends to the equilibrium state.
More precisely, we study the global stability of the Hartmann layer $(\us,\vs)=(1-e^{-y},0)$.
It is easy to check that $(\us,\vs)$ is a special solution of the system \eqref{H-eq}-\eqref{H-bc}.
Since the higher regularity result in Theorem \ref{high-regu}
require the compatibility condition on the initial data,
thus we only investigate the asymptotic behavior result of solution
in low regularity initial data.
Specifically, we have the asymptotic behavior result as follows:

\begin{theo}[Asympototic Behavior Result]\label{asym-beha}
	Assume that $u_0(y)>0$ for $y>0$; $u_0(0)=0$, $u_0'(0)>0$, $u_0(y)\rightarrow 1$ as $y\rightarrow\infty$; $u_0(y)$, $u'_0(y)$, $u''_0(y)$ are bounded for $0\leq y<\infty$ and satisfy the H{\"o}lder condition. Moreover, assume that for small $y$ the following compatiblity condition is satisfied at the point $(0,0)$:
	\[-u''_0(y)+u_0(y)-1=O(y^2).\]
For a given real number $\gamma_0>1$,
suppose that the initial data satisfy $f(0)\leq \gamma_0$.
Furthermore, there exists a positive constanst $\sigma_0$ depending on $\gamma_0$ such that 
\beq\label{phi-condition}
\|\f{\phi_0}{u_0^{3/2}}\|_{L^2_\psi}
+\|\phi_{0\psi}\|_{L^2_\psi}+\|\f{L\phi_0}{u_0^{3/2}}\|_{L^2_\psi}
+\|(L\phi_{0})_{\psi}\|_{L^2_\psi}\leq \sigma_0, \deq
	where $\phi(x,\psi):=u^2(x,y(x,\psi))-\us^2(y(\psi))$, $f(x):=\underset{x'\in I}{\sup} \Big\{\|\f{u}{\us}\|_{L^\infty_\psi},\|\f{\us}{u}\|_{L^\infty_\psi} \Big\}$, the operator $L\phi:=-u\phi_{\psi\psi}+2\frac{\phi}{\us(u+\us)}$
and $(\us,\vs):=(1-e^{-y},0)$ be the Hartmann layer.
	Then there exists a global strong solution $(u,v)$ to the
    Prandtl-Hartmann system \eqref{H-eq}.
    Moreover, the solution satisfies the decay estimate:
	\begin{equation}\label{y-decay}
	\begin{aligned}
	\|u(x,y)-\us(y)\|_{H^2_y}+\|u_{y}(x,y)-\us_y(y)\|_{L^\infty_y}\leq C e^{-x},
	\end{aligned}
	\end{equation}
	where $C$ is a positive constant independent of $x$.
\end{theo}
\begin{rema}
Since we apply the von-Mises transformation to overcome the loss of $x$ derivative,
the difference $\phi(x,\psi)$ satisfies the new system
\beq\label{phi-eq}
\phi_x+L\phi=0
\deq
in the new variable $(x, \psi)$.
Then, we establish the closed estimate for the system \eqref{phi-eq}
under the small condition \eqref{phi-condition}.
This condition can be expressed in terms of the initial data $u_0(y)$;
however, we only give condition \eqref{phi-condition}
for the sake of simplicity.
\end{rema}

\begin{rema}
Denote $\bs_{y}:=(1-\us)$, then we apply the relation \eqref{re-eq}
and decay estimate \eqref{y-decay} to obtain
\[\|b_{y}(x,y)-\bs_{y}(y)\|_{H^2_y}+\|b_{yy}(x,y)-\bs_{yy}(y)\|_{L^\infty_y}
\leq C e^{-x}, \]
where $C$ is a positive constant independent of $x$.
\end{rema}

We now review some related works involving the classical Prandtl
and MHD Prandtl type equations to the problem studied in this paper.

\textbf{(I)Some results for nonstationary boundary layer equations.}
The vanishing viscosity limit of the incompressible Navier-Stokes equations that,
in a domain with Dirichlet boundary condition, is an important problem in
both physics and mathematics.
As the viscosity coefficient tends to zero, the solution undergoes a sharp transition
from a solution of the Euler system to the zero non-slip boundary condition on boundary of the
Navier-Stokes system. This sharp transition will lead to the formation of the boundary layer.
Indeed, Prandtl \cite{Prandtl} derived the Prandtl equations for boundary layer
from the incompressible Navier-Stokes equations with non-slip boundary condition.
Now, let us introduce the related matter of well-posedness results for the Prandtl equation.
If the tangential velocity field in the normal direction to the boundary satisfies
the monotonicity condition, Oleinik \cite{{Oleinik2},{Oleinik4}} applied the Crocco transform
to establish the global-in-time regular solutions on $[0, L]\times \mathbb{R}_+$
for small $L$, and local-in-time solutions on $[0, L]\times \mathbb{R}_+$ for arbitrary
large by finite $L$.
Under the monotonicity condition and a favorable pressure gradient of the Euler flow, the
global-in-time weak solutions were obtained in \cite{Xin-Zhang-2004} for arbitrarily $L$.
It should be noted that all of the above results are achieved by using Crocco transformations.
By taking care of the cancelation in the convection term to overcome the
loss of derivative in the tangential direction of velocity, the researchers
in \cite{Xu-Yang-Xu} and \cite{Masmoudi} independently used the simply
energy method to establish well-posedness theory for the
two-dimensional Prandtl equations in the framework of Sobolev space.
For more results in this direction, the interested readers can refer to
the well-posedness results in the analytic or Geverey setting without monotonicity
\cite{{Sammartino-Caflisch1},{Sammartino-Caflisch2},{Kukavica-Vicol-2013},{Li-Yang-2020},
{Lombardo-Cannone-Sammartino-2003},{Ignatova-Vicol-2016},{David-Masmoudi-2015}},
ill-posedness results in the Sobolev setting without monotonicity
\cite{{David-Dormy-2010},{David-Nguyen-2012}}, generic invalidity of boundary
layer expansions in the Sobolev spaces \cite{{Grenier-Guo-Nguyen-2015},{Grenier-Guo-Nguyen-2016-01},
{Grenier-Guo-Nguyen-2016-02},{Guo-Nguyen-2011},
{Grenier-Nguyen-2017},{Grenier-Nguyen-2018},{Grenier-Nguyen-2019}}
and references therein.

\textbf{(II)Some results for stationary boundary layer equations.}
For the classical stationary Prandtl equation, the local-in-time well-posedness result
was obtained by Oleinik \cite{Oleinik3}, who used the von-Mises transformation
and maximum principle. Due to the the degenerate property  of horizontal velocity
near the boundary, it is hard to obtain the higher regularity for the stationary
Prandtl equation. This problem was settled by Guo and Iyer \cite{Guo-Iyer-2021}
since they found the good unknown quantity to establish the closed energy estimate.
For the incompressible steady Navier-Stokes equations, Guo and Nguyen \cite{Guo-Nguyen2} justified the boundary layer expansion for the flow with a non-slip boundary
condition on a moving plate.
This result has been extended to the case of a rotating disk and to the case
of nonshear Euler flows(\cite{{Iyer-ARMA},{Iyer-SIAM}}).
Recently, Guo and Iyer \cite{Guo-Iyer} studied the boundary layer expansion for the small viscous flows with the classical no slip boundary conditions or on the static plate.
This work was extend to the case of global theory in the $x-$variable
for a large class of boundary layer with sharp decay rates.
For more results about the boundary layer expansion, the reader should consult
\cite{{Li-Ding-2020},{Gao-Zhang-2020},{Iyer-PMJ-2019-1},{Iyer-PMJ-2019-2},{Iyer-PMJ-2020}}.
In terms of the asymptotic behavior of the solution, Serrin \cite{Serrin-1967}
used the maximum principle techniques to single out that
similarity solution as those which asymptotically develop downstream, whatever may be the state of motion at the initial position at $x=0$.
In the case of localized data near the Blasius solution, Iyer \cite{Iyer-2020-ARMA}
applied the energy method instead of maximum principle method for the good known quantity(called ``quotient") to establish specific convergence rate.

\textbf{(III)Some results for Prandtl type equation influencing magnetic field.}
Under the influence of electro-magnetic field, the system of
magnetohydrodynamics(denoted by MHD) is a fundamental system
to describe the movement of electrically conducting fluid,
for example plasmas and liquid metals(cf.\cite{Alfven}).
By asymptotic analysis of the incompressible MHD system, G\'{e}rard-Varet
and Prestipino \cite{G-V-P} established a systematic derivation of boundary layer
models in MHD. The most important point is that they performed some stability analysis
for the boundary layer system, and emphasized the stabilizing effect of the magnetic field.
At the same time, Liu, Xie and Yang\cite{Liu-Xie-Yang} studied the local-in-time
well-posedness for the Prandtl type equation in MHD that is built if both the hydrodynamic Reynolds numbers and magnetic Reynolds numbers tend to infinity at the same rate.
Since they found that magnetic field has good effect on stability, the well-posedness
result holds true under the condition that the initial tangential magnetic field is not zero
instead of the monotonicity condition on the tangential velocity field requiring in classical
Prandtl equation \cite{Oleinik4}.
This well-posednss result is generalized to the case of inhomogeneous incompressible flow
that is Hyperbolic-parabolic coupled equations in \cite{Gao-Huang-Yao-2021}.
Furthermore, the Prandtl ansatz boundary layer expansion for the unsteady MHD system
was justified \cite{Liu-Xie-Yang2} when no-slip boundary and perfect conducting boundary
conditions are imposed on velocity field and magnetic field respectively.
Recently there are many mathematical results on two dimensional MHD boundary layer system,
for the almost global existence \cite{Lin-Zhang-MMAS},
ill-posedness results \cite{{Liu-Xie-Yang-SCI},{Liu-Wang-Xie-Yang-JFA}},
lifespan of solution with analytic perturbation of general shear flow \cite{Xie-Yang-2019},
boundary layer expansions  for steady MHD over a moving plate \cite{Ding-Ji-Lin-2021},
and local-in-time well-posedness for compressible MHD boundary layer \cite{Huang-Liu-Yang}.
For the 2D MHD boundary layer equations in the mixed Prandtl and
Hartmann regime derived by formal multiscale expansion in \cite{G-V-P},
Xie and Yang \cite{Xie-Yang-Siam} obtained the global existence of solutions
with analytic regularity.
If the initial data is near the Hartmann layer, the solution of Prandtl-Hartmann system
is globally well-posedness in Sobolev space and will converge to the Hartmann
layer with exponent decay rate in \cite{Chen-Ren-Wang-Zhang-2020}.
However, there is no result investigating the global well-posedness
and convergence rate for the steady Prandtl-Hartmann system.
The difficulty comes from the degeneracy of horizontal velocity near the boundary
$y=0$.

\emph{Thus, motivated by \cite{Guo-Iyer-2021}, our first target is to establish
the higher regularity result of solution for the Prandtl-Hartmann system \eqref{H-eq}.
The advantage of our method here is to constructed new approximated system such that
higher regularity result can be obtained without using the property of lower regular solution.
Motivated by the recent work \cite{Iyer-2020-ARMA}, our second target is to use the energy method to build the asymptotic behavior of global solution for the Prandtl-Hartmann system \eqref{H-eq}. Due to the damping effect, we point out that this global solution will converge
to the equilibrium state(called Hartmann profile) with exponent decay rate.}

\textbf{Notations:}Through  this paper, all constant $C$ may be different in different lines.
Subscript(s) of a constant illustrates the dependence of the constant, for example,
$C_s$ is a constant depending on $s$ only.
We write $o_L(1)$ to refer to a constant that is bounded by some unspecified, perhaps small,
power of $L$: that is, $a=o_L(1)$ if $|a|\le C L^\theta$ for some $\theta>0$.
We also write the $C^\infty$ cutoff function:
\beq\label{f-cutoff}
\chi(y):=\left\{
\begin{aligned}
&1, \quad y\in [0,1),\\
&0, \quad y\in (2, \infty),
\end{aligned}
\right.
\deq
and $\chi'(y)\le 0$ for all $y>0$.
$A\lesssim  B$ $(A\gtrsim B)$ represents that there exists a positive constant $C$ such that $A\leq CB$ $(A\geq CB)$.
And $A\sim B$ stands for $A\lesssim B$ and $A\gtrsim B$.
Finally, $\mathcal{P}_i(\cdot , \cdot)$ stands for
a polynomial function independent of $\eps$,
and the index $i$ denote it changing from line to line.

\section{Difficulties and outline of our approach}

The main goal of this section is to explain main difficulties of proving
Theorems \ref{high-regu} and \ref{asym-beha} as well as our strategies for overcoming them.
In order to establish the well-posedness in some higher regularity  Sobolev space
and asymptotic behavior for the Prandtl type Eq.\eqref{H-eq}, the main difficulty
comes from the vertical velocity $v=-\py^{-1} \px u$ creating a loss of $x-$derivative.
Thus, we will introduce the method to overcome this difficulty.

First of all, we state the main idea to establish higher regularity of
Prandtl type Eq.\eqref{H-eq} in Theorem \ref{high-regu}.
In order to overcome the loss of $x-$derivative coming from the vertical velocity,
the classical method is to apply the so-called von-Mises transformation
that transforms the original Eq.\eqref{H-eq} into a single
quasi-linear parabolic equation.
Due to the degenerate property of horizontal velocity on the boundary, it is not easy
to establish higher regularity for this quasi-linear parabolic equation.
In this respect, Guo and Iyer \cite{Guo-Iyer-2021} settled this tricky subject,
arising from the classical steady Prandtl equation,
by establishing the energy estimate for the good unknown quantity $q:=\f{v}{\overline{u}}$
via the linear derivative Prandtl equation.
Here the quantity $\overline{u}$ is the solution of classical steady Prandtl equation
obtained by Oleinik \cite{Oleinik2}.
In other words, this method depends on the property of lower regular solution of
original steady Prandtl equation.
\textbf{\emph{Since there is no result about the well-posedness of Eq.\eqref{H-eq}, we hope to
find a method to establish the well-posedness in higher regularity space
without using the property of lower regular solution}}.
Thus, in this paper, we consider the following approximated system
\beq\label{app-sym}
  	\left\{\begin{aligned}
  &u^\eps \px u^\eps +v^\eps \py u^\eps-\pyy u^\eps+u^\eps-1-\eps=0,\\
  &\px u^\eps+\py v^\eps=0,
  	\end{aligned}\right.
\deq
with the boundary condition
\beq\label{bdyc}
u^\eps|_{y=0}=\eps, \quad
\lim_{y \rightarrow +\infty}u^\eps=1+\eps,\quad
v^\eps|_{y=0}=0,
\deq
and the initial data
\beq\label{ic}
u^\eps|_{x=0}=u_0^\eps(y):=u_0(y)+\eps,
\deq
for any parameter $\eps>0$.
Since the initial and boundary conditions of horizontal velocity are positive,
we can follow the idea as Guo and Nguyen(see Lemma 2.2 in \cite{Guo-Nguyen2})
to obtain the local well-posedness in higher regularity space for any fixed $\eps>0$.
However, it is worth nothing that the life time interval $[0, L^\eps]$ depends
on the parameter $\eps$.
As we hope that the solution $(\u, \v)$ of approximate system \eqref{app-sym}
with positive boundary condition will converge to the solution $(u, v)$
of original Prandtl type Eq.\eqref{H-eq}
as $\eps$ tends to zero.
To achieve the target, we need to obtain the uniform a priori estimates of solution
$(\u, \v)$ on existence time independent of $\eps$.
Thanks to the recent paper \cite{Guo-Iyer-2021}, the so called quotient $\py(\f{\v}{\u})$
is a good unknown for us to avoid the loss of derivative
and the degeneracy of quasi-linear parabolic equation to establish closed a priori estimate.
And hence, we need to control the quantity $\f{\v}{\u}$ by the good unknown $\py(\f{\v}{\u})$
by using the Hardy inequality.
Thus, our method here is to establish the energy estimate including
energy part and dissipative one both with suitable weighted $\y^l$.
Then, the local well-posedness of the Eq.\eqref{H-eq} in higher regular space
follows directly as the parameter $\eps$ tends to zero.

Next, we state the main approach to obtain the asymptotic behavior of solution for the Prandtl Eq.\eqref{H-eq} in Theorem \ref{asym-beha}.
For classical Prandtl equation, in the case of localized data near the Blasius solution, Iyer \cite{Iyer-2020-ARMA} applied the energy method instead of maximum principle method for the
good known quantity(called ``quotient") to establish specific convergence rate for the quantity $\phi(x,\psi):=u^2(x,\psi)-\us^2(\psi)$.
However, he did not build any convergence rate for the quantity $u(x,y)-\us(y)$.
\textbf{\emph{Since there is no result about the asymptotic behavior of Eq.\eqref{H-eq}, we hope to find a method to establish the convergence rate in physical variable}}.
By using the method as in \cite{Iyer-2020-ARMA}, we first show the global existence of the solution for 2-D Prandtl-Hartmann system via standard continuity argument.
It should be noted that the term $\f{\phi}{\us(u+\us)}$ in the equation \eqref{phi-eq} plays an important role as a damping term. Thus, we can get
that the difference $\phi$ decays exponentially in $x$ variable, that is
\beq\label{phi-decay}
\begin{aligned}	\|\phi\|_{H^1_\psi}^2+\|\f{\phi}{\sqrt{u}}
\|_{L^2_\psi}^2+\|\frac{\phi}{u^{\f32}}\|_{L^2_\psi}^2
+\|\phi_{x}\|_{H^1_\psi}^2+\|\f{\phi_{x}}{\sqrt{u}}\|_{L^2_\psi}^2
+\|\frac{\phi_{x}}{u^{\f32}}\|_{L^2_\psi}^2
\leq Ce^{-x}.
\end{aligned}
\deq
More details of the proof for this decay estimate can be found in Proposition \ref{them1} in Section \ref{Asymptotic}.

The asymptotic behavior \eqref{phi-decay} occurs at equal values of the stream function $\psi$, but rather than at equal values of the physical variable $y$.
In order to obtain the decay estimate in $x$ variable of the difference of $u$ and $\us$ as functions of $(x,y)$, as required for \eqref{y-decay}, we need further and more arguments.
For any given value of the stream function $\psi$, we denote $y_1$ and $y_2$ represent the corresponding physical variable for the flows $u$ and $\us$, respectively.
Motivated by \cite{Serrin-1967}, we decompose the normal gradient of the tangential velocity as the following two parts:
\beqq
\begin{aligned}
	|u_{y}(x,y_1)-\us_y(y_1)|\leq |u_{y}(x,y_1)-\us_y(y_2)|+|\us_{y}(y_2)-\us_y(y_1)|.
\end{aligned}
\deqq
The first part can be transformed into von-Mises variable and then controlled by the exponential decay estimate \eqref{phi-decay}.
After some basic calculations, the second part about the solution $\us$ of Hartmann boundary layer can also be controlled by the exponential decay estimate \eqref{phi-decay}.
Therefore, we conclude that $(u_y-\us_y)$ decays exponentially in $x$ variable in $L^\infty_y-$norm.
It then follows from the similar way above to show that $(u-\us)$ also decay exponentially in $x$ variable in $H^2_y-$norm.
The main ideas and strategy of establishing the decay estimates in $x$ variable of the difference of $u$ and $\us$ as functions of $(x,y)$ will be given in Section \ref{proofthem4}.

\section{Local existence and uniqueness of Prandtl-Hartmann system}

In this section, we will establish the local existence and uniqueness
of solution with higher regularity for the Prandtl-Hartmann system
\eqref{o-MHD}-\eqref{bc2} in Theorem \ref{high-regu}.
In section \ref{uniform}, we will establish some uniform estimates with respect to the
parameter $\eps$ for the approximated system \eqref{app-sym}.
The local existence and uniqueness of original Prandtl-Hartmann system
are be investigated in sections \ref{Local_existence} and \ref{uniqueness} respectively.

\subsection{Uniform estimate for approximated system}\label{uniform}

In this subsection, we will establish uniform estimate with respect to
the parameter $\eps$ for the approximated system \eqref{app-sym} such that there exists
a positive time $L$ independent of $\eps$.
Motivated by the recent work \cite{Guo-Iyer-2021},
we used the divergence-free condition $\eqref{app-sym}_2$
to rewrite the approximated system \eqref{app-sym} as follows
\beqq
-(\u)^2 \py\left(\f{\v}{\u}\right)-\pyy \u+\u-1-\eps=0.
\deqq
Apply the $\px-$operator to the above equation
and use the divergence-free condition $\eqref{app-sym}_2$, we have
\beqq
-\px \left\{ (\u)^2 \py \left(\f{\v}{\u}\right)\right\}
+\p_y^3 \v-\py \v=0.
\deqq
Let us denote the notation(called ``quotient" in \cite{Guo-Iyer-2021})
\beq\label{def-q}
\q:=\f{\v}{\u},
\deq
then it holds
\beq\label{app-q}
-\px \{ (\u)^2 \py \q \}+\p_y^3 \v-\py \v=0.
\deq

Due to the conditions of initial data $u_0(y)$ in Theorem \ref{high-regu},
there exists a small $0<\delta_0 \ll1$ such that
\beqq
\f12 y \le u_0(y) \le \f32 y, \quad \forall y\in [0, \delta_0]
\deqq
and
\beqq
u_0(y)\ge \f12 \delta_0, \quad \forall y\in [\delta_0,+\infty).
\deqq
Then, the initial data $u_0^\eps(y)$(see the definition in \eqref{ic}) will satisfy
\beq\label{cond-1}
\f12 (y+\eps) \le \u_0(y) \le \f32 (y+\eps), \quad \forall y\in [0, \delta_0]
\deq
and
\beq\label{cond-2}
\u_0(y)\ge \f{1}{2}(\delta_0+\eps), \quad \forall y\in [\delta_0,+\infty).
\deq

Let us define the norms:
\beq\label{energy-e}
\mathcal{E}^k_l(x)
:=\sum_{j=0}^k \|\u \px^j \py \q \y^{l} \|_{L^\infty_x L^2_y}
  +\sum_{j=0}^k \|\sqrt{\u}\px^j \py^2 \q \y^l\|_{L^2_x L^2_y},
\deq
\beq\label{energy-x}
X^k_l(x):=\mathcal{E}^k_l(x)
+\sum_{j=0}^{k-1} \sum_{i=1}^{3}\|\px^j \py^i \v \y^l \|_{L^2_x L^2_y}
+\|\py^4 \v \chi(\f{y}{\d})\|_{L^2_x L^2_y},
\deq
and
\beq\label{initial-data}
\bal
\mathcal{B}^k_l(x)
:=&\sum_{j=1}^2 \|\py^j \u(x, y)\y^l\|_{L^\infty([0, x])L^2_y}
   +\|\py^3 \u(x, y)\chi(\f{y}{\d})\|_{L^\infty([0, x])L^2_y}\\
   &+\sum_{j=1}^3 \|\px^{\la k-2\ra}\py^j \v(x, y)\y^l\|_{L^\infty([0, x])L^2_y}
   +\|\px^{\la k-1 \ra}\py \q(x, y)\y^l\|_{L^\infty([0, x])L^2_y}\\
   &+\|\px^{\la k-1 \ra}\py^2 \q(x, y) (1-\chi(\f{y}{\d}))\y^l\|_{L^\infty([0, x])L^2_y}.
\dal
\deq
Thus, the parameter $\delta$ appearing in cutoff function
$\chi(\f{y}{\delta})$(see \eqref{f-cutoff})
will be chosen to satisfy $\delta\le \delta_0$
such that the assumption conditions \eqref{cond-1}
and \eqref{cond-2} hold on.

\begin{prop}[Uniform estimate]\label{Uniform estimate}
Let $k \ge 2$ be an integer, $l \ge 1$ be a real number and $\eps \in (0, 1]$,
and $(\u, \v)$ be sufficiently  smooth solution, defined on $[0, L^\eps]$,
to the non-degenerated approximated system \eqref{app-sym}-\eqref{ic}.
The initial data $u_0^\eps(y)$ is defined by \eqref{ic} and
satisfies the conditions \eqref{cond-1} and \eqref{cond-2}. Then, there exists a time
$L_a=L_a(k, l, \d_0, u_0)>0$
independent of $\eps$ such that the following a priori estimate holds true for all
$x \in [0, L_a]$:
\beq\label{uniform-1}
X_l^k(x)^2+\b^k_l(x)^2
\le 2\overline{C}_{k, l,\delta_0}(1+C(u_0)),
\deq
and
\beq\label{uniform-2}
\f14 y \le \u(x, y) \le 2(y+\eps), \ (x, y) \in [0, L_a] \times [0, \d_0]; \quad
\u(x, y)\ge \f14 \d_0, \ (x, y) \in [0, L_a] \times [\d_0, \infty).
\deq
Furthermore, it also holds
\beq\label{uniform-estimate-mix}
\sum_{0\le 2\alpha+\beta \le 2k+1}
\|\px^\alpha \py^{\beta}\py \v\y^l\|_{L^2_x L^2_y}
\le {C}_{k, l,\delta_0}(1+C(u_0)),
\deq
where $C(u_0)$ is a constant depends on the initial data $u_0$.
\end{prop}

Throughout this subsection, for some positive constants $k_*$
and $k^*(>k_*)$, we assume that the following a priori assumptions:
\beq\label{equi-relation}
k_* (y+\eps) \le \u(x, y) \le k^* (y+\eps),
\quad (x, y)\in [0, L^\eps] \times [0, \delta_0],
\deq
and
\beq\label{lower-bound}
\u(x, y)\ge k_* (\delta_0+\eps),\quad (x, y) \in [0, L^\eps] \times [\delta_0, +\infty)
\deq
hold on.
Here the parameter $\delta_0$ is defined before.

\begin{lemm}\label{lemma-one}
Let $(\u, \v)$ be the smooth solution, defined on $[0, L^\eps]$,
of the approximated equations \eqref{app-sym}-\eqref{ic}.
Under the conditions of \eqref{equi-relation} and \eqref{lower-bound}, then we have the
following estimates:
\begin{align}
\label{231}&\|\px^k \py^2 \v \y^l\|_{L^2_x L^2_y}
\le C_{k_*,k^*,l,k,\delta_0}(1+\b^k_l(0)^2+X^k_l(x)^2);\\
\label{232}&\|\px^{\la k\ra} \py v \y^l \|_{L^2_x L^2_y}
           \le  o_L(1)C_{k_*,l,k,\delta_0}(1+\b^k_l(0)^2+X^k_l(x)^2);\\
\label{233}&\|\px^{\la k-1\ra} \py^3 \v \y^l\|_{L^2_x L^2_y}
           \le o_L(1)C_{k_*,l,k,\delta_0}(1+\b^k_l(0)^4+X^k_l(x)^4);\\
\label{234}&\|\px^{\la k-1 \ra} \py^2 v \y^l\|_{L^2_x L^2_y}
           \le o_L(1)C_{k_*,l,k,\delta_0}(1+\b^k_l(0)^4+X^k_l(x)^4);\\
\label{235}&\|\px^{\la k-1 \ra} \py^4 v \y^l\|_{L^2_x L^2_y}
           \le C_{k_*,k^*,l,k,\delta_0}(1+\b^k_l(0)^4+X^k_l(x)^4);\\
\label{236}&\|\py^4 \v \chi(\f{y}{\d})\|_{L^2_x L^2_y}
           \le  o_L(1)C_{k_*,l,k,\delta_0}(1+\b^k_l(0)^8+X^k_l(x)^8),
\end{align}
where $l \ge 1$ and $k \ge 2$.
\end{lemm}

\begin{proof}
Step 1: Using the definition of $\q$(see \eqref{def-q}), we can get
\beqq
\bal
\px^k \py^2 \v
=
&\py^2 \u \px^k \q +2\py \u \px^k \py \q +\u \px^k \py^2 \q
 +\sum_{j=1}^{k}C^k_j \px^j \u \px^{k-j} \py^2 \q\\
& +2\sum_{j=1}^k C^k_j \px^j \py \u \px^{k-j} \py \q
 +\sum_{j=1}^k C^k_j \px^j \py^2 \u \px^{k-j}\q\\
:=
&I_{11}+I_{12}+I_{13}+I_{14}+I_{15}+I_{16}.
\dal
\deqq
Using estimates \eqref{claim21}-\eqref{claim31}
in Appendix \ref{appendix-a}, it is easy to check that
\beqq
\bal
\|I_{11}\y^l\|_{L^2_x L^2_y}
 &\le \|\py^2 \u \y^l\|_{L^\infty_x L^2_y}
     \|\px^k \q \|_{L^2_x L^\infty_y}\\
 &\le o_L(1)C_{k_*,l}
      (\|\py^2 \u \y^l|_{x=0}\|_{L^2_y}+\|\py^3 \v \y^l\|_{L^2_x L^2_y})\e^k_l(x);\\
\|I_{12}\y^l\|_{L^2_x L^2_y}
 &\le \|\py \u\|_{L^\infty_x L^\infty_y}
     \|\px^k \py \q \y^l\|_{L^2_x L^2_y}\\
 &\le o_L(1)C_{k_*}(\|\py \u \y|_{x=0}\|_{L^2_y}+\|\py^2 \v \y\|_{L^2_x L^2_y})\e^k_l(x);\\
\|I_{13}\y^l\|_{L^2_x L^2_y}
 &\le \|\u\|_{L^\infty_x L^\infty_y}^{\f12}
      \|\sqrt{\u}\px^k \py^2 \q \y^l\|_{L^2_x L^2_y}\\
 &\le C(1+\|\py \u \y|_{x=0}\|_{L^2_y}+\|\py^2 \v \y\|_{L^2_x L^2_y})\e^k_l(x);\\
 \|I_{15}\y^l\|_{L^2_x L^2_y}
 &\le \|\px^{\la k-1\ra} \py^2 \v\|_{L^2_x L^\infty_y}
      \|\px^{\la k-1\ra} \py \q \y^l\|_{L^\infty_x L^2_y}\\
 &\le C\|\px^{\la k-1\ra} \py^3 \v \y\|_{L^2_x L^2_y}
      (\|\px^{\la k-1\ra} \py \q \y^l|_{x=0}\|_{L^2_y}+o_L(1)C_{k_*,\delta_0}\e^k_l(x));\\
 \|I_{16}\y^l\|_{L^2_x L^2_y}
 &\le \|\px^{\la k-1\ra}\py^3 \v \y^l\|_{L^2_x L^2_y}
      \|\px^{\la k-1 \ra}\q \|_{L^\infty_x L^\infty_y}\\
 &\le \|\px^{\la k-1\ra}\py^3 \v \y^l\|_{L^2_x L^2_y}
      (\|\px^{\la k-1\ra} \py q \y |_{x=0}\|_{L^2_y}
        +o_L(1)C_{k_*, l, \delta_0}\e^k_l(x)).
\dal
\deqq
Based on the combination of the above estimates, we have for all $l \ge 1$
\beq\label{i116}
\bal
&\|I_{11}\y^l\|_{L^2_x L^2_y}+
\|I_{12}\y^l\|_{L^2_x L^2_y}+
\|I_{13}\y^l\|_{L^2_x L^2_y}+
\|I_{15}\y^l\|_{L^2_x L^2_y}+
\|I_{16}\y^l\|_{L^2_x L^2_y}\\
\le
&C_{k_*, l, \delta_0}(1+\b^k_l(0)^2+\x^2).
\dal
\deq
Let us write
\beqq
I_{14}
=\sum_{j=1}^{k}C^k_j \px^j \u \px^{k-j} \py^2 \q \chi
    +\sum_{j=1}^{k}C^k_j \px^j \u \px^{k-j} \py^2 \q (1-\chi)
:=I_{141}+I_{142}.
\deqq
First of all, let us deal with the $I_{141}$ term.
Due to $k\ge 2$, using H\"{o}lder inequality,
estimates \eqref{claim1} and \eqref{claim4}, we can obtain
\beq\label{i141}
\bal
\|I_{141}\y^l\|_{L^2_x L^2_y}
\le
& C_k(\d_0+\eps)^{\f12}
  \|\f{\px^{\la [\f{k-1}{2}] \ra}\px \u}{\u} \chi\|_{L^\infty_x L^\infty_y}
  \|\px^{\la k-1 \ra}\py^2 \q \chi \|_{L^2_x L^2_y}\\
& +C_k (\d_0+\eps)^{\f12}\|\f{\px^{\la k-1 \ra}\px \u}{\u} \chi\|_{L^2_x L^\infty_y}
  \|\px^{\la [\f{k-1}{2}] \ra}\py^2 \q \chi \|_{L^\infty_x L^2_y}\\
\le
&o_L(1)C_{k,k_*, k^*,\d_0}(\|\px^{\la [\f{k-1}{2}] \ra} \py \v|_{x=0}\|_{H^2_y}
             +\|\px^{\la [\f{k-1}{2}]+1\ra} \py \v\|_{L^2_x H^2_y})
             \|\px^{\la k-1 \ra}\py^3 \v\|_{L^2_x L^2_y}\\
&+o_L(1)C_{k,k_*,\d_0} \|\px^{\la k-1 \ra}\py^3 \v\|_{L^2_x L^2_y}
           (\|\px^{\la [\f{k-1}{2}] \ra} \py^3 \v|_{x=0}\|_{L^2_y}
             +\|\px^{\la [\f{k-1}{2}]+1\ra} \py^3 \v\|_{L^2_x L^2_y})\\
\le
&o_L(1)C_{k_*,k^*,k,\delta_0}(1+\b^k_l(0)^2+\x^2).
\dal
\deq
Here the condition $k\ge 2$ implies directly $[\f{k-1}{2}]\le k-2$.
Using the H\"{o}lder inequality, we have
\beq\label{i41}
\bal
\|I_{142}\y^l\|_{L^2_x L^2_y}
\le
&C_k\|{\px^{\la k-1 \ra}\px \u}(1-\chi)\|_{L^2_x L^\infty_y}
  \|\px^{\la k-1 \ra}\py^2 \q (1-\chi) \y^l\|_{L^\infty_x L^2_y}.
\dal
\deq
Away from the boundary $y=0$, we have
\beq\label{i42}
\bal
&\|\px^{\la k-1 \ra}\py^2 \q (1-\chi) \y^l\|_{L^\infty_x L^2_y}\\
\le &\|\px^{\la k-1 \ra}\py^2 \q (1-\chi) \y^l|_{x=0}\|_{L^2_y}
     +o_L(1)\|\px^{\la k \ra}\py^2 \q (1-\chi) \y^l\|_{L^2_x L^2_y}\\
\le &\|\px^{\la k-1 \ra}\py^2 \q (1-\chi)\y^l|_{x=0}\|_{L^2_y}
     +o_L(1)C_{k_*,\d_0}\|\sqrt{\u}\px^{\la k \ra}\py^2 \q \y^l\|_{L^2_x L^2_y}.
\dal
\deq
Substituting the estimate \eqref{i42} into \eqref{i41}
and applying the divergence-free condition, we have
\beqq
\|I_{142}\y^l\|_{L^2_x L^2_y}
\le \|{\px^{\la k-1 \ra}\py \v}\|_{L^2_x H^1_y}
     (\|\px^{\la k-1 \ra}\py^2 \q (1-\chi)\y^l|_{x=0}\|_{L^2_y}
     +o_L(1)C_{k_*}\|\sqrt{\u}\px^{\la k \ra}\py^2 \q \y^l\|_{L^2_x L^2_y}).
\deqq
which, together with estimate \eqref{i141}, yields directly
\beqq
\|I_{14}\y^l\|_{L^2_x L^2_y}
\le o_L(1)C_{k_*,k^*,k,\delta_0}(1+\b^k_l(0)^2+\x^2).
\deqq
This and the estimate \eqref{i116} imply directly for all $l \ge 1$
\beqq
\|\px^k \py^2 \v \y^l\|_{L^2_x L^2_y}
\le C_{k_*,k^*,l,k,\delta_0}(1+\b^k_l(0)^2+ X^k_l(x)^2),
\deqq
which yields the estimate \eqref{231}.

Step 2: Obviously, it holds
\beqq
\bal
\px^k \py \v
&=\px^k(\py \u \q)+\px^k (\u \py \q)\\
&=(\py u \px^k \q +\u \px^k \py \q)
 +\sum_{j=1}^{k}C^k_j \px^j \py \u \px^{k-j}\q
 +\sum_{j=1}^k C^k_j \px^j \u \px^{k-j}\py \q\\
&:=I_{21}+I_{22}+I_{23}.
\dal
\deqq
Using H\"{o}lder inequality and estimate \eqref{claim31},
it holds for all $l \ge 1$
\beqq
\bal
\|I_{21}\y^l \|_{L^2_x L^2_y}
&\le
\|\py \u\y^l\|_{L^\infty_x L^2_y}\|\px^{k} \q \|_{L^2_x L^\infty_y}
  +\sqrt{L}\|\u \px^k \py \q \y^l\|_{L^\infty_x L^2_y}\\
&\le o_L(1)C_{k_*,l}
 (1+\|\py \u \y^l |_{x=0}\|_{L^2_y}
    +o_L(1)\|\py^2 \v \y^l \|_{L^2_x L^2_y})\e^{k}_l(x).
\dal
\deqq
Using the divergence-free condition,
and estimates \eqref{claim31} and \eqref{claim32}, it holds for $k\ge 2$
\beqq
\bal
\|I_{22}\y^l \|_{L^2_x L^2_y}
\le
& \|\px^{\la [\f{k-1}{2}] \ra} \py^2 \v \y^l\|_{L^\infty_x L^2_y}
  \|\px^{\la k-1 \ra}\q \|_{L^2_x L^\infty_y}\\
&+\|\px^{\la k-1 \ra} \py^2 \v \y^l\|_{L^2_x L^2_y}
  \|\px^{\la [\f{k-1}{2}] \ra}\q \|_{L^\infty_x L^\infty_y}\\
\le
& o_L(1)C_{k_*,k,l}(\|\px^{\la [\f{k-1}{2}] \ra} \py^2 \v\y^l|_{x=0}\|_{L^2_y}
   +\|\px^{\la [\f{k-1}{2}]+1\ra} \py^2 \v \y^l\|_{L^2_x L^2_y})\e^{k-1}_l(x)\\
&+o_L(1)\|\px^{\la k-1 \ra} \py^2 \v \y^l\|_{L^2_x L^2_y}
    (\|\px^{\la [\f{k-1}{2}] \ra}\py \q \y^{l}|_{x=0}\|_{L^2_y}
     +o_L(1)C_{k_*,l,\delta_0}\e^{[\f{k-1}{2}]+1}_l(x))\\
\le
&o_L(1)C_{k_*,k,l,\delta_0}(1+\b^k_l(0)^2+X^k_l(x)^2).
\dal
\deqq
Using the divergence-free condition,
and estimates \eqref{claim21} and \eqref{claim22}, it holds for $k \ge 2$
\beqq
\bal
\|I_{23}\y^l \|_{L^2_x L^2_y}
\le
&C_k \|\px^{\la [\f{k-1}{2}] \ra}\py \v \|_{L^\infty_x L^\infty_y}
  \|\px^{\la k-1 \ra}\py \q \y^{l}\|_{L^2_x L^2_y}\\
&+C_k\|\px^{\la k-1 \ra} \py \v\|_{L^2_x L^\infty_y}
  \|\px^{\la [\f{k-1}{2}] \ra}\py \q \y^{l}\|_{L^\infty_x L^2_y}\\
\le
&o_L(1)C_{k_*,k}(\|\px^{\la [\f{k-1}{2}] \ra}\py^2 \v \y |_{x=0}\|_{L^2_y}
         +\|\px^{\la [\f{k-1}{2}]+1\ra}\py^2 \v \y\|_{L^2_x L^2_y})\e^k_l(x)\\
&+o_L(1)C_k\|\px^{\la k-1\ra} \py^2 \v \y\|_{L^2_x L^2_y}
        (\|\px^{\la [\f{k-1}{2}] \ra}\py \q \y^{l} |_{x=0}\|_{L^2_y}
         +o_L(1)C_{k_*,\delta_0}\e^{[\f{k-1}{2}]+1}_l(x))\\
\le
&o_L(1)C_{k_*,k,\delta_0}(1+\b^k_l(0)^2+X^k_l(x)^2).
\dal
\deqq
Thus, we have for all integer $k \ge 2$ and $l \ge 1$
\beqq
\|\px^k \py \v \y^l \|_{L^2_x L^2_y}
\le  o_L(1)C_{k_*,k,l,\delta_0}(1+ \b^k_l(0)^2 +X^k_l(x)^2).
\deqq
This implies directly the estimate \eqref{232}.

{Step 3:}
Using the equation \eqref{app-q}, it holds
\beq\label{i31}
\bal
&\|\px^{k-1} \py^3 \v \y^l\|_{L^2_x L^2_y}\\
\le
&\|\px^{k-1} \py \v \y^l\|_{L^2_x L^2_y}
    +\|(\u)^2 \px^k\py \q \y^l \|_{L^2_x L^2_y}
 +\sum_{j=1}^{k} \|\u \px^{j-1} \py \v \px^{k-j} \py \q \y^l \|_{L^2_x L^2_y}\\
&+C_k\sum_{j=1}^{k} \sum_{j_1=1}^{j-1}
            \|\px^{j_1-1}\py \v \px^{j-j_1-1}\py \v \px^{k-j} \py \q \y^l \|_{L^2_x L^2_y}\\
:=
&I_{31}+I_{32}+I_{33}+I_{34}.
\dal
\deq
It is easy to check that
\beq\label{i32}
\bal
|I_{32}|
\le &\|\u\|_{L^2_x L^\infty_y}  \|\u \px^{k}\py \q \y^l\|_{L^\infty_x L^2_y}\\
\le &o_L(1)(1+\|\u-1-\eps\|_{L^\infty_x L^\infty_y})
     \|\u \px^{k}\py \q \y^l\|_{L^\infty_x L^2_y}\\
\le &o_L(1)(1+\|\py \u_0 \y\|_{L^2_y}+o_L(1)\|\py^2 \v \y\|_{L^2_x L^2_y})
     \|\u \px^{k}\py \q \y^l\|_{L^\infty_x L^2_y}.
\dal
\deq
Using divergence-free condition, estimates  \eqref{claim21} and  \eqref{232},
it is easy to check that
\beq\label{i33}
\bal
|I_{33}|
\le
&C_k \|\px^{\la k-2 \ra}\py \v\|_{L^\infty_x L^\infty_y}
     \|\u \px^{\la k-1 \ra} \py \q \y^l\|_{L^2_x L^2_y}
 +C_k\|\px^{k-1} \py \v \y^l\|_{L^2_x L^2_y}
     \|\u \py \q\|_{L^\infty_x L^\infty_y}\\
\le
&o_L(1)C_k (\|\px^{\la k-2 \ra}\py^2 \v \y |_{x=0}\|_{L^2_y}
             +\|\px^{\la k-1 \ra}\py^2 \v \y\|_{L^2_x L^2_y})
     \|\u \px^{\la k-1 \ra} \py \q \y^l\|_{L^\infty_x L^2_y}\\
&+C_k\|\px^{k-1} \py \v \y^l\|_{L^2_x L^2_y}
     \|\py(\u \py \q)\y\|_{L^\infty_x L^2_y}\\
\le
&C_{k_*,\d_0}(1+\b^k_l(0)^4+\x^4),
\dal
\deq
where we have used the estimate in the last inequality
\beqq
\bal
&\|\py(\u \py \q)\y\|_{L^\infty_x L^2_y}\\
\le  &\|\py \u \py \q \y\|_{L^\infty_x L^2_y}
      +\|\u \py^2 \q \y [(1-\chi)+\chi]\|_{L^\infty_x L^2_y}\\
\le &C_{k_*,\d_0}(1+\b^k_l(0)^2+\x^2).
\dal
\deqq
Here the above terms can be estimated as follows:
\beqq
\bal
&\|\py \u \py \q \y\|_{L^\infty_x L^2_y}\\
\le  &\|\py \u\|_{L^\infty_x L^\infty_y}
      \|\py \q \y\|_{L^\infty_x L^2_y}\\
\le &(\|\py^2 \u\y|_{x=0}\|_{L^2_y}+o_L(1)\|\py^3 \v \y\|_{L^2_x L^2_y})
     (\|\py \q \y|_{x=0}\|_{L^2_y}+o_L(1)\|\px \py \q \y\|_{L^2_x L^2_y})\\
\le &C(\b^k_l(0)^2+\x^2),
\dal
\deqq
and
\beq\label{uq}
\bal
&\|\u \py^2 \q \y [(1-\chi)+\chi]\|_{L^\infty_x L^2_y}\\
\le &(\|\u \py^2 \q \y \chi |_{x=0}\|_{L^2_y}
       +\|\u \px \py^2 \q \y \chi\|_{L^2_x L^2_y})
     +(1+\|\py u_0 \y\|_{L^2_y}+o_L(1)\|\py^2 \v \y\|_{L^2_x L^2_y})\\
&\quad \times  (\|\py^2 \q \y(1-\chi)|_{x=0}\|_{L^2_y}
      +C_{k_*,\d_0}\|\sqrt{\u}\px \py^2 \q \y(1-\chi)\|_{L^2_x L^2_y})\\
\le &C_{k_*,\d_0}(1+\b^k_l(0)^2+\x^2).
\dal
\deq
Similarly, we can get
\beq\label{i34}
\bal
|I_{34}|\le
&C_k\|\px^{\la k-2 \ra}\py \v\|_{L^\infty_x L^\infty_y}^2
  \|\px^{\la k-1 \ra}\py \q \y^l\|_{L^2_x L^2_y}\\
\le
&o_L(1)C_{k_*,k}(\|\px^{\la k-1 \ra}\py^2 \v \y|_{x=0}\|_{L^2_y}
         +\|\px^{\la k-1 \ra}\py^2 \v \y\|_{L^2_x L^2_y})^2\e^{k-1}_l(x).
\dal
\deq
Substituting estimates \eqref{i32}, \eqref{i33} and \eqref{i34}
into \eqref{i31}, then it holds for all $l \ge 1$
\beqq
\|\px^{k-1} \py^3 \v \y^l\|_{L^2_x L^2_y}
\le \|\px^{k-1} \py \v \y^l\|_{L^2_x L^2_y}
    +o_L(1)C_{k_*,k,\d_0}(1+\b^k_l(0)^4+ X^k_l(x)^4),
\deqq
which, together with estimate \eqref{232}, yields directly
\beqq
\|\px^{k-1} \py^3 \v \y^l\|_{L^2_x L^2_y}
\le o_L(1)C_{k_*,l,k,\delta_0}(1+\b^k_l(0)^4+ X^k_l(x)^4),
\deqq
which yields the estimate \eqref{233}.

{Step 4:}
Integrating by part, using Hardy inequality \eqref{Hardy-one},
H\"{o}lder  and Cauchy inequalities, it holds true
\beqq
\bal
\iy |\px^{k-1} \py^2 \v|^2 \y^{2l} dy
&\le \iy |\px^{k-1} \py \v(\px^{k-1} \py^3 v\y^{2l}
+2l \px^{k-1} \py^2 \v \y^{2l-2})| dy\\
&\le \|\px^{k-1} \py \v \y^l\|_{L^2_y}\|\px^{k-1} \py^3 \v \y^l\|_{L^2_y}
     +2l\|\px^{k-1} \py \v \y^l\|_{L^2_y}\|\px^{k-1} \py^2 \v \y^{l-1}\|_{L^2_y}\\
&\le C_l \|\px^{k-1} \py \v \y^l\|_{L^2_y}\|\px^{k-1} \py^3 \v \y^l\|_{L^2_y}\\
&\le \|\px^{k-1} \py^3 \v \y^l\|_{L^2_y}^2+\|\px^{k-1} \py \v \y^l\|_{L^2_y}^2.
\dal
\deqq
Then, for any integer $k \ge 2$ and $l \ge 1$,
we use the estimates \eqref{232} and \eqref{233} to obtain
\beqq
\|\px^{k-1} \py^2 \v \y^l\|_{L^2_x L^2_y}
\le o_L(1)C_{k_*,l, k,\delta_0}(1+\b^k_l(0)^4+X^k_l(x)^4),
\deqq
which yields the estimate \eqref{234}.

{Step 5:}
Using the equation \eqref{app-q}, we have for any integer $k \ge 2$
\beqq
\bal
\px^{k-1} \py^4 \v=
&\px^{k-1} \py^2 \v+(\u)^2 \px^{k} \py^2 \q+2\py \u \u \px^{k} \py \q
 +\sum_{j=1}^{k}C^{k}_j \px^j(\u \py \u)\px^{k-j}\py \q \\
&+\sum_{j=1}^{k}C^{k}_j \px^j\{(\u)^2\}\px^{k-j}\py^2 \q
:=I_{41}+I_{42}+I_{43}+I_{44}+I_{45}.
\dal
\deqq
It is easy to check that
\beqq
\bal
\|I_{42}\y^l\|_{L^2_x L^2_y}
&\le \|\u\|_{L^\infty_x L^\infty_y}^{\f32}
     \|\sqrt{\u}\px^{k}\py^2 \q \y^l\|_{L^2_x L^2_y}\\
&\le C(1+\|\py \u \y|_{x=0}\|_{L^2_y}^2+o_L(1)\|\py^2 \v \y\|_{L^2_x L^2_y}^2)
     \|\sqrt{\u}\px^{k}\py^2 \q \y^l\|_{L^2_x L^2_y},\\
\|I_{43}\y^l\|_{L^2_x L^2_y}
&\le \|\py \u\|_{L^2_x L^\infty_y}
      \|\u \px^{k}\py \q \y^l\|_{L^\infty_x L^2_y}\\
&\le o_L(1)(\|\py^2 \u \y|_{x=0}\|_{L^2_y}+\|\py^3 \v \y\|_{L^2_x L^2_y})
     \|\u \px^{k}\py \q \y^l\|_{L^\infty_x L^2_y}.
\dal
\deqq
To deal with the term $I_{44}$, we can follow the idea
as terms $I_{33}$.
Let us write
\beqq
\bal
I_{44}
=&\sum_{j=1}^k C^k_j \u \px^j \py \u \px^{k-j}\py \q
 +\sum_{j=1}^k C^k_j \py \u \px^j \u \px^{k-j}\py \q\\
 &+\sum_{j=1}^k \sum_{j_1=1}^{j-1}C^k_j C^j_{j_1}
  \px^{j_1}\u \px^{k-j_1} \py \u \px^{k-j}\py \q
:=I_{441}+I_{442}+I_{443}.
\dal
\deqq
Using estimate \eqref{uq}, it is easy to check that
\beqq
\bal
\|I_{441}\y^l\|_{L^2_x L^2_y}
\le & C_k \|\px^{\la k-2 \ra}\py^2 \v\|_{L^\infty_x L^\infty_y }
        \|\u \px^{\la k-1 \ra} \py \q \y^l\|_{L^2_x L^2_y }
      +C_k \|\px^{\la k-1 \ra}\py^2 \v \y^l\|_{L^2_x L^2_y }
         \|\u \py \q\|_{L^\infty_x L^\infty_y }\\
\le & C_k (\|\px^{\la k-2 \ra}\py^3 \v \y|_{x=0}\|_{L^2_y }
            +o_L(1)\|\px^{\la k-1 \ra}\py^3 \v \y\|_{L^\infty_x L^2_y })
        \|\u \px^{\la k-1 \ra} \py \q \y^l\|_{L^2_x L^2_y }\\
    &+C_k \|\px^{\la k-1 \ra}\py^2 \v \y^l\|_{L^2_x L^2_y }
         \|\u \py \q\|_{L^\infty_x L^\infty_y }\\
\le &C_k(1+\b^k_l(0)^4+\x^4).
\dal
\deqq
Similarly, using estimate \eqref{claim21}, we can obtain for all $k \ge 2$
\beqq
\bal
\|I_{442}\y^l\|_{L^2_x L^2_y}
\le &C_k \|\px^{\la [\f{k-1}{2}]\ra}\py \v\|_{L^\infty_x L^\infty_y}
         \|\px^{\la k-1 \ra}\py \q \y^l\|_{L^2_x L^2_y}\|\py \u\|_{L^\infty_x L^\infty_y}\\
      &+C_k\|\px^{\la k-1 \ra}\py \v\|_{L^2_x L^\infty_y}
            \|\px^{\la [\f{k-1}{2}] \ra}\py \q \y^l\|_{L^\infty_x L^2_y}
            \|\py \u\|_{L^\infty_x L^\infty_y}\\
\le &C_k (\|\px^{\la [\f{k-1}{2}]\ra}\py^2 \v \y|_{x=0}\|_{L^2_y}
           +o_L(1)\|\px^{\la [\f{k-1}{2}]+1\ra}\py^2 \v \y\|_{L^2_x L^2_y})\\
    &\quad \times \|\px^{\la k-1 \ra}\py \q \y^l\|_{L^2_x L^2_y}
             (\|\py^2 \u \y|_{x=0}\|_{L^2_y}
              +o_L(1)\|\py^3 \v \y\|_{L^2_x L^2_y})\\
     &+C_k (\|\px^{\la [\f{k-1}{2}] \ra}\py \q \y^l|_{x=0}\|_{L^2_y}
             +\|\px^{\la [\f{k-1}{2}] +1\ra}\py \q \y^l\|_{L^2_x L^2_y})\\
     &\quad \times  \|\px^{\la k-1 \ra}\py^2 \v \y\|_{L^2_x L^2_y}
             (\|\py^2 \u \y|_{x=0}\|_{L^2_y}
              +o_L(1)\|\py^3 \v \y\|_{L^2_x L^2_y})\\
\le &C_{k_*,k}(1+\b^k_l(0)^4+\x^4),
\dal
\deqq
and
\beqq
\bal
\|I_{443}\y^l\|_{L^2_x L^2_y}
\le &C_k \|\px^{k-2}\py \v\|_{L^\infty_x L^\infty_y}
     \|\px^{k-2}\py^2 \v\|_{L^\infty_x L^\infty_y}
     \|\px^{k-1} \py \q \y^l\|_{L^2_x L^2_y}\\
\le &C_k (\|\px^{k-2}\py^2 \v\y|_{x=0}\|_{L^2_y}
            +\|\px^{k-1}\py^2 \v\y\|_{L^2_x L^2_y})
     \|\px^{k-1} \py \q \y^l\|_{L^2_x L^2_y}\\
    &\quad \times (\|\px^{k-2}\py^3 \v\y|_{x=0}\|_{L^2_y}
                     +\|\px^{k-1}\py^3 \v\y\|_{L^2_x L^2_y})\\
\le &C_{k_*,k}(1+\b^k_l(0)^4+\x^4).
\dal
\deqq
Collecting the estimates from terms $I_{441}$ to $I_{443}$, we get that
\beqq
\|I_{44}\y^l\|_{L^2_x L^2_y}\le C_{k_*,k}(1+\b^k_l(0)^4+\x^4).
\deqq
Finally, let us deal with the term $I_{45}$.
Indeed, we use the divergence-free condition to write
\beqq
\bal
I_{45}
=&\sum_{j=1}^{k} 2C^{k}_j \u  \px^j \u \px^{k-j}\py^2 \q
   +\sum_{j=1}^{k} \sum_{j_1=1}^{j-1}
    2C^{k}_j C_{j_1}^j \px^{j_1-1}\py \v
    \px^{j-j_1-1}\py \v \px\px^{k-j}\py^2 \q\\
:=&I_{451}+I_{452}.
\dal
\deqq
Using divergence-free condition, estimates \eqref{uq} and \eqref{claim21}, we can get
\beqq
\bal
\|I_{451} \y^l\|_{L^2_x L^2_y}
\le &C_k (1+\|\u-1-\eps\|_{L^\infty_x L^\infty_y})
         \|\px^{\la k-2 \ra}\py \v\|_{L^\infty_x L^\infty_y}
         \|\sqrt{\u}\px^{\la k-1\ra}\py^2 \q \y^l\|_{L^2_x L^2_y}\\
    &+C_k \|\px^{\la k-1 \ra}\py \v \y^{l-1}\|_{L^2_x L^\infty_y}
          \|\u \py^2 \q \y\|_{L^\infty_x L^2_y}\\
\le &C_k(1+\|\py \u \y|_{x=0}\|_{L^2_y}+\|\py^2 \v \y\|_{L^2_x L^2_y})
         \|\sqrt{\u}\px^{\la k-1\ra}\py^2 \q \y^l\|_{L^2_x L^2_y}\\
   &\times
         (\|\px^{\la k-2 \ra}\py^2 \v\y|_{x=0}\|_{L^2_y}
          +\|\px^{\la k-1 \ra}\py^2 \v \y\|_{L^2_x L^2_y})\\
    &+C_{k,l} \|\px^{\la k-1 \ra}\py^2 \v \y^{l}\|_{L^2_x L^2_y}
          (1+\b^k_l(0)^2+\x^2)\\
\le &C_{k,l}(1+\b^k_l(0)^4+\x^4),
\dal
\deqq
and
\beqq
\bal
\|I_{452} \y^l\|_{L^2_x L^2_y}
\le &C_k \|\px^{\la k-2 \ra}\py \v\|_{L^\infty_x L^\infty_y}^2
         \|\px^{\la k-1 \ra}\py \q \y^l\|_{L^2_x L^2_y}\\
\le &C_{k,l} (\|\px^{\la k-2 \ra}\py^2 \v \y|_{x=0}\|_{L^2_y}^2
              +\|\px^{\la k-1 \ra}\py^2 \v \y\|_{L^2_x L^2_y}^2)
         \|\px^{\la k-1 \ra}\py \q \y^l\|_{L^2_x L^2_y}\\
\le &C_{k_*,k,l}(1+\b^k_l(0)^2+\x^2).
\dal
\deqq
Then, we have the estimate
\beqq
\|I_{45}\y^l\|_{L^2_x L^2_y}
\le C_{k_*,k,l}(1+\b^k_l(0)^4+\x^4).
\deqq
The combination of estimates from terms $I_{42}$ to $I_{45}$
and the estimate \eqref{231}, it holds for all $k\ge 2$ and $l \ge 1$
\beqq
\|\px^{k-1} \py^4 \v \y^l\|_{L^2_x L^2_y}
\le  C_{k_*,k^*,l,k,\delta_0}(1+\b^k_l(0)^4+X^k_l(x)^4),
\deqq
which yields the estimate \eqref{235}.

{Step 6:}
Using the equation \eqref{app-q}, then we have
\beqq
\bal
\|\py^4 \v \chi\|_{L^2_x L^2_y}
\le
    &\|\py^2 \v \chi\|_{L^2_x L^2_y}
    +2\|\py \u \px \u \py \q \chi\|_{L^2_x L^2_y}
    +2\|\u \px \py \u \py \q \chi\|_{L^2_x L^2_y}\\
    &+2\|\u \px \u \px \py \q\chi\|_{L^2_x L^2_y}
    +2\|\u \py \u \px \py \q \chi\|_{L^2_x L^2_y}
    +\|(\u)^2 \px \py^2 \q\chi\|_{L^2_x L^2_y}\\
:=&I_{51}+I_{52}+I_{53}+I_{54}+I_{55}+I_{56}.
\dal
\deqq
Using the a priori assumption \eqref{equi-relation}, it is easy to check that
\beqq
\bal
&I_{52}\le C(\|\py^2 \u \y|_{x=0}\|_{L^2_y}
        +o_L(1)\|\py^3 \v \y\|_{L^2_x L^2_y})
        (\|\py \q|_{x=0}\|_{L^2_y}
          +o_L(1)\|\px \py \q\|_{L^2_x L^2_y})
          \|\py^2 \v \y\|_{L^2_x L^2_y};\\
&I_{53}\le k^*(\d+\eps) \|\py^3 \v \y\|_{L^2_x L^2_y}
             ((\|\py \q|_{x=0}\|_{L^2_y}
          +o_L(1)\|\px \py \q\|_{L^2_x L^2_y});\\
&I_{54}\le \sqrt{L}\|\py^2 \v \y\|_{L^2_x L^2_y}
             \|\u \px \py \q\|_{L^\infty_x L^2_y};\\
&I_{55}\le \sqrt{L}(\|\py^2 \u \y|_{x=0}\|_{L^2_y}
                 +o_L(1)\|\py^3 \v \y\|_{L^2_x L^2_y})
             \|\u \px \py \q\|_{L^2_x L^2_y};\\
&I_{56}\le (k^*)^{\f32}(\d+\eps)^{\f32}\|\sqrt{\u} \px \py^2 \q\|_{L^2_x L^2y}.
\dal
\deqq
Combining estimates from $I_{52}$ to $I_{56}$ and using the estimates
\eqref{claim21} and \eqref{234}, we can choose $\d=\sqrt{L}$ to obtain
\beqq
\|\py^4 \v \chi\|_{L^2_x L^2_y}
\le o_L(1)C_{k_*,l,k,\delta_0}(1+\b^k_l(0)^8+X^k_l(x)^8),
\deqq
which yields the estimate \eqref{236}.
Therefore, we complete the proof of this lemma.
\end{proof}

Next, we will establish the energy estimate for the
approximated equation \eqref{app-sym}.

\begin{lemm}\label{lemma-two}
For any smooth solution $(\u, \v)$ of equations \eqref{app-sym}-\eqref{ic},
it holds true for all $x\in [0, L^\eps]$
\beq\label{241}
\e^k_l(x)^2 \le \e^k_l(0)^2
              +C_{k_*, k^*, k, l,\delta_0}(1+\b^k_l(0)^{16}+\x^{16}),
\deq
where $k \ge 2$ and $l \ge 1$.
\end{lemm}

\begin{proof}
Taking $\px^{k}$ differential operator to the equation \eqref{app-q}, we have
\beqq
\px^{k+1} \{(\u)^2 \py \q\}-\px^k \py^3 \v+\px^k \py \v=0.
\deqq
Multiplying the above equation by $\px^k \py \q \y^{2l}$ and
integrating over $[0, x]\times[0,\infty)$, we have
\beq\label{242}
\bal
&\f12 \iy ((\u)^2 |\px^k \py \q|^2)(x, y) \y^{2l} dy
  -\int_0^x \iy \px^k \py^3 \v \cdot \px^k \py \q \y^{2l}d\tau dy \\
&+\int_0^x \iy \u |\px^k \py \q|^2 \y^{2l} d\tau dy\\
&=\f12 \iy ((\u)^2 |\px^k \py \q|^2)(0, y) \y^{2l} dy+II_1+II_2+II_3+II_4,
\dal
\deq
where the terms $II_i(i=1,2,3,4)$ are defined by
\beqq
\begin{aligned}
&II_1:=\int_0^x \iy \u \px \u |\px^k \py \q|^2 \y^{2l} d\tau dy,\\
&II_2:=-\sum_{j=1}^{k+1}C^{k+1}_j \int_0^x \iy \px^j\{(\u)^2\}
      \px^{k+1-j} \py \q \cdot \px^k \py \q \y^{2l} d\tau dy,\\
&II_3:=-\sum_{j=1}^k C^k_j \int_0^x \iy \px^j \u \px^{k-j} \py \q
      \cdot \px^k \py \q \y^{2l} d\tau dy,\\
&II_4:=-\int_0^x \iy \px^k(\py \u \q)\cdot \px^k \py \q \y^{2l} d \tau dy.
\end{aligned}
\deqq
Next, we deal with the dissipative term.
For any $\eps>0$, we use the boundary condition \eqref{bdyc} to obtain
\beq\label{BL-Condtion}
\px^k \py \q|_{y=0}
=\left.\px^k \left\{\f{\py \v \u-\v \py \u}{(\u)^2} \right\}\right|_{y=0}=0.
\deq
Then, we integrate by part to get
\beq\label{243}
\bal
&-\int_0^x \iy \px^k \py^3 \v \cdot \px^k \py \q \y^{2l}d\tau dy\\
=
&\int_0^x \px^k \py^2 \v \cdot \px^k \py \q \y^{2l}|_{y=0}d\tau
 +\int_0^x \iy \px^k \py^2 \v \cdot \py \{\px^k \py \q \y^{2l}\}d\tau dy\\
=
&\int_0^x \iy \px^k \{\py^2 \u \q+2\py \u \py \q+\u \py^2 \q\}
  \cdot \py \{\px^k \py \q \y^{2l}\}d\tau dy.
\dal
\deq
Then, substituting the equality \eqref{243} into \eqref{242}, we have
\beq\label{244}
\begin{aligned}
& \underset{0\le x \le L}{\sup}\iy (\u)^2 |\px^k \py \q|^2 \y^{2l}dy
  +2\ix \iy \u|\px^k \py \q|^2 \y^{2l} dx dy
  +2\ix \iy \u|\px^k \py^2 \q|\y^{2l} dx dy\\
&\le \iy ((\u)^2 |\px^k \py \q|^2)(0, y) \y^{2l} dy+2\sum_{i=1}^{i=8}|II_i|,
\end{aligned}
\deq
where the terms $II_i(i=5,6,7,8)$ are defined by
\beqq
\begin{aligned}
&II_5:=-\ix \iy \px^k (\py^2 \u \q )\cdot \px^k \py^2 \q \y^{2l} dx dy,\\
&II_6:=-2 \ix \iy \px^k (\py \u \py \q)\cdot \px^k \py^2 \q \y^{2l} dx dy,\\
&II_7:=-\sum_{j=1}^k C^k_j \ix \iy \px^j \u \px^{k-j} \py^2 \q
       \cdot \px^k \py^2 \q \y^{2l} dx dy,\\
&II_8:=-2l\ix \iy \px^k \{\py^2 \u \q+2\py \u \py \q+ \u \py^2 q\}
      \cdot\px^k \py \q \y^{2l-1} dx dy.
\end{aligned}
\deqq
Now, we claim that the terms $II_i(i=1,...,8)$ can be estimated as follows:
\begin{align}
\label{245}&|II_1|\le o_L(1)C_{k_*,l,k,\delta_0}(1+\b^k_l(0)^8+X^k_l(x)^8),\\
\label{246}&|II_2|\le o_L(1)C_{k_*,l,k,\delta_0}(1+\b^k_l(0)^{16}+X^k_l(x)^{16}),\\
\label{247}&|II_3|\le o_L(1)C_{k_*,k^*,k,\delta_0}(1+\b^k_l(0)^4+\x^4),\\
\label{248}&|II_4|\le o_L(1)C_{k_*,k,l,\delta_0}(1+\b^k_l(0)^4+\x^4),\\
\label{249}&|II_5|\le o_L(1)C_{k_*,k^*,l,k,\delta_0}(1+\b^k_l(0)^8+\x^{8}),\\
\label{2410}&|II_6|\le o_L(1)C_{k_*,k^*,k,l,\delta_0}(1+\b^k_l(0)^8+\x^{8}),\\
\label{2411}&|II_7|\le  o_L(1)C_{k_*,k^*,k,\delta_0}(1+\b^k_l(0)^8+\x^8),\\
\label{2412}&|II_8|\le o_L(1)C_{k_*, k^*, k, l,\delta_0}(1+\b^k_l(0)^4+\x^{4}).
\end{align}
Then, substituting the estimates \eqref{245}-\eqref{2412} into \eqref{244},
it is easy to obtain the estimate \eqref{241}.
\end{proof}

In order to complete the proof of Lemma \ref{lemma-two}, it remains to
show the claim estimates \eqref{245}-\eqref{2412} as follows.

\begin{proof}[\underline{\textbf{Proof of estimate \eqref{245}}}]
It is easy to deduce that
\beqq
|II_1|\le \|\f{\px \u}{\u}\|_{L^2_x L^\infty_{y}}
\|\u \px^k \py \q \y^l\|_{L^\infty_x L^2_y}^2,
\deqq
which, together with the estimates \eqref{claim1},
\eqref{232}, \eqref{233} and \eqref{234}, yields directly
\beqq
|II_1|
\le C_{k_*,\delta_0}\|\py \v\|_{L^2_x H^2_{y}}
    \|\u \px^k \py \q \y^l\|_{L^\infty_x L^2_y}^2
\le o_L(1)C_{k_*,l,k,\delta_0}(1+\b^k_l(0)^8+X^k_l(x)^8),
\deqq
which is the inequality \eqref{245}.
\end{proof}

\begin{proof}[\underline{\textbf{Proof of estimate \eqref{246}}}]
Let us write
\beqq
\begin{aligned}
II_2
=&-\sum_{j=1}^{[\f{k+1}{2}]}\sum_{j_1=0}^j C^{k+1}_j C^j_{j_1}
    \ix \iy \px^{j_1}\u \px^{j-j_1}\u \px^{k+1-j}\py \q \cdot \px^k \py \q \y^{2l} dx dy\\
 &-\sum_{[\f{k+1}{2}]}^{k+1}\sum_{j_1=0}^j C^{k+1}_j C^j_{j_1}
    \ix \iy \px^{j_1}\u \px^{j-j_1}\u \px^{k+1-j}\py \q \cdot \px^k \py \q \y^{2l} dx dy\\
:=&II_{21}+II_{22}.
\end{aligned}
\deqq
Applying the estimates \eqref{claim1}, \eqref{232}, \eqref{233} and \eqref{234}, we have
\beqq
\bal
|II_{21}| \le
&\|\f{\px^{\la[\f{k+1}{2}]\ra} \u}{\u}\|_{L^2_x L^\infty_{y}}
        \|\u \px^{\k}\py \q \y^l\|_{L^\infty_x L^2_{y}}^2
        +\|\f{\px^{\la [\f{k+1}{2}] \ra} \u}{\u}\|_{L^2_x L^\infty_{y}}^2
        \|\u \px^{\k}\py \q \y^l\|_{L^\infty_x L^2_{y}}^2\\
\le
&C_{k_*,\delta_0}(\|\px^{\la k-1 \ra} \py \v\|_{L^2_x H^2_y}
  +\|\px^{\la k-1 \ra} \py \v\|_{L^2_x H^2_y}^2)
        \|\u \px^{\k}\py \q \y^l\|_{L^\infty_x L^2_{y}}^2\\
\le & o_L(1)C_{k_*,l,k,\delta_0}(1+\b^k_l(0)^{16}+X^k_l(x)^{16}).
\dal
\deqq
Using Sobolev inequality, divergence-free condition,
estimates \eqref{claim1}, \eqref{claim21} and \eqref{231}, we have
\beqq
\bal
|II_{22}|
\le
&\|\px^{\langle k+1 \rangle} \u\|_{L^2_x L^\infty_{y}}
     \|\px^{\la [\f{k+1}{2}] \ra}\py \q \y^l \|_{L^2_x L^2_{y}}
     \|\u \px^k \py \q \y^l \|_{L^\infty_x L^2_{y}}\\
&+\|\f{\px^{\langle [\f{k+1}{2}]\rangle} \u}{\u}\|_{L^2_x L^\infty_{y}}
     \|\px^{\langle k+1 \rangle} \u\|_{L^2_x L^\infty_{y}}
     \|\px^{\la [\f{k+1}{2}] \ra}\py \q \y^{l} \|_{L^\infty_{x} L^2_{y}}
     \|\u \px^k \py \q \y^l \|_{L^\infty_x L^2_{y}}\\
\le
&   C_{k_*,\delta_0} \|\px^{\la k\ra} \py^2 \v \y \|_{L^2_x L^2_y}
     \|\px^{\la [\f{k+1}{2}] \ra}\py \q \y^l \|_{L^2_{x}L^2_{y}}
     \|\u \px^k \py \q \y^l \|_{L^\infty_x L^2_{y}}\\
&+\|\px^{\langle [\f{k+1}{2}]-1\rangle} \py^2 \v\|_{L^2_x H^1_{y}}
     \|\px^{\langle k \rangle}\py^2 \v \y\|_{L^2_x L^2_{y}}
     \|\u \px^k \py \q \y^l \|_{L^\infty_x L^2_{y}}\\
&\quad \times (\|\px^{\la [\f{k+1}{2}] \ra}\py \q \y^{l}|_{x=0}\|_{L^2_{y}}
      +o_L(1)\|\px^{\la [\f{k+1}{2}]+1\ra}\py \q \y^{l} \|_{L^2_{x} L^2_{y}})\\
\le
&o_L(1)C_{k_*, k^*, k, \delta_0}(1+\b^k_l(0)^{16}+X^k_l(x)^{16}).
\dal
\deqq
Therefore, we can obtain the following estimate
\beqq
|II_2|\le o_L(1)C_{k_*, k^*, k, \delta_0}(1+\b^k_l(0)^{16}+X^k_l(x)^{16}),
\deqq
which is inequality \eqref{246}.
\end{proof}

\begin{proof}[\underline{\textbf{Proof of estimate \eqref{247}}}]
Apply the divergence-free condition, we have
\beqq
II_3
=\sum_{j=1}^k C^k_j \int_0^x \iy \px^{j-1} \py \v \px^{k-j} \py \q
      \cdot \px^k \py \q \y^{2l} d\tau dy.
\deqq
Due to $k\ge 2$, using the Sobolev inequality, estimates
\eqref{claim21}, \eqref{claim22}, \eqref{231} and \eqref{232}, we have
\beqq
\bal
|II_3|
=&|\sum_{j=1}^k C^k_j \int_0^x \iy \px^{j-1} \py \v \px^{k-j} \py \q
      \cdot \px^k \py \q \y^{2l} d\tau dy|\\
\le
&  C_k \|\px^{\la [\f{k-1}{2}] \ra} \py \v\|_{L^\infty_x L^\infty_y}
    \|\px^{\la k-1\ra}\py \q \y^l \|_{L^2_x L^2_y}
    \|\px^k \py \q \y^l \|_{L^2_x L^2_y}\\
&    +C_k\|\px^{\la k-1 \ra} \py \v\|_{L^2_x L^\infty_y}
    \|\px^{\la [\f{k-1}{2}]\ra}\py \q \y^l \|_{L^\infty_x L^2_y}
    \|\px^k \py \q \y^l \|_{L^2_x L^2_y}\\
\le
&o_L(1)C_{k,k_*}(\|\px^{\la [\f{k-1}{2}] \ra} \py^2 \v \y |_{x=0}\|_{L^2_y}
         +\|\px^{\la [\f{k-1}{2}]+1 \ra} \py^2 \v \y \|_{L^2_x L^2_y})\e^k_l(x)^2\\
&+o_L(x)C_{k,k_*}\|\px^{\la k-1 \ra} \py^2 \v \y\|_{L^2_x L^2_y}\e^k_l(x)\\
&\quad \times (\|\px^{\la [\f{k-1}{2}]\ra}\py \q \y^l |_{x=0}\|_{L^2_y}
        +o_L(1)C_{k_*,\d_0}\e^{\la [\f{k-1}{2}]+1\ra}_l(x))\\
\le
&o_L(1)C_{k_*,k^*,k,\delta_0}(1+\b^k_l(0)^4+\x^4).
\dal
\deqq
This yields inequality \eqref{247} directly.
\end{proof}

\begin{proof}[\underline{\textbf{Proof of estimate \eqref{248}}}]
Using the divergence-free condition, we can write
\beqq
\bal
II_4=
&-\int_0^x \iy \py \u \px^k \q \cdot \px^k \py \q \y^{2l} d \tau dy\\
&+\sum_{j=1}^k C^k_j\int_0^x \iy \px^{j-1}\py^2 \v \px^{k-j}\q \cdot \px^k \py \q \y^{2l} d \tau dy\\
:=&II_{41}+II_{42}.
\dal
\deqq
Using H\"{o}lder inequality, divergence-free condition,
estimates \eqref{claim21} and \eqref{claim31}, it holds true
\beq\label{ii41}
\bal
|II_{41}|
&\le \|\py \u \y^l \|_{L^\infty_x L^2_y}
    \|\px^{k} \q \|_{L^2_x L^\infty_y}
    \|\px^k \py \q \y^l \|_{L^2_x L^2_y}\\
&\le o_L(1)C_{k_*,l}(\|\py \u_0 \y^l \|_{L^2_y}
         +\|\py^2 \v \y^l \|_{L^2_x L^2_y})\e^k_l(x)^2.
\dal
\deq
Due to $k\ge 2$, using the H\"{o}lder inequality, estimates
\eqref{claim21}, \eqref{claim31}, \eqref{claim32}, \eqref{231} and \eqref{232}, we have
\beq\label{ii422}
\bal
|II_{42}| \le
& C_k\|\px^{\la [\f{k-1}{2}]\ra}\py^2 \v \y^l \|_{L^\infty_x L^2_y}
  \|\px^{\la k-1 \ra}\q \|_{L^2_x L^\infty_y}
  \|\px^k \py \q \y^l\|_{L^2_x L^2_y}\\
&+C_k \|\px^{\la k-1 \ra}\py^2 \v \y^l\|_{L^2_x L^2_y}
  \|\px^{\la [\f{k-1}{2}]\ra}\q \|_{L^\infty_x L^\infty_y}
  \|\px^k \py \q \y^l\|_{L^2_x L^2_y}\\
\le
&o_L(1)C_{k_*,k,l,\delta_0}
        (\|\px^{\la [\f{k-1}{2}]\ra}\py^2 \v \y^l|_{x=0}\|_{L^2_y}
        +\|\px^{\la [\f{k-1}{2}]+1\ra}\py^2 \v \y^l \|_{L^2_x L^2_y})\e^k_l(x)^2\\
&+o_L(1)C_{k_*,k,l,\delta_0}\|\px^{\la k-1 \ra}\py^2 \v \y^l\|_{L^2_x L^2_y}
  (\|\px^{\la [\f{k-1}{2}]\ra}\py \q \y|_{x=0}\|_{L^2_y}
   +o_L(1)\e^{\la [\f{k-1}{2}]+1\ra}_l(x))\e^k_l(x)\\
\le
&o_L(1)C_{k_*,k,l,\delta_0}(1+\b^k_l(0)^4+\x^4).
\dal
\deq
Then, the combination of estimates \eqref{ii41}
and \eqref{ii422} yields the inequality \eqref{248}.
\end{proof}

\begin{proof}[\underline{\textbf{Proof of estimate \eqref{249}}}]
Using the divergence-free condition, it holds true
\beqq
\bal
II_5=&-\ix \iy \pyy \u \px^k \q \cdot \px^k \py^2 \q \y^{2l} dy dx\\
    &+\sum_{j=1}^{k}C^k_j \ix \iy \px^{j-1}\py^3 \v \px^{k-j} \q
     \cdot \px^k \py^2 \q \y^{2l} dy dx\\
   :=&II_{51}+II_{52}.
\dal
\deqq
First of all, let us deal with the term $II_{51}$.
We can write
\beqq
\bal
II_{51}
=&-\ix \iy \pyy \u \px^k \q \cdot \px^k \py^2 \q \chi \y^{2l} dy dx
  -\ix \iy \pyy \u \px^k \q \cdot \px^k \py^2 \q (1-\chi)\y^{2l} dy dx\\
:=&II_{511}+II_{512}.
\dal
\deqq
Using the H\"{o}lder inequality, it holds true
\beq\label{ii51}
|II_{511}|
\le C \|\f{y}{\sqrt{\u}}\py^2 \u \chi\|_{L^\infty_x L^\infty_y}
       \|\f{1}{y}\px^k \q \y^{l}\|_{L^2_x L^2_y}
       \|\sqrt{\u} \px^k \py^2 \q \|_{L^2_x L^2_y}.
\deq
Using the Hardy inequality and estimate \eqref{claim21}, it is easy to check that
\beq\label{ii511}
\|\f1y \px^k \q \|_{L^2_x L^2_y}
\le \|\px^k \py \q \|_{L^2_x L^2_y}
\le o_L(1)C_{k_*}(\|\u \px^k \py \q \y^l \|_{L^\infty_x L^2_y}
           +\|\sqrt{\u}\px^k \py^2 \q \y^l\|_{L^2_x L^2_y}).
\deq
On the other hand, using divergence-free condition,
relations \eqref{equi-relation} and \eqref{lower-bound}, it holds true
\beq\label{ii512}
\bal
&\|\f{y}{\sqrt{\u}}\py^2 \u \chi\|_{L^\infty_x L^\infty_y}
\le C_{k_*}\|\sqrt{y}\py^2 \u \chi\|_{L^\infty_x L^\infty_y}\\
&\le C_{k_*}(\|\py^2 \u_0 \chi\|_{L^2_y}+\|\py^3 \u_0 \chi\|_{L^2_y})
      +o_L(1)C_{k_*}(\|\py^3 \v \chi\|_{L^2_x L^2_y}+\|\py^4 \v \chi\|_{L^2_x L^2_y}).
\dal
\deq
Then, we can get that
\beq\label{ii513}
|II_{511}|
\le o_L(1)C_{k_*}(1+\b^k_l(0)^4+\x^4).
\deq
Using H\"{o}lder inequality and estimate \eqref{claim31}, it holds true
\beqq
\bal
|II_{512}|\le
&C_{k^*,\d_0}\|\py^2 \u \y^l\|_{L^\infty_x L^2_y}
  \|\px^k \q \|_{L^2_x L^\infty_y}
  \|\sqrt{\u}\px^k \py^2 \q \y^l\|_{L^2_x L^2_y}\\
\le
&o_L(1)C_{k_*,l,\d_0}(\|\py^2 \u \y^l|_{x=0}\|_{L^2_y}
                        +o_L(1)\|\py^3 \v \y^l\|_{L^2_x L^2_y})\e^k_l(x)
 \|\sqrt{\u}\px^k \py^2 \q \y^l\|_{L^2_x L^2_y}\\
\le
&o_L(1)C_{k_*,l,\d_0}(1+\b^k_l(0)^4+\x^4),
\dal
\deqq
which, together with estimate \eqref{ii513}, yields directly
\beq\label{ii51-1}
|II_{51}|
\le o_L(1)C_{k_*,l,\d_0}(1+\b^k_l(0)^4+\x^4).
\deq
Integrating by part and using boundary condition \eqref{BL-Condtion}, we have
\beqq
\bal
II_{52}=
&-\sum_{j=1}^{k}C^k_j \ix \iy \px^{j-1}\py^4 \v \px^{k-j} \q
     \cdot \px^k \py \q \y^{2l} dy dx\\
&-\sum_{j=1}^{k}C^k_j \ix \iy \px^{j-1}\py^3 \v \px^{k-j} \py \q
     \cdot \px^k \py \q \y^{2l} dy dx\\
&-2l\sum_{j=1}^{k}C^k_j \ix \iy \px^{j-1}\py^3 \v \px^{k-j} \q
     \cdot \px^k \py \q \y^{2l-1} dy dx\\
:=
&II_{521}+II_{522}+II_{523}.
\dal
\deqq
Using the H\"{o}lder inequality, estimates \eqref{claim31}, \eqref{claim32}
and \eqref{235}, we have
\beqq
\bal
|II_{521}|
\le
& C_k \|\px^{\la k-1 \ra} \py^4 \v \y^l \|_{L^2_x L^2_y}
  \|\px^{\la k-1 \ra}\q \|_{L^\infty_x L^\infty_y}
  \|\px^k \py \q \y^l\|_{L^2_x L^2_y}\\
\le
&o_L(1)C_{k_*,k, l,\delta_0}\|\px^{\la k-1 \ra} \py^4 \v \y^l \|_{L^2_x L^2_y}
       (\|\px^{\la k-1\ra}\py q \y |_{x=0}\|_{L^2_y}
         +o_L(1)\e^k_l(x))\e^k_l(x)\\
\le
&o_L(1)C_{k_*,k^*,l,k,\delta_0}(1+\b^k_l(0)^8+\x^8).
\dal
\deqq
Similarly, using \eqref{claim21}, \eqref{claim31}, \eqref{claim32}
and \eqref{233}, we also have
\beqq
\bal
|II_{522}|
\le
& C_k \|\px^{\la k-1 \ra} \py^3 \v \|_{L^2_x L^\infty_y}
  \|\px^{\la k-1 \ra}\py \q \y^{l}\|_{L^\infty_x L^2_y}
  \|\px^k \py \q \y^l\|_{L^2_x L^2_y}\\
\le
& C_k \|\px^{\la k-1 \ra} \py^4 \v \y\|_{L^2_x L^2_y}
  (\|\px^{\la k-1 \ra}\py \q \y^{l}|_{x=0}\|_{L^2_y}
   +o_L(1)\|\px^{\la k\ra}\py \q \y^{l}\|_{L^2_x L^2_y})
  \|\px^k \py \q \y^l\|_{L^2_x L^2_y}\\
\le
&o_L(1)C_{k_*,k,l,\delta_0}(1+\b^k_l(0)^{8}+\x^{8}).
\dal
\deqq
Finally, we use the estimates \eqref{claim21} and \eqref{claim32} to obtain
\beqq
\bal
|II_{523}|
\le
& C_k\|\px^{\la [\f{k-1}{2}]\ra} \py^3 \v \y^{l-1}\|_{L^\infty_x L^2_y}
  \|\px^{\la k-1 \ra}\q \|_{L^2_x L^\infty_y}
  \|\px^k \py \q \y^l\|_{L^2_x L^2_y}\\
&+C_k\|\px^{\la k-1 \ra} \py^3 \v \y^{l-1}\|_{L^2_x L^2_y}
  \|\px^{\la [\f{k-1}{2}]\ra}\q \|_{L^\infty_x L^\infty_y}
  \|\px^k \py \q \y^l\|_{L^2_x L^2_y}\\
\le
&o_L(1)C_{k_*,k,l,\delta_0}(\|\px^{\la [\f{k-1}{2}]\ra} \py^3 \v \y^{l-1}|_{x=0}\|_{L^2_y}
        +\|\px^{\la [\f{k-1}{2}]+1\ra} \py^3 \v \y^{l-1}\|_{L^2_x L^2_y})\e^k_l(x)^2\\
&+o_L(1)C_k\|\px^{\la k-1 \ra} \py^3 \v \y^{l-1}\|_{L^2_x L^2_y}
  (\|\px^{\la [\f{k-1}{2}]\ra}\py \q \y |_{x=0}\|_{L^2_y}
    +o_L(1)C_{k_*,l,\d_0}\e^{\la [\f{k-1}{2}]+1\ra}_l(x))\e^k_l(x)\\
\le
&o_L(1)C_{k_*,k,l,\delta_0}(1+\b^k_l(0)^4+\x^4).
\dal
\deqq
Based on the above estimates from terms $II_{521}$ to $II_{523}$, we can obtain the estimate
\beqq
|II_{52}|\le o_L(1) C_{k_*,k^*,l,k,\delta_0} (1+\b^k_l(0)^{8}+\x^{8}),
\deqq
which, together with estimate \eqref{ii51-1}, yields the following estimate directly
\beqq
|II_5|\le o_L(1)C_{k_*,k^*,l,k,\delta_0}(1+\b^k_l(0)^{8}+\x^{8}).
\deqq
Therefore, we complete the proof of estimate \eqref{249}.
\end{proof}

\begin{proof}[\underline{\textbf{Proof of estimate \eqref{2410}}}]
Obviously, we have
\beqq
\bal
II_6
=&-2 \ix \iy \py \u \px^k \py \q\cdot \px^k \py^2 \q \y^{2l} dx dy,\\
 &+2 \sum_{j=1}^k C^k_j \ix \iy \px^{j-1} \py^2 \v \px^{k-j} \py \q
     \cdot \px^k \py^2 \q \y^{2l} dx dy\\
:=&II_{61}+II_{62}.
\dal
\deqq
Let us write
\beqq
\bal
II_{61}
=&-2 \ix \iy \py \u \px^k \py \q\cdot \px^k \py^2 \q \y^{2l} \chi dx dy\\
  &-2 \ix \iy \py \u \px^k \py \q\cdot \px^k \py^2 \q \y^{2l} (1-\chi)dx dy\\
=&:II_{611}+II_{612}
\dal
\deqq
Integrating by part and using boundary condition \eqref{BL-Condtion},  we have
\beqq
\bal
II_{611}
&=\ix \iy |\px^k \py \q|^2 \py \{\py \u  \y^{2l}\chi \} dx dy\\
&=\ix \iy |\px^k \py \q|^2 \{\py^2 \u  \y^{2l}\chi
                             +2l \py \u  \y^{2l-1}\chi
                             +\f1{\delta}\py \u  \y^{2l}\chi' \} dx dy.
\dal
\deqq
Then, choosing $\delta=\delta_0$ and using the estimate \eqref{claim21}, we get
\beqq
\bal
|II_{611}|
&\le C_{l,\d_0}(\|\py^3 \u \chi\|_{L^\infty_x L^2_y}+\|\py^2 \u\|_{L^\infty_x L^2_y})
     \|\px^k \py \q \y^l\|_{L^2_x L^2_y}^2\\
&\le o_L(1)C_{k_*,l,\d_0}(\|\py^3 \u \chi|_{x=0}\|_{L^2_y}+\|\py^2 \u |_{x=0}\|_{L^2_y}
     +\|\py^3 \v\|_{L^2_x L^2_y}+\|\py^4 \v \chi \|_{L^2_x L^2_y})\e^k_l(x)^2.
\dal
\deqq
Using the H\"{o}lder inequality, lower bound assumption \eqref{lower-bound}
and divergence-free condition, we have
\beqq
\bal
|II_{612}|
&\le o_L(1)C_{k_*,\d_0}\|\py \u\|_{L^\infty_x L^\infty_y}
     \|\u \px^k \py \q \y\|_{L^\infty_x L^2_y}
     \|\sqrt{\u}\px^k \py^2 \q \y\|_{L^2_x L^2_y}\\
&\le o_L(1)C_{k_*,\d_0}(\|\py^2 \u \y|_{x=0}\|_{L^2_y}
                        +o_L(1)\|\py^3 \v \y\|_{L^2_x L^2_y})\e^k_l(x)^2.
\dal
\deqq
Then, we have the estimate
\beq\label{ii61}
|II_{61}|
\le o_L(1)C_{k_*,k,l,\delta_0}(1+\b^k_l(0)^4+\x^4).
\deq
Finally, let us deal with term $II_{62}$.
Due to $k\ge 2$, integrating by part and using boundary
condition \eqref{BL-Condtion}, we have
\beqq
\bal
II_{62}
=
&-2\sum_{j=1}^k C^k_j \ix \iy \px^{j-1} \py^3 \v \px^{k-j} \py \q
     \cdot \px^k \py \q \y^{2l} dx dy\\
&-2\sum_{j=1}^k C^k_j \ix \iy \px^{j-1} \py^2 \v \px^{k-j} \py^2 \q
     \cdot \px^k \py  \q \y^{2l} dx dy\\
&-4l\sum_{j=1}^k C^k_j \ix \iy \px^{j-1} \py^2 \v \px^{k-j} \py \q
     \cdot \px^k \py \q \y^{2l-1} dx dy\\
:=&II_{621}+II_{622}+II_{623}.
\dal
\deqq
Using the Sobolev inequality, estimates
\eqref{claim21}, \eqref{claim4}, \eqref{233} and \eqref{235},
it is easy to deduce
\beqq
\bal
|II_{621}|\le
&C_k\|\px^{\la [\f{k-1}{2}]\ra} \py^3 \v\y\|_{L^\infty_x L^2_y}
  \|\px^{\la k-1 \ra} \py \q \y^{l-1} \|_{L^2_x L^\infty_y}
  \|\px^k \py \q \y^l\|_{L^2_x L^2_y}\\
& +C_k\|\px^{\la k-1  \ra} \py^3 \v\|_{L^2_x L^\infty_y}
  \|\px^{\la [\f{k-1}{2}] \ra} \py \q \y^l \|_{L^\infty_x L^2_y}
  \|\px^k \py \q \y^l\|_{L^2_x L^2_y}\\
\le
&o_L(1)C_{k_*,k,l}(\|\px^{\la [\f{k-1}{2}]\ra} \py^3 \v|_{x=0}\|_{L^2_y}
        +\|\px^{\la [\f{k-1}{2}]+1\ra} \py^3 \v\|_{L^2_x L^2_y})
        \|\px^{\la k-1 \ra} \py^2 \q \y^{l} \|_{L^2_x L^2_y} \e^k_l(x)\\
& +o_L(1)C_{k_*,k,l}\|\px^{\la k-1  \ra} \py^4 \v \y\|_{L^2_x L^2_y}
  \|\px^k \py \q \y^l\|_{L^2_x L^2_y}\\
&\times (\|\px^{\la [\f{k-1}{2}] \ra} \py \q \y^l|_{x=0}\|_{L^2_y}
   +\|\px^{\la [\f{k-1}{2}]+1 \ra} \py \q \y^l \|_{L^2_x L^2_y})\\
\le
&o_L(1)C_{k_*,k^*,k,l,\delta_0}(1+\b^k_l(0)^8+\x^{8}).
\dal
\deqq
Similarly, for the case of $k\ge 2$, it is easy to check that
\beqq
\bal
|II_{623}|\le
&C_k\|\px^{\la [\f{k-1}{2}]\ra} \py^2 \v\|_{L^\infty_x L^\infty_y}
  \|\px^{\la k-1 \ra} \py \q \y^l \|_{L^2_x L^2_y}
  \|\px^k \py \q \y^l\|_{L^2_x L^2_y}\\
& +C_k\|\px^{\la k-1  \ra} \py^2 \v\|_{L^2_x L^\infty_y}
  \|\px^{\la [\f{k-1}{2}] \ra} \py \q \y^l \|_{L^\infty_x L^2_y}
  \|\px^k \py \q \y^l\|_{L^2_x L^2_y}\\
\le
&o_L(1)C_{k_*,k}(\|\px^{\la [\f{k-1}{2}]\ra} \py^3 \v\y|_{x=0}\|_{L^2_y}
        +\|\px^{\la [\f{k-1}{2}]+1  \ra} \py^3 \v \y\|_{L^2_x L^2_y})\e^k_l(x)^2\\
&+o_L(1)C_{k_*,k}\|\px^{\la k-1  \ra} \py^3 \v \y\|_{L^2_x L^2_y}
   (\|\px^{\la [\f{k-1}{2}] \ra} \py \q \y^l|_{x=0}\|_{L^2_y}
          +\|\px^{\la [\f{k-1}{2}]+1 \ra} \py \q \y^l \|_{L^2_x L^2_y})\e^k_l(x)\\
\le
&o_L(1)C_{k_*,k^*,k,l,\delta_0}(1+\b^k_l(0)^4+\x^{4}).
\dal
\deqq
and
\beqq
\bal
|II_{622}|\le
&C_k\|\px^{\la [\f{k-1}{2}]\ra} \py^2 \v\|_{L^\infty_x L^\infty_y}
  \|\px^{\la k-1 \ra} \py^2 \q \y^l \|_{L^2_x L^2_y}
  \|\px^k \py \q \y^l\|_{L^2_x L^2_y}\\
& +C_k\|\px^{\la k-1  \ra} \py^2 \v\|_{L^2_x L^\infty_y}
  \|\px^{\la [\f{k-1}{2}] \ra} \py^2 \q \y^l \|_{L^\infty_x L^2_y}
  \|\px^k \py \q \y^l\|_{L^2_x L^2_y}\\
\le
&o_L(1)C_{k_*,k}(\|\px^{\la [\f{k-1}{2}]\ra} \py^3 \v \y|_{x=0}\|_{L^2_y}
        +\|\px^{\la k-1  \ra} \py^3 \v \y\|_{L^2_x L^2_y})
        \|\px^{\la k-1 \ra} \py^2 \q \y^l \|_{L^2_x L^2_y}\e^k_l(x)\\
&+o_L(1)C_{k_*,k}\|\px^{\la k-1  \ra} \py^3 \v \y\|_{L^2_x L^2_y}
  (\|\px^{\la [\f{k-1}{2}] \ra} \py^2 \q \y^l|_{x=0}\|_{L^2_y}
   +\|\px^{\la [\f{k-1}{2}]+1\ra} \py^2 \q \y^l \|_{L^2_x L^2_y})
  \e^k_l(x)\\
\le
&o_L(1)C_{k_*,k^*,k,l,\delta_0}(1+\b^k_l(0)^4+\x^{4}),
\dal
\deqq
where we have used the following inequality in the last inequality
\beqq
\|\px^{\la [\f{k-1}{2}] \ra}\py^2 \q \y^{l}|_{x=0}\|_{L^2_y}
\le C_{k_*,k^*}(\|\px^{\la [\f{k-1}{2}] \ra}\py^2 \q \y^{l}(1-\chi)|_{x=0}\|_{L^2_y}
  +\|\px^{\la [\f{k-1}{2}] \ra}\py^3 \v|_{x=0}\|_{L^2_y}).
\deqq
Thus, we can obtain the following estimate for all $k \ge 2$
\beq\label{ii622}
|II_{62}|\le o_L(1)C_{k_*,k^*,k,l,\delta_0}(1+\b^k_l(0)^8+\x^{8}).
\deq
Thus, the combination of estimates \eqref{ii61} and \eqref{ii622} yields directly
\beqq
|II_{6}|\le o_L(1)C_{k_*,k^*,k,l,\delta_0}(1+\b^k_l(0)^8+\x^{8}),
\deqq
which is the inequality \eqref{2410}.
\end{proof}

\begin{proof}[\underline{\textbf{Proof of estimate \eqref{2411}}}]
Let us write
\beqq
\bal
II_7
=
&-\sum_{j=1}^k C^k_j \ix \iy \px^j \u \px^{k-j} \py^2 \q
     \cdot \px^k \py^2 \q \y^{2l} \chi(\f{y}{\delta_0}) \y^{2l}dy dx \\
&-\sum_{j=1}^k C^k_j \ix \iy \px^j \u \px^{k-j} \py^2 \q
       \cdot \px^k \py^2 \q \y^{2l} (1-\chi(\f{y}{\delta_0})) \y^{2l}dy dx \\
:=&II_{71}+II_{72}.
\dal
\deqq
First of all, let deal with the term $II_{71}$.
Using H\"{o}lder inequality,
estimates \eqref{claim1} and \eqref{claim4}, we can obtain
\beq\label{ii714}
\bal
|II_{71}|\le
&|\sum_{j=1}^k C^k_j \ix \iy \px^j \u \px^{k-j} \py^2 \q
       \cdot \px^k \py^2 \q \y^{2l} \chi dy dx |\\
\le
& C_k(\d_0+\eps)^{\f12}
  \|\f{\px^{\la [\f{k-1}{2}] \ra}\px \u}{\u} \chi\|_{L^\infty_x L^\infty_y}
  \|\px^{\la k-1 \ra}\py^2 \q \chi \|_{L^2_x L^2_y}
  \|\sqrt{\u} \px^k \py^2 \q \chi \|_{L^2_x L^2_y}\\
& +C_k (\d_0+\eps)^{\f12}\|\f{\px^{\la k-1 \ra}\px \u}{\u} \chi\|_{L^2_x L^\infty_y}
  \|\px^{\la [\f{k-1}{2}] \ra}\py^2 \q \chi \|_{L^\infty_x L^2_y}
  \|\sqrt{\u} \px^k \py^2 \q \chi \|_{L^2_x L^2_y}\\
\le
&o_L(1)C_{k,k_*, k^*,\d_0}(\|\px^{\la [\f{k-1}{2}] \ra} \py \v|_{x=0}\|_{H^2_y}
             +\|\px^{\la [\f{k-1}{2}]+1\ra} \py \v\|_{L^2_x H^2_y})
             \|\px^{\la k-1 \ra}\py^3 \v\|_{L^2_x L^2_y}\e^k_l(x)\\
&+o_L(1)C_{k,k_*,\d_0} \|\px^{\la k-1 \ra}\py^3 \v\|_{L^2_x L^2_y}
           (\|\px^{\la [\f{k-1}{2}] \ra} \py^3 \v|_{x=0}\|_{L^2_y}
             +\|\px^{\la [\f{k-1}{2}]+1\ra} \py^3 \v\|_{L^2_x L^2_y})\e^k_l(x)\\
\le
&o_L(1)C_{k_*,k^*,k,\delta_0}(1+\b^k_l(0)^4+\x^4).
\dal
\deq
Using the H\"{o}lder inequality, we have
\beqq
\bal
|II_{72}|=
&|\sum_{j=1}^k C^k_j \ix \iy \px^j \u \px^{k-j} \py^2 \q
       \cdot \px^k \py^2 \q \y^{2l} (1-\chi) dy dx |\\
\le
&C_k\|{\px^{\la k-1 \ra}\px \u}(1-\chi)\|_{L^2_x L^\infty_y}
  \|\px^{\la k-1 \ra}\py^2 \q (1-\chi) \y^l\|_{L^\infty_x L^2_y}
  \|\px^k \py^2 \q (1-\chi) \y^l\|_{L^2_x L^2_y}.
\dal
\deqq
Away from the boundary $y=0$, we have
\beqq
\bal
&\|\px^{\la k-1 \ra}\py^2 \q (1-\chi) \y^l\|_{L^\infty_x L^2_y}\\
\le &\|\px^{\la k-1 \ra}\py^2 \q (1-\chi) \y^l|_{x=0}\|_{L^2_y}
     +o_L(1)\|\px^{\la k \ra}\py^2 \q (1-\chi) \y^l\|_{L^2_x L^2_y}\\
\le &\|\px^{\la k-1 \ra}\py^2 \q (1-\chi)\y^l|_{x=0}\|_{L^2_y}
     +o_L(1)C_{k_*}\|\sqrt{\u}\px^{\la k \ra}\py^2 \q \y^l\|_{L^2_x L^2_y}.
\dal
\deqq
Then, we have
\beqq
\bal
|II_{72}|
\le &\|{\px^{\la k-1 \ra}\py \v}\|_{L^2_x H^1_y}
     \|\sqrt{\u}\px^k \py^2 \q \|_{L^2_x L^2_y}\\
    &\times (\|\px^{\la k-1 \ra}\py^2 \q (1-\chi)\y^l|_{x=0}\|_{L^2_y}
     +o_L(1)C_{k_*}\|\sqrt{\u}\px^{\la k \ra}\py^2 \q \y^l\|_{L^2_x L^2_y})\\
\le
   &o_L(1)C_{k_*,k^*,k,\delta_0}(1+\b^k_l(0)^8+\x^8),
\dal
\deqq
which, together with the estimate \eqref{ii714}, yields directly
\beqq
|II_7|\le o_L(1)C_{k_*,k^*,k,\delta_0}(1+\b^k_l(0)^8+\x^8),
\deqq
which is the inequality \eqref{2411}.
\end{proof}

\begin{proof}[\underline{\textbf{Proof of estimate \eqref{2412}}}]
Recall
\beqq
\bal
II_8
:=&-2l\ix \iy \px^k (\py^2 \u \q )\cdot\px^k \py \q \y^{2l-1} dy dx
   -4l\ix \iy \px^k (\py \u \py \q) \cdot\px^k \py \q \y^{2l-1} dy dx\\
&\quad  -2l\ix \iy \px^k (\u \py^2 \q) \cdot\px^k \py \q \y^{2l-1} dy dx\\
:=&II_{81}+II_{82}+II_{83}.
\dal
\deqq
Using the divergence-free condition, estimates \eqref{claim21},
\eqref{claim22}, \eqref{claim31} and \eqref{claim32}, we have
\beq\label{ii81}
\bal
|II_{81}|
\le
& C_{k,l}\|\px^k \q \|_{L^2_x L^\infty_y}
    \|\py^2 \u \y^{l-1}\|_{L^\infty_x L^2_y}\|\px^k \py \q \y^l\|_{L^2_x L^2_y}\\
    &+C_{k,l}\|\px^{\la k-1 \ra}\q \|_{L^\infty_x L^\infty_y}
     \|\px^{\la k-1 \ra}\py^3 \v \y^{l-1}\|_{L^2_x L^2_y}
     \|\px^k \py \q \y^l\|_{L^2_x L^2_y}\\
\le
&o_L(1)C_{k_*,k,l}(\|\py^2 \u \y^{l-1}|_{x=0}\|_{L^2_y}
         +o_L(1)\|\py^3 \v \y^{l-1} \|_{L^2_x L^2_y})\e^k_l(x)^2\\
&+o_L(1)C_{k_*,k,l,\delta_0}(\|\px^{\la k-1 \ra}\py \q \y|_{x=0}\|_{L^2_y}
                             +o_L(1)\e^k_l(x))
     \|\px^{\la k-1 \ra}\py^3 \v \y^{l-1}\|_{L^2_x L^2_y}\e^k_l(x)\\
\le
&o_L(1)C_{k_*,k,l,\delta_0}(1+\b^k_l(0)^4+\x^4).
\dal
\deq
Similarly, it is easy to get that
\beq\label{ii82}
\bal
|II_{82}|
\le
&C_{k,l}\|\py \u\|_{L^\infty_x L^\infty_y}
    \|\px^{k}\py \q \y^{l-1}\|_{L^2_x L^2_y}
    \|\px^{k}\py \q \y^{l}\|_{L^2_x L^2_y}\\
&+C_{k,l}\|\px^{\la k-1 \ra}\py^2 \v\|_{L^2_x L^\infty_y}
    \|\px^{\la k-1 \ra}\py \q \y^{l-1}\|_{L^\infty_x L^2_y}
    \|\px^{k}\py \q \y^{l}\|_{L^2_x L^2_y}\\
\le
&o_L(1)C_{k_*,k,l}(\|\py \u|_{x=0}\|_{H^1_y}
 +o_{L}(1)\|\py^2 \v\|_{L^2_x H^1_y})\e^k_l(x)^2\\
&+o_L(1)C_{k_*,k,l}\|\px^{\la k-1\ra}\py^2 \v\|_{L^2_x H^1_y}
  (\|\px^{\la k-1 \ra}\py \q \y^{l-1}|_{x=0}\|_{L^2_y}+\e^k_l(x))\e^k_l(x)\\
\le
&o_L(1)C_{k_*,k,l,\delta_0}(1+\b^k_l(0)^4+\x^4).
\dal
\deq
Finally, let us deal with the term $II_{83}$.
Indeed, Let us write
\beqq
\bal
II_{83}
=&-2l\ix \iy \u \px^k \py^2 \q \cdot\px^k \py \q \y^{2l-1} dy dx\\
 &-2l\sum_{j=1}^k C^k_j \ix \iy \px^j \u \px^{k-j}\py^2 \q
   \cdot\px^k \py \q \y^{2l-1} dy dx\\
:=&II_{831}+II_{832}.
\dal
\deqq
Using the H\"{o}lder inequality, we get
\beqq
|II_{831}|\le  C_l\|\sqrt{\u} \px^k \py \q \y^{l-1}\|_{L^2_x L^2_y}
                  \|\sqrt{\u} \px^k \py^2 \q \y^l\|_{L^2_x L^2_y}.
\deqq
Using the equivalent relation \eqref{equi-relation}
and lower bound \eqref{lower-bound}, it holds for all $\delta\le \delta_0$
\beqq
\bal
&\|\sqrt{\u} \px^k \py \q \y^{l-1}\|_{L^2_x L^2_y}\\
\le
&\|\sqrt{\u} \px^k \py \q \y^{l-1}(1-\chi)\|_{L^2_x L^2_y}
 +\|\sqrt{u} \px^k \py \q \y^{l-1} \chi\|_{L^2_x L^2_y}\\
\le
&\f{L}{\d}\|\u \px^k \py \q \y^{l-1}(1-\chi)\|_{L^\infty_x L^2_y}
 +\d \|\px^k \py \q \y^{l-1} \chi\|_{L^2_x L^2_y}\\
\le
&o_L(1)C_{k_*}(\|\u \px^k \py \q \y^l \|_{L^\infty_x L^2_y}
           +\|\sqrt{\u}\px^k \py^2 \q\|_{L^2_x L^2_y}),
\dal
\deqq
where we have chosen $\d=L^{\f12}$ and used estimate
\eqref{claim31} in the last inequality. Thus, it holds
\beqq
|II_{831}| \le o_L(1)C_{k_*,l}\x^2.
\deqq
Using Sobolev inequality, estimates \eqref{claim21},
\eqref{claim4}, \eqref{231}, \eqref{232} and \eqref{233}, we get
\beqq
\bal
|II_{832}|
\le &C_{k,l}\|\px^{\la[\f{k-1}{2}]\ra} \py \v\|_{L^\infty_x L^\infty_y}
    \|\px^{\la k-1 \ra}\py^2 \q \y^{l-1}\|_{L^2_x L^2_y}
    \|\px^k \py \q \y^l\|_{L^2_x L^2_y}\\
    &+C_{k,l}\|\px^{\la k-1 \ra} \py \v\|_{L^2_x L^\infty_y}
    \|\px^{\la [\f{k-1}{2}] \ra}\py^2 \q \y^{l-1}\|_{L^\infty_x L^2_y}
    \|\px^k \py q \y^l\|_{L^2_x L^2_y}\\
\le &o_L(1)C_{k_*, k, l}(\|\px^{\la [\f{k-1}{2}] \ra} \py^2 \v \y|_{x=0}\|_{L^2_y}
             +\|\px^{\la[\f{k-1}{2}]+1\ra} \py^2 \v\y\|_{L^2_x L^2_y})
             \|\px^{\la k-1 \ra}\py^2 \q \y^{l-1}\|_{L^2_x L^2_y}
             \e^k_l(x)\\
&+o_L(1)C_{k_*, k, l}\|\px^{\la k-1 \ra} \py^2 \v \y \|_{L^2_x L^2_y}
             (\|\px^{\la [\f{k-1}{2}] \ra}\py^2 \q \y^{l-1}|_{x=0}\|_{L^2_y}
              +\|\px^{\la [\f{k-1}{2}]+1 \ra}\py^2 \q \y^{l-1}\|_{L^2_x L^2_y})
             \e^k_l(x)\\
\le
&o_L(1)C_{k_*, k^*, k, l,\delta_0}(1+\b^k_l(0)^4+\x^4),
\dal
\deqq
where we have used the following inequality in the last inequality
\beqq
\|\px^{\la [\f{k-1}{2}] \ra}\py^2 \q \y^{l-1}|_{x=0}\|_{L^2_y}
\le C_{k_*,k^*}(\|\px^{\la [\f{k-1}{2}] \ra}\py^2 \q \y^{l-1}(1-\chi)|_{x=0}\|_{L^2_y}
  +\|\px^{\la [\f{k-1}{2}] \ra}\py^3 \v|_{x=0}\|_{L^2_y}).
\deqq
Thus, we can obtain the following estimate
\beqq
|II_{83}| \le o_L(1)C_{k_*, k^*, k, l,\delta_0}(1+\b^k_l(0)^4+\x^{4}).
\deqq
This, together with the estimates \eqref{ii81} and \eqref{ii82}, yields directly
\beqq
|II_{8}| \le o_L(1)C_{k_*, k^*, k, l,\delta_0}(1+\b^k_l(0)^4+\x^{4}),
\deqq
which is the inequality \eqref{2412}.
\end{proof}

Now let us give the proof for the Proposition \ref{Uniform estimate}.

\begin{proof}[\textbf{Proof of Proposition \ref{Uniform estimate}}]
Based on the estimates obtained so far, we can complete the proof of
Proposition \ref{Uniform estimate} in this subsection.
First of all, we give the proof for the estimate.
For three parameters $R, k_*$ and $k^*$, which will be defined later, we have
\beqq
\bal
L^\eps_*:=
&\sup\left\{L\in[0, 1]|
\x^2+\b^k_l(x)^2\le R, \quad
\u(x, y)\ge k_* (\delta_0+\eps),\quad
(x, y) \in [0, L^\eps] \times [\delta_0, +\infty);\right.\\
&\quad \quad \quad \quad \quad \quad \quad
\left. k_* (y+\eps) \le \u(x, y) \le k^* (y+\eps),
\ (x, y)\in [0, L^\eps] \times [0, \delta_0]\right\},
\dal
\deqq
where the constant $\delta_0$ is defined by the property of initial
data $u_0(y)$.
Recall the estimates \eqref{232}, \eqref{233}, \eqref{234}
and \eqref{236} in Lemma \ref{lemma-one}, we have
\beqq
\sum_{j=1}^3\|\px^{\la k-1 \ra} \py^j \v \y^l\|_{L^2_x L^2_y}^2
+\|\py^4 \v \chi\|_{L^2_x L^2_y}^2
\le o_L(1)C_{k_*,l,k,\delta_0}(1+\b^k_l(0)^{16}+\x^{16}),
\deqq
and estimate \eqref{241} in Lemma \ref{lemma-two}
\beqq
\e^k_l(x)^2 \le \e^k_l(0)^2
            +o_L(1)C_{k_*, k^*, k, l,\delta_0}(1+\b^k_l(0)^{16}+\x^{16}),
\deqq
then we can obtain the following estimate
\beq\label{c1}
X^k_l(x)^2 \le \e^k_l(0)^2
               +o_L(1)C_{k_*, k^*, k, l,\delta_0}(1+\b^k_l(0)^{16}+X^k_l(x)^{16}).
\deq
On the other hand, the quantity $\b^k_l(x)$ can be controlled as follows
\beq\label{c2}
\b^k_l(x)
\le \b^k_l(0)+o_L(1)C_{k_*,k^*}\x.
\deq
Then, the combination of estimates \eqref{c1} and \eqref{c2} yields directly
\beq\label{c4}
X^k_l(x)^2+\b^k_l(x)^2
\le \e^k_l(0)^2+C_{k_*, k^*, k, l,\delta_0}(1+\b^k_l(0)^{16})
    +o_L(1)C_{k_*, k^*, k, l,\delta_0}\x^{16}.
\deq
It is easy to check that there exists a constant
only depends on the initial data $u_0$ such that
\beq\label{initial-control}
\e^k_l(0)+\b^k_l(0)\le C(u_0),
\deq
which can be found in \eqref{initial-control-01} in Appendix \ref{appendix-c}.
Then, the combination of \eqref{c4} and \eqref{initial-control} yields directly
\beqq
X^k_l(x)^2+\b^k_l(x)^2
\le C_{k_*, k^*, k, l,\delta_0}(1+ C(u_0))
    +o_L(1)C_{k_*, k^*, k, l,\delta_0}\x^{16}.
\deqq
Thus, we concludes for all $L_1\le L^\eps$ that
\beqq
X^k_l(L_1)^2+\b^k_l(L_1)^2
\le C_{k_*, k^*, k, l,\delta_0}(1+C(u_0))
    +o_{L_1}(1)C_{k_*, k^*, k, l,\delta_0}R^{16}.
\deqq
Let us choose constants $k_*=\f18,~ k^*=4$ and
$R=4 C_{k_*, k^*, k, l,\delta_0}(1+C(u_0))$, then we have
\beq\label{c3}
X^k_l(L_1)^2+\b^k_l(L_1)^2
\le \overline{C}_{k, l,\delta_0}(1+C(u_0))
    +o_{L_1}(1)\overline{C}_{k, l,\delta_0}\{4(1+C(u_0))\}^{16}.
\deq
Here we denote $\overline{C}_{k, l,\delta_0}:=
C_{k_*, k^*, k, l,\delta_0}|_{(k_*, k^*)=(\f18, 4)}$.
Choose the time $L_1$ small enough such that
\beqq
o_{L_1}(1)\overline{C}_{k, l,\delta_0}\{4(1+C(u_0))\}^{16}
\le \overline{C}_{k, l,\delta_0}(1+C(u_0)),
\deqq
and hence, we deduce from \eqref{c3} that
\beq\label{uniform-bound}
X^k_l(L_1)^2+\b^k_l(L_1)^2
\le 2\overline{C}_{k, l,\delta_0}(1+C(u_0))
=\f{R}{2}.
\deq
It is easy to check that
\beq\label{u-pre}
\u(x, y)=\u(0, y)+\int_0^x \px \u(s, y)ds,
\deq
and
\beqq
\px \u(x, y)=\px \u(x, 0)+\int_0^y \py \px \u(x, \tau)d\tau.
\deqq
Then, using the boundary condition \eqref{bdyc}, we have for all $y\in [0, \d_0]$
\beqq
|\px \u(x, y)|
\le y \|\py \px \u(x, y)\|_{L^\infty_y([0, y])}
\le y \|\py^2 \px \u(x,y) \y\|_{L^2_y([0,+\infty))},
\deqq
and hence, it is easy to check that
\beq\label{c5}
|\int_0^x \partial_x \u (x, y)dx|
\le \sqrt{x}y\|\py^3 \v \y\|_{L^2_x L^2_y}
\deq
Then, the combination of representation \eqref{u-pre} and estimate \eqref{c5} yields directly
\beqq
\u(x, y)
\ge \u(0, y)-|\int_0^L \px \u(s, y)d s|
\ge \f12(y+\eps)-\sqrt{L}y\|\py^3 \v \y\|_{L^2_x L^2_y}.
\deqq
Choose
$L_2:=\min\{\f{1}{32\overline{C}_{k, l,\delta_0}\{1+C(u_0)\}}, L_1\}$,
then we have
\beqq
\u(x, y)
\ge \f12(y+\eps)-\f14 y\ge \f14(y+\eps)=2k_*(y+\eps),
\text{~for~all~}(x, y)\in [0, L_2]\times [0, \d_0].
\deqq
Similarly, it is easy to check that
\beqq
\u(x, y)
\le \u(0, y)+|\int_0^L \px \u(s, y)d s|
\le \f32(y+\eps)+\sqrt{L}y\|\py^3 \v \y\|_{L^2_x L^2_y}.
\deqq
Choose
$L_3=\min\{L_2, \f{1}{8\overline{C}_{k, l,\delta_0}
     \{1+C(u_0)\}}\}$,
then we have
\beqq
\u(x, y)
\le \f32(y+\eps)+\f12 y
=2(y+\eps)=\f12 k^*(y+\eps),
\text{~for~all~}(x, y)\in [0, L_3]\times [0, \d_0].
\deqq
Finally, using the Sobolev inequality, H\"{o}lder inequality
and divergence-free condition, we have
\beqq
|\ix \partial_x \u (x, y)dx|
\le \ix \|\py \px \u \y\|_{L^2_y} dx
\le \sqrt{L}\|\py^2 \v \y\|_{L^2_x L^2_y},
\deqq
which, together with \eqref{u-pre}, yields directly
\beqq
\u(x, y)
\ge \u(0, y)-|\int_0^L \px \u(s, y)d s|
\ge (\d_0+\eps)-\sqrt{L}\|\py^2 \v \y\|_{L^2_x L^2_y}.
\deqq
Choose
$L_4:=\min\{L_3, \f{\d_0^2}{32\overline{C}_{k, l,\delta_0}
             \{1+C(u_0)\}}\}$,
then we have
\beqq
\u(x, y)
\ge \f12(\d_0+\eps)-\f14 \d_0
\ge \f14(\d_0+\eps)=2 k_*(\d_0+\eps),
\text{~for~all~}(x, y)\in [0, L_4]\times [\d_0, +\infty).
\deqq
Obviously, we conclude that there exists $L_4>0$ depending on
$k,l,\d_0$ and the initial data(hence independent of parameter $\eps$)
such that for all $L\le \min\{L_4, L^\eps\}$, the estimates \eqref{uniform-1}
and \eqref{uniform-2} in Proposition \ref{Uniform estimate} hold true.
Of course, it holds that $L_4 \le L^\eps$.
Indeed otherwise, our criterion about the continuation of the solution would
contradict the definition of $L^\eps_*$.
Then, taking $L_{a}=L_4$, we obtain the estimates \eqref{uniform-1}
and \eqref{uniform-2}.
Furthermore, due to the equation \eqref{app-q},
we can apply the estimates \eqref{uniform-bound} and
\eqref{238}(see Proposition \ref{p3} in Appendix \ref{appendix-a}) to get
\beqq
\sum_{0\le 2\alpha+\beta \le 2k+1}
\|\px^\alpha \py^{\beta} \py \v\y^l\|_{L^2_x L^2_y}
\le C_{k,l,\d_0}(1+C(u_0)),
\deqq
where $C(u_0)$ is a constant depends on the initial data $u_0$.
Therefore, the proof of Proposition \ref{Uniform estimate} is completed.
\end{proof}

\subsection{Local existence of Prandtl-Hartmann system}\label{Local_existence}

In this subsection, we will give the proof of local existence of
original system \eqref{H-eq}.
For any $\eps>0$ and some large $N>0$, we introduce an artificial truncation at $y=N$
and then send $N \rightarrow \infty$.
For any $\eps \in (0, 1]$ and large $N$,
let us consider the problem on $(x, y)\in (0, L) \times (0, N)$,
\beq\label{app-sym-n}
  	\left\{\begin{aligned}
  &u^{\eps, N} \px u^{\eps, N} +v^{\eps, N} \py u^{\eps, N}
   -\pyy u^{\eps, N}+u^{\eps, N}-u_0(N)-\eps=0,\\
  &\px u^{\eps, N}+\py v^{\eps, N}=0,
  	\end{aligned}\right.
\deq
with the boundary conditions
\beq\label{bdyc-n}
u^{\eps, N}(x, y)|_{y=0}=\eps, \
u^{\eps, N}(x, y)|_{y=N}=u_0(N)+\eps,\
v^{\eps, N}(x, y)|_{y=0}=v^{\eps, N}(x, y)|_{y=N}=0, \quad  x\in (0, L),
\deq
and the initial data
\beq\label{ic-n}
u^{\eps, N}(x, y)|_{x=0}:={u}_0(y)+\eps,\quad y \in [0, N].
\deq
Applying the classical von-Mises transformation
(i.e.,$\eta:=\int_0^y u^{\eps, N}(x, y')dy'$), the system \eqref{app-sym-n}
can be translated into a quasi-linear parabolic equation that will not cause
the loss of $x-$derivative.
Then, it is easy to obtain the local-in-time well-posedness with
higher order Sobolev regularity for this translated system.
Due to the non-degeneracy boundary of horizontal velocity in \eqref{bdyc-n},
the variable $\eta$ after von-Mises transformation
is equivalent to the vertical variable $y$ in the original Eulerian coordinates.
Then, the system \eqref{app-sym-n}-\eqref{ic-n} be obtained local
well-posedness with higher Sobolev regularity on the life span $[0, L^\eps]$.
For more details on the results of this local existence, the reader can refer to
the Lemma 2.2 in \cite{Guo-Nguyen2}.
Then, applying the a priori estimates given in Proposition \ref{Uniform estimate}
to the solution $({u}^{\eps, N}, {v}^{\eps, N})$, we obtain a uniform life span time
$L_a>0$ and a constant $C_1$(independent of $\eps$ and $N$) such that it holds true
\beq
\sum_{0\le 2\alpha+\beta \le 2k+1}
\|\px^\alpha \py^\beta \py {v}^{\eps, N} \y^l\|_{L^2_x L^2_y[0, N]}\le C_1,
\deq
and
\beq
\begin{aligned}
&\f14 y \le {u}^{\eps, N}(x, y) \le 2(y+\eps), \ (x, y) \in [0, L_a] \times [0, \d_0],\\
&{u}^{\eps, N}(x, y)\ge \f14 \d_0, \ (x, y) \in [0, L_a] \times [\d_0, N].
\end{aligned}
\deq
Based on the uniform estimates for $({u}^{\eps, N}, {v}^{\eps, N})$,
one can pass the limit $\eps \rightarrow 0^+$ and $N\rightarrow +\infty$
to get a solution $(u, v)$ satisfying \eqref{H-eq} by using a strong compactness arguments.
Indeed, it follows from estimate that $\py {v}^{\eps, N}$ is bounded uniformly in
$L^\infty([0, L_a]; H^{2k+1}_l([0, N]))$,
and $\partial_{x}(\py v^{\eps, N}([0, N])) $ is bounded
uniformly in $L^2([0, L_a]; H^{2k-1}_l([0, N]))$.
Then, it follows from a strong compactness argument,
after taking a subsequence $\eps_k \rightarrow 0^+$,
\beqq
\py {v}^{\eps_k, N} \  {\rightharpoonup}\  w^N
\quad {\rm in}~L^2([0, L_a]; H^{2k+1}_l([0, N])),
\deqq
and
\beqq
\py {v}^{\eps_k, N} \  {\longrightarrow} \ w^N
\quad {\rm in}~C([0, L_a];H^{2k-2}([0,N])),
\deqq
where $w^N \in L^2([0, L_a]; H^{2k+1}_l([0, N])) \cap L^\infty([0, L_a];H^{2k-2}([0,N]))$.
Let us define $v^N(x, y):=\int_0^y w^N (x, y') dy'$ and
$u^N(x, y):={u}_0(y)-\int_0^x w^N(x', y) dx'$, it holds true
\beqq
\px u^N(x, y)+\py v^N(x, y)=0.
\deqq
Furthermore, it is easy to deduce that
\beqq
\bal
{u}^{\eps_k, N}(x, y)-u^N(x, y)
&=-\int_0^x \py {v}^{\eps_k, N}(x',y)dx'+{u}_0(y)+\eps_k
 -{u}_0(y)+\int_0^x w(x', y)dx'\\
&=-\int_0^x (\py {v}^{\eps_k, N}(x',y)- w^N(x', y))dx'+\eps_k,
\dal
\deqq
and hence, we have
\beqq
\bal
\|{u}^{\eps_k, N}-u^N\|_{L^\infty_x H^{2k-2}([0, N])}
&\le \|\int_0^x (\py {v}^{\eps_k, N}(x',y)- w^N(x', y))dx'\|_{L^\infty_x H^{2k-2}
([0, N])}+\eps_k \sqrt{N}\\
&\le L_a\|\py {v}^{\eps_k, N}-w^N\|_{L^\infty_x H^{2k-2}([0, N])}+\eps_k \sqrt{N}.
\dal
\deqq
Therefore, it holds true
\beqq
{u}^{\eps_k, N} \  {\longrightarrow} \ u^N \quad {\rm in}
~L^\infty([0, L_a];H^{2k-2}([0,N])).
\deqq
Therefore, the solution $(u^N(x, y), v^N(x,y))$ satisfies
\beqq
  	\left\{\begin{aligned}
  &u^N \px u^N +v^N \py u^N-\pyy u^N+u^N-u_0(N)=0,\ (x, y) \in [0, L_a]\times(0, N);\\
  &\px u^N +\py v^N=0,\ (x, y) \in [0, L_a] \in (0, N),
  	\end{aligned}\right.
\deqq
with the boundary condition
\beqq
u^N(x, y)|_{y=0}=0, \quad
u^N(x, y)|_{y=N}=u_0(N),\quad
v^N(x, y)|_{y=0}=0, \quad  x\in (0, L_a),
\deqq
and the initial data
\beqq
u^N(x, y)|_{x=0}={u}_0(y),\quad y \in [0, N].
\deqq
Furthermore, the solution $(u^N, v^N)$ also satisfies the estimates
\beq\label{est-n1}
\sum_{0\le 2\alpha+\beta \le 2k+1}
\|\px^\alpha \py^\beta \py v^N \y^l\|_{L^2_x L^2_y[0, N]}
\le C_{k,l,\delta_0}(1+C(u_0)),
\deq
and
\beq\label{est-n2}
\f14 y \le u^N(x, y) \le 2 y, \ (x, y) \in [0, L_a] \times [0, \d_0];\quad \quad
u^N(x, y)\ge \f14 \d_0, \ (x, y) \in [0, L_a] \times [\d_0, N].
\deq
Since the estimate \eqref{est-n1} is bounded independent of $N$,
and similarly taking $N\rightarrow +\infty$, it is easy to obtain the solution
$(u, v)$ to the original problem Eqs.\eqref{H-eq}-\eqref{H-bc}.
At the same time, the estimates \eqref{est-n1} and \eqref{est-n2} hold on
over the domain $[0, +\infty)$. Therefore, we complete the proof of local
existence and estimate \eqref{THM_estimate} in Theorem \ref{high-regu}.

\subsection{Uniqueness of Prandtl-Hartmann system}\label{uniqueness}

In this subsection, we will show the uniqueness of solution to the original
system \eqref{H-eq}.
Let $(u_1, v_1)$ and $(u_2, v_2)$ be two solutions in the existence time $[0, L_a]$
constructed in the previous subsection, with respect to the same initial data.
Let us set
$$
(\bu, \bv, \bq):=(u_2-u_1, v_2-v_1, \f{v_2-v_1}{u_2})
$$
then they satisfy the following evolution
\beqq
(u_2)^2 \py \bq+\py^2 \bu-\bu=\bu \px u_1+v_1 \py \bu.
\deqq
Taking $\px$ operator to the above equation and
using the divergence-free condition, we have
\beq\label{equation-d}
\px\{(u_2)^2 \py \bq\}-\py^3 \bv+\py \bv
=\px(\bu \px u_1+v_1 \py \bu),
\deq
with the zero boundary condition and initial data
\beq\label{initial-boundary-condtion}
(\bu, \bv)|_{y=0}=0, \quad \lim_{y \rightarrow+\infty}\bu=0,
\quad \bu|_{x=0}=0.
\deq

Next, we will establish the following estimate.
\begin{prop}\label{lemma-uniq}
Let $(u_1, v_1)$ and $(u_2, v_2)$ be two solutions of the equation \eqref{H-eq}
with the same initial data \eqref{id}, and satisfying the estimate 
\eqref{THM_estimate} respectively.
Then, the quantity $(\bu, \bv, \bq)$ given by satisfies
\beq\label{stability}
\|(u_2\py \bq)(x) \|_{L^2_y}^2
+\int_0^x \|(\sqrt{u_2} \py^2 \bq)(\tau)\|_{L^2_y}^2 d\tau
+\int_0^x \|\sqrt{u_2} \py \bq\|_{L^2_y}^2 d\tau
\le C\int_0^x  \|(u_2\py \bq)(\tau) \|_{L^2_y}^2 d\tau,
\deq
for any $x\in (0, L_a]$.
\end{prop}

\begin{proof}[\textbf{Proof of Uniqueness.}]
Applying the Gronwall's lemma to the estimate \eqref{stability},
we obtain $(u_2\py \bq)(x, y)\equiv 0$.
Due to the fact $u_2>0$ in the fluid domain, then we have
\beqq
v_2-v_1=u_2 w
\deqq
for some function $w:=w(x)$.
Using the assumption $u_2>0$  in the fluid domain
and boundary condition $\bv|_{y=0}=0$,
we know that $w \equiv 0$ and hence $v_2 \equiv v_1$.
Using the divergence-free condition, it holds
\beqq
u_2(x, y)-u_1(x, y)
=\int_0^x \partial_\tau\{u_2(\tau, y)-u_1(\tau, y)\}d\tau
=-\int_0^x \partial_y \{v_2(\tau, y)-v_1(\tau, y)\}d\tau
=0.
\deqq
Then, we have $u_2(x, y)\equiv u_1(x, y)$.
This proves the uniqueness of solution in Theorem \ref{high-regu}.
\end{proof}

In the rest of this subsection, we will give the proof
of Proposition \ref{lemma-uniq} as follows.
\begin{proof}[\textbf{Proof of Proposition \ref{lemma-uniq}}]
Multiplying the equation \eqref{equation-d}
by $\py \bq$ integrating over $[0, x]\times[0,+\infty)$
and integrating by part, we have
\beq\label{g1}
\bal
&\f12\iy \!(u_2^2 |\py \bq|^2) (x, y)dy
 +\!\int_0^x \!\!\iy \! \!u_2 |\py^2 \bq|^2  dy d\tau
 +\!\int_0^x \! \!\iy \!\! u_2 |\py^2 \bq|^2  dy d\tau
 +\!\int_0^x \!\!2 \py u_2 |\py \bq|^2|_{y=0} d\tau\\
=&-\int_0^x \iy u_2 \px u_2 |\py \bq|^2  dy d\tau
  -\int_0^x \iy (2 \py u_2 \py \bq+\py^2 u_2 \bq)\cdot \py^2 \bq  dy d\tau\\
&-\int_0^x \iy \py u_2 \bq \py \bq  dy d \tau
+\int_0^x \iy(\px v_1 \py \bu-v_1 \py^2 \bv+\py \bu \py v_1-\bu \px \py v_1)
\py \bq dy  d\tau \\
:=&III_{1}+III_{2}+III_{3}+III_{4}+III_{5}.
\dal
\deq
Similar to the estimates in Lemma \ref{lemma-two}, we can obtain
\beqq
\bal
|III_{1}|&\le (1+\|\py u_2 \y|_{x=0}\|_{L^2_y}+\|\py^2 v_2 \y|_{x=0}\|_{L^2_y}
               +\|\py^2 \px^{\la 1 \ra} v_2 \y\|_{L^2_x L^2_y})
                \int_0^x \iy |\py \bq|^2 dy d\tau ,\\
|III_{2}|&\le \int_0^x \py u_2 |\py \bq|^2|_{y=0} dx
             +C(\|\py^2 u_2\y|_{x=0}\|_{L^2_y}+\|\py^2 v_2 \y\|_{L^2_x L^2_y})
               \int_0^x \iy |\py \bq|^2 dy d\tau ,\\
|III_{3}|&\le (\|\py^3 u_2 \y|_{x=0}\|_{L^2_y}+\|\py^4 v_2 \y\|_{L^2_x L^2_y})
              \int_0^x \iy |\py \bq|^2 d\tau dy,\\
|III_{4}|&\le  (\|\py u_2 \y|_{x=0}\|_{L^2_y}+\|\py^2 v_2 \y\|_{L^2_x L^2_y})
              \int_0^x \iy |\py \bq|^2 dy d\tau .
\dal
\deqq
It is easy to check that
\beqq
\bal
|III_{5}|
\le  &\|\px v_1\|_{L^2_x L^\infty_y}
     \|\py \bu\|_{L^\infty_x L^2_y}
     \|\py \bq\|_{L^2_x L^2_y}
     +\|v_1\|_{L^\infty_x L^\infty_y}
     \|\py^2 \bv\|_{L^2_x L^2_y}
     \|\py \bq\|_{L^2_x L^2_y}\\
     &+\|\px \bu\|_{L^\infty_x L^2_y}
     \|\py v_1\|_{L^2_x L^\infty_y}
     \|\py \bq\|_{L^2_x L^2_y}
     +\|\bu \y^{-1}\|_{L^\infty_x L^\infty_y}
     \|\px \py v_1\|_{L^2_x L^2_y}
     \|\py \bq\|_{L^2_x L^2_y}\\
\le &(\|\py v_1 \y|_{x=0}\|_{L^2_y}+\|\px \py v_1 \y\|_{L^2_x L^2_y}
      +\|\py^2 v_1 \y\|_{L^2_x L^2_y})
      \|\py^2 \bv\|_{L^2_x L^2_y}\|\py \bq\|_{L^2_x L^2_y}.
\dal
\deqq
Recall the definition $\bq:=\f{v_2-v_1}{u_2}=\f{\bv}{u_2}$,
then it holds true
\beqq
\bal
\|\py^2 \bv\|_{L^2_x L^2_y}
\le
&(\|\py u_2 \y|_{x=0}\|_{L^2_y}+\|\py^2 u_2 \y|_{x=0}\|_{L^2_y}
     +\|\py^2 v_2 \y\|_{L^2_x L^2_y}+\|\py^3 v_2 \y\|_{L^2_x L^2_y})
     \|\py \bq\|_{L^2_x L^2_y}\\
&+(1+\|\py u_2 \y|_{x=0}\|_{L^2_y}+\|\py^2 \y\|_{L^2_x L^2_y})
   \|\sqrt{u_2}\py^2 \bq\|_{L^2_x L^2_y}.
\dal
\deqq
Thus, collecting the estimates from $III_{1}$ to $III_{5}$, we can get
\beq\label{g2}
|III_{1}+III_{2}+III_{3}+III_{4}+III_{5}|
\le \int_0^x \py u_2 |\py \bq|^2|_{y=0} dx
     +\f14 \|\sqrt{u_2}\py^2 \bq\|_{L^2_x L^2_y}^2
     +C\|\py \bq\|_{L^2_x L^2_y}^2.
\deq
Similar to the estimate \eqref{claim21}, it is easy to obtain
\beqq
\|\py \bq \|_{L^2_x L^2_y}
\le 4C\|u_2 \py \bq\|_{L^2_x L^2_y}
    +\f{1}{4C}\|\sqrt{u_2} \py^2 \bq \|_{L^2_x L^2_y},
\deqq
which, together with estimate \eqref{g2}, yields directly
\beq\label{g3}
|III_{1}+III_{2}+III_{3}+III_{4}+III_{5}|
\le \int_0^x \py u_2 |\py \bq|^2|_{y=0} dx
     +\f12 \|\sqrt{u_2}\py^2 \bq\|_{L^2_x L^2_y}^2
     +4C^2\|u_2 \py \bq\|_{L^2_x L^2_y}.
\deq
Substituting the estimate \eqref{g3} into the inequality \eqref{g1},
then we obtain the claim estimate \eqref{stability}.
Therefore, we complete the proof of Proposition \ref{lemma-uniq}.
\end{proof}

\section{Asymptotic behavior of Prandtl-Hartmann system}\label{Asymptotic}
In this section, we aim to investigate the asymptotic behaviour of velocity profiles in the 2-D magnetic Prandtl boundary layer.
More specifically, we would like to establish the decay estimate in $x$ variable, thus $C$ denotes a generic positive constant independent of $x$ which may vary in following different estimates in this section.

Then an important question is whether the system \eqref{H-eq} is globally well-posedness in Sobolev space when the initial data are near the Hartmann layer.
More precisely, we study the global stability of the Hartmann layer $(\us,\vs)=(1-e^{-y},0)$.
Thus, we need to define an appropriate notion of the difference between two pairs $(u,v)$ and $(\us,\vs)$. The notion we used in our case does not coincide with pointwise perturbations $u(x,y)-\us(y)$ and $v(x,y)$, but rather the modulated perturbation $\phi(x,\psi):=u^2(x,\psi)-\us^2(\psi)$.
Then it follows from \eqref{u2eq} and \eqref{u2bc} that the difference $\phi$ satisfies\beq\label{eq3}
\left\{\begin{array}{lr}
	\phi_x-u\phi_{\psi\psi}+2\frac{\phi}{\us(u+\us)}=0,\quad (x,\psi)\in\mathbb{R}_{+}\times\mathbb{R}_{+},\\
	\phi|_{x=0}=\phi_0(\psi):=u_0^2(\psi)-\us_0^2(\psi),\quad \psi\in\mathbb{R}_{+},\\
	\phi|_{\psi=0}=0,~~\phi|_{\psi\rightarrow\infty}=0,\quad x\in\mathbb{R}_{+}.
\end{array}\right.
\deq	
We rewrite $\eqref{eq3}_1$ as
\begin{equation*}
\phi_x+L\phi=0,
\end{equation*}
where the operator $L\phi:=-u\phi_{\psi\psi}+2\frac{\phi}{\us(u+\us)}$.
The global existence and the asymptotic behaviour of the solution $(u,v)(x,y)$ to \eqref{H-eq} can be easily translated into that of the perturbed solution $\phi(x,\psi)$.
And it is universally known that the global existence of solutions for system \eqref{eq3} will be obtained by combining the local existence result with some a priori estimates and then employing the standard continuity argument.
Thus, we first review the classical local existence result, which can be obtained by a similar way to Oleinik's result in \cite{Oleinik2}. And we omit detalis of proof here.
\begin{prop}\label{local-result}
	Assume that $u_0(y)>0$ for $y>0$; $u_0(0)=0$, $u_0'(0)>0$, $u_0(y)\rightarrow 1$ as $y\rightarrow\infty$; $u_0(y)$, $u'_0(y)$, $u''_0(y)$ are bounded for $0\leq y<\infty$ and satisfy the H{\"o}lder condition. Moreover, assume that for small $y$ the following compatiblity condition is satisfied at the point $(0,0)$:
	\[-u''_0(y)+u_0(y)-1=O(y^2).\]
	Then, for some $X>0$ there exists a solution $(u,v)$ of problem \eqref{H-eq} and \eqref{H-bc} in $D:=\{0<x<X,0<y<\infty\}$, which has the following properties: $u(x,y)$ is bounded and continuous in $\bar D$, $u>0$ for $y>0$; $u_y>m>0$ for $0<y\leq y_0$, where $m$ and $y_0$ are constants; $u_y$ and $u_{yy}$ are bounded and continuous in $D$; $u_x$, $v$ and $v_y$ are bounded and continuous in any finite portion of $\bar D$.
\end{prop}
It is easy to check that the local existence above is also valid for system \eqref{eq3} in new variables $(x,\psi)$.
First, we set $I:=[0,x]$, and denote the norms
\beqq
\begin{aligned}
f(x):=&\underset{x'\in I}{\sup} \Big\{\|\f{u}{\us}\|_{L^\infty_\psi},\|\f{\us}{u}\|_{L^\infty_\psi} \Big\},\\
\e(x):=&
\sum_{k=0,1}\sup_{x'\in I}\Big(\|\p_x^k\phi\|_{L^2_\psi}
 +\|\f{\p_x^k\phi}{\sqrt{u}}\|_{L^2_\psi}
 +\|\p_x^k\p_{\psi}\phi\|_{L^2_\psi}+\|\frac{\p_x^k\phi}{u^{\f32}}\|_{L^2_\psi}\Big),\\
\D(x):=&\e(x)+
\sum_{k=0,1}\left\{\Big\|\|\sqrt{u}\p_x^k\p_{\psi}\phi\|_{L^2_\psi}\Big\|_{L^2_x}
+\Big\|\|\frac{\p_x^k\phi}{u}\|_{L^2_\psi}\Big\|_{L^2_x}\right\}\\
&+\sum_{k=0,1}\left\{\Big\|\|\p_x^k\p_{\psi}\phi\|_{L^2_\psi}\Big\|_{L^2_x}
+\Big\|\|\frac{\p_x^k\phi}{u^{\f32}}\|_{L^2_\psi}\Big\|_{L^2_x}
+\Big\|\|\f{\p_x^{k+1}\phi}{\sqrt{u}}\|_{L^2_\psi}\Big\|_{L^2_x}\right\}.
\end{aligned}
\deqq
Based on the local existence result, we can obtain the global existence and asympotic stability results as follow.
\begin{prop}[Global well-posedness and asymptotic stability in von-Mises variable]\label{them1}
Under the conditions of Proposition \ref{local-result}, assume that
initial data satisfies $f(0)\leq \gamma_0$ with any positive constant $\gamma_0>1$.
Furthermore, there exists a small positive constanst $\sigma_0$ such that
\begin{equation}\label{initial-data-assumption}
\|\phi_{0\psi}\|_{L^2_\psi}+\|\f{\phi_0}{u_0^{3/2}}\|_{L^2_\psi}
+\|(L\phi_{0})_{\psi}\|_{L^2_\psi}
+\|\f{L\phi_0}{u_0^{3/2}}\|_{L^2_\psi}\leq \sigma_0,
\end{equation}
then the system \eqref{eq3} has a global-in-$x$ solution
$\phi$ satisfying for any $0<x<\infty$,
\[f(x)\leq 2 \gamma_0,\quad \D(x)\leq C_1 \e(0), \]
and the decay estimate
	\begin{equation}\label{decay}
    \e(x) \leq C_2e^{-x}.
	\end{equation}
Here $C_1$ and $C_2$ are positive constants independent of $x$.
\end{prop}
\begin{rema}
The initial condition $f(0)\leq \gamma_0$ implies that $u$ and $\us$
tend to zero in the same order near the boundary $\{\psi=0\}$.
And in general $u \neq \us$, thus it is reasonable to assume that $\gamma_0>1$.
\end{rema}
\begin{rema}
Indeed, we require the initial data satisfy the compatibility conditions:
$(\phi_x-L \phi)|_{x=0}=(\phi_{x\psi}-L (\phi_\psi))|_{x=0}=0$ in Proposition \ref{them1}.
At the same time, we apply the fact $u \lesssim 1$ to check that
\begin{equation*}
\e(0)\le \|\phi_{0\psi}\|_{L^2_\psi}+\|\f{\phi_0}{u_0^{3/2}}\|_{L^2_\psi}
+\|(L\phi_{0})_{\psi}\|_{L^2_\psi}
+\|\f{L\phi_0}{u_0^{3/2}}\|_{L^2_\psi}.
\end{equation*}
This is reason why we require the small initial data \eqref{initial-data-assumption}
instead of $\e(0) \le \sigma_0$.
\end{rema}

\subsection{A priori estimates}\label{pe}
The global existence of solution will be obtained by combing the local existence
result with a priori estimates.
Since the local existence has been built in Proposition \ref{local-result},
our main task is to establish the closed a priori estimate.
Therefore, for any small constant $\sg>0$ and positive constant $\gm>1$,
we assume a priori estimates
\beq\label{priassum}
f(x)\leq 2 \gm, \quad \e(x)\leq \sg.
\deq
Thus, we will establish some estimates for the system \eqref{eq3}
under the a priori estimates \eqref{priassum}.
Let us denote
$$\om:=\f{\phi}{u+\us}=u-\us.$$
Next, we analyze the properties of the solution $\us$ for classical Hartmann boundary layer. We claim that there exists a positive constant $\eta_0$ such that
\beq\label{uu}
\begin{aligned}
	\us\sim \sqrt{\psi}~~\text{for}~~0<\psi\leq \eta_0~~\text{and}~~
	\us\gtrsim 1 ~~\text{for}~~\psi\geq \eta_0,
\end{aligned}
\deq
where we recall again that $A\lesssim B$ ($A\gtrsim B$) represents that there exists a positive constant $C$ such that $A\leq CB$ ($A\geq CB$),
and $A\sim B$ stands for $A\lesssim CB$ and $A\gtrsim CB$ (see the end of Section \ref{introduction}). The parameter $\eta_0$ appearing in \eqref{uu} will be chosen in the following proof, which is related to the Taylor expansion of $\us(y)$.

Indeed, we suppress the dependence of the solution $\us$ on $x$ in the following calculations for simplicity.
Using Taylor expansion, $\us|_{y=0}=0$ and $\us_{y}|_{y=0}=1>0$, we find
\[\us(y)=y+o(y). \]
Then, there exists a positive constant $\eta_1$ such that
\beq\label{usy}
\f12y\leq\us(y)\leq \f32y,~~\text{for}~~0<y\leq\eta_1.
\deq
Since $\psi=\int_{0}^{y}\us(y')dy'$, then a direct calculation yields for $0<y\leq \eta_1$,
\begin{equation}\label{psisimy}
\begin{aligned}
\psi\leq \f32\int_{0}^{y}y'dy'=\frac{3}{4} y^2,~~~
\psi\geq \f12\int_{0}^{y}y'dy'=\frac{1}{4} y^2.
\end{aligned}
\end{equation}
The estimates \eqref{psisimy}, together with another equivalent relation \eqref{usy}, yields that there exists a positive constant $\eta_0=\f14\eta_1^2$ such that
\[\us(\psi)\sim \sqrt{\psi},~~\text{for}~~0<\psi  \leq  \eta_0.\]
The fact $\psi\geq \eta_0$ implies that there exists a
positive constant $\eta_2$ such that $y \geq \eta_2$. And it is easy to check that
\[\us(y)=1-e^{-y}\geq 1-e^{-\eta_2}\gtrsim 1,
   ~~\text{for}~~y\geq\eta_2.\]
which immediately implies that
\[\us(\psi)\gtrsim1,~~\text{for}~~\psi\geq \eta_0. \]
This completes the proof of the claim \eqref{uu}.

First, our aim is to show that $f(x)\leq \gm$. And the proofs of following estimates in Lemmas \ref{a}-\ref{b} are inter-connected because of their nonlinear nature, thus it will be convenient for us to give the following notation:
$$\alpha(x):= \sup_{x'\in I}\|\f{\om}{u}\|_{L^\infty_{\psi}}.$$
\begin{lemm}\label{a}
Under the  assumption \eqref{priassum}, we have
	\begin{equation}\label{alpha}
	\alpha(x)\leq {\bar Cf^{\f{17}{2}}(x)}\e(x),
	\end{equation}
	where $\bar C$ is a positive constant independent of $x$.
\end{lemm}
\begin{proof}
	It follows from equation $\eqref{eq3}_1$ that
	\begin{equation}\label{u32phi}
	\|u^{\f32}\phi_{\psi\psi}\|_{L^2_\psi}\leq C\big(\|\sqrt{u}\phi_x\|_{L^2_\psi}+\|\frac{\phi}{u^{\f32}}\|_{L^2_\psi}\|\f{u}{u+\us} \|_{L^\infty_\psi}\|\f{u}{\us} \|_{L^\infty_\psi}\big)\leq Cf(x) \big(\|\frac{\phi}{u^{\f32}}\|_{L^2_\psi}+\|\sqrt{u}\phi_x\|_{L^2_\psi}\big),
	\end{equation}
	and
	\begin{equation}\label{uph1}
	\begin{aligned}
	\|\sqrt{u}\phi_{\psi\psi}\|_{L^2_\psi}\leq  C\big(\|\frac{\phi_x}{\sqrt{u}}\|_{L^2_\psi}+\|\frac{\phi}{u^{\f52}}\|_{L^2_\psi}\|\f{u}{u+\us} \|_{L^\infty_\psi}\|\f{u}{\us} \|_{L^\infty_\psi}\big)
	\leq Cf(x)\big(\|\frac{\phi}{u^{\f52}}\|_{L^2_\psi}+\|\frac{\phi_x}{\sqrt{u}}\|_{L^2_\psi}\big).
	\end{aligned}
	\end{equation}
	It is easy to check from the definition of $\e(x)$ that the second term on the right handside of \eqref{u32phi} can be controlled by $\e(x)$. For the first term, consider the case where $\psi\geq \eta_0$,
	\[\|\f{\phi(x,\psi)}{u^{\f32}(x,\psi)}(1-\chi(\f{\psi}{\eta_0}))\|_{L^2_\psi}\leq C\|\phi\|_{L^2_\psi}, \]
	where the definition of the cutoff function $\chi$ can be found in the end of Section \ref{introduction}.
	And for the other case where $0\leq\psi\leq \eta_0$, using equivalent relation $\us\sim \sqrt{\psi}$ and Hardy inequality, we get
\begin{equation*}
\begin{aligned}
\|\f{\phi(x,\psi)}{u^{\f32}(x,\psi)}\chi(\f{\psi}{\eta_0})\|_{L^2_\psi}\leq
& Cf^{\f32}(x) \|\f{\phi(x,\psi)}{\psi^{\f34}}\chi(\f{\psi}{\eta_0})\|_{L^2_\psi} \leq Cf^{\f32}(x)\big(\|\phi(x,\psi)\psi^{\f14}\chi'(\f{\psi}{\eta_0})\|_{L^2_\psi} +\|\phi_{\psi}(x,\psi)\psi^{\f14}\chi(\f{\psi}{\eta_0})\|_{L^2_\psi}\big)\\
	\leq& Cf^{2}(x)\big(\|\sqrt{u}\phi\|_{L^2_\psi}+\|\sqrt{u}\phi_{\psi}\|_{L^2_\psi} \big),
	\end{aligned}
	\end{equation*}
	where the positive constant $C$ may depend on $\eta_0$.
	Combining two estimates above, and using the fact that $f(x)\geq 1$, we find that
	\begin{equation}\label{pu31}
	\|\f{\phi}{u^{\f32}}\|_{L^2_\psi}\leq Cf^{2}(x)\big(\|\phi\|_{L^2_\psi}+\|\sqrt{u}\phi_{\psi}\|_{L^2_\psi}\big).
	\end{equation}
	Inserting \eqref{pu31} into \eqref{u32phi}, we then obtain that
	\begin{equation}\label{u32}
	\|u^{\f32}\phi_{\psi\psi}\|_{L^2_\psi}\leq Cf^{3}(x) \big(\|\phi\|_{L^2_\psi}+\|\sqrt{u}\phi_{\psi}\|_{L^2_\psi}+\|\sqrt{u}\phi_x\|_{L^2_\psi}\big),
	\end{equation}
	Now returning to estimate \eqref{uph1}, it is easy to find the second term on the right handside of \eqref{uph1} can be  controlled by $\e(x)$.
	To obtain the estimate of the first term, we calculate the far-field contribution immediately by using the fact that $u\gtrsim1$ in the support of $1-\chi$:
	\begin{equation*}
	\begin{aligned}
	\|\f{\phi(x,\psi)}{u^{\f52}(x,\psi)}(1-\chi(\f{\psi}{\eta_0}))\|_{L^2_\psi}\leq \|\phi\|_{L^2_\psi}.
	\end{aligned}
	\end{equation*}
	And then, we can localize based on the location of $\psi$ and use Hardy inequality to find
	\begin{equation*}
	\begin{aligned}
	\|\f{\phi(x,\psi)}{u^{\f52}(x,\psi)}\chi(\f{\psi}{\eta_0})\|_{L^2_\psi}\leq& Cf^{\f52}(x)\|\f{\phi(x,\psi)}{\psi^{\f54}}\chi(\f{\psi}{\eta_0})\|_{L^2_\psi}\leq Cf^{\f52}(x)\big(\|\f{\phi(x,\psi)}{\psi^{\f14}}\chi'(\f{\psi}{\eta_0})\|_{L^2_\psi}
+\|\f{\phi_\psi(x,\psi)}{\psi^{\f14}}\chi(\f{\psi}{\eta_0})\|_{L^2_\psi}\big)\\
	\leq& Cf^{\f52}(x)\big(\|\phi(x,\psi)\psi^{\f34}\chi''(\f{\psi}{\eta_0})\|_{L^2_\psi}
+\|\phi_{\psi}(x,\psi)\psi^{\f34}\chi'(\f{\psi}{\eta_0})\|_{L^2_\psi}
+\|\phi_{\psi\psi}(x,\psi)\psi^{\f34}\chi(\f{\psi}{\eta_0})\|_{L^2_\psi}\big)\\ \leq&Cf^{4}(x)\big(\|u^{\f32}\phi\|_{L^2_\psi}+\|u^{\f32}\phi_{\psi}\|_{L^2_\psi}
+\|u^{\f32}\phi_{\psi\psi}\|_{L^2_\psi}\big)\\ \leq&Cf^{7}(x)\big(\|\phi\|_{L^2_\psi}+\|\sqrt{u}\phi_{\psi}\|_{L^2_\psi}
+\|\phi_x\|_{L^2_\psi}\big),
	\end{aligned}
	\end{equation*}
	where, we have used the estimate \eqref{u32} and the fact that $u\lesssim 1$ in the last inequality.
	Hence, we obtain that
	\begin{equation}\label{u521}
	\|\f{\phi}{u^{\f52}}\|_{L^2_\psi}\leq Cf^7(x)\big(\|\phi\|_{L^2_\psi}+\|\sqrt{u}\phi_{\psi}\|_{L^2_\psi}+\|\phi_x\|_{L^2_\psi}\big).
	\end{equation}
	Upon inserting into \eqref{uph1} yields that
	\begin{equation}\label{u2pp1}
	\begin{aligned}
	\|\sqrt{u}\phi_{\psi\psi}\|_{L^2_\psi}\leq  Cf^8(x)\big(\|\phi\|_{L^2_\psi}+\|\sqrt{u}\phi_{\psi}\|_{L^2_\psi}+\|\frac{\phi_x}{\sqrt{u}}\|_{L^2_\psi}\big).
	\end{aligned}
	\end{equation}
	And it then follows from \eqref{u32} that
	\begin{equation}\label{phs1}
	\begin{aligned}
	\|\f{\phi_\psi}{\sqrt{u}}\|_{L^2_\psi} \leq& Cf^{\f12}(x)\|\f{\phi_\psi(x,\psi)}{\psi^{\f14}}\chi(\f{\psi}{\eta_0})\|_{L^2_\psi}
+C\|\f{\phi_\psi(x,\psi)}{\sqrt{u(x,\psi)}}(1-\chi(\f{\psi}{\eta_0}))\|_{L^2_\psi}\\
	\leq& Cf^{\f12}(x)\big(\|\phi_{\psi}(x,\psi)\psi^{\f34}\chi'(\f{\psi}{\eta_0})\|_{L^2_\psi}
+\|\phi_{\psi\psi}(x,\psi)\psi^{\f34}\chi(\f{\psi}{\eta_0})\|_{L^2_\psi}\big)
+C\|\phi_\psi\|_{L^2_\psi}\\
	\leq&Cf^2(x)\big(\|\phi_\psi\|_{L^2_\psi}+\|u^{\f32}\phi_{\psi\psi}\|_{L^2_\psi}\big)\\
	\leq& Cf^5(x)\big(\|\phi\|_{L^2_\psi}+\|\phi_\psi\|_{L^2_\psi}+\|\phi_{x}\|_{L^2_\psi}\big).
	\end{aligned}
	\end{equation}
	By Sobolev inequality, we deduce from the estimates \eqref{u2pp1}-\eqref{phs1} that
	\begin{equation*}
	\begin{aligned}
	|\phi_{\psi}|^2\leq C\|\frac{\phi_\psi}{\sqrt{u}}\|_{L^2_\psi} \|\sqrt{u}\phi_{\psi\psi}\|_{L^2_\psi}
	\leq Cf^{13}(x)\big(\|\phi\|_{L^2_\psi}+\|\phi_\psi\|_{L^2_\psi}+\|\phi_{x}\|_{L^2_\psi}\big)\big(\|\phi\|_{L^2_\psi}+\|\sqrt{u}\phi_{\psi}\|_{L^2_\psi}+\|\frac{\phi_x}{\sqrt{u}}\|_{L^2_\psi}\big),
	\end{aligned}
	\end{equation*}
	from which, we obtain\beq\label{pinty1}
	\|\phi_{\psi}\|_{L^\infty_\psi}\leq Cf^{\f{13}{2}}(x)\big(\|\phi\|_{L^2_\psi}+\|\phi_\psi\|_{L^2_\psi}+\|\frac{\phi_x}{\sqrt{u}}\|_{L^2_\psi}\big).
	\deq
	We then deduce from \eqref{pinty1} that
	\beq\label{phiu2}
	\begin{aligned}
		\f{\phi(x,\psi)}{u^{2}(x,\psi)}\chi(\f{\psi}{\eta_0})\leq \f{Cf^2(x)}{\psi}\int_{0}^{\psi}|\phi_{\psi'}|d\psi'\leq Cf^2(x)\|\phi_{\psi}\|_{L^\infty_\psi}\leq Cf^{\f{17}{2}}(x)\big(\|\phi\|_{L^2_\psi}+\|\phi_\psi\|_{L^2_\psi}+\|\frac{\phi_x}{\sqrt{u}}\|_{L^2_\psi}\big).
	\end{aligned}
	\deq
	The far-field term can be estimated easily as follows:
	\beq\label{phiu2chi}
	\|\frac{\phi(x,\psi)}{u^2(x,\psi)}(1-\chi(\f{\psi}{\eta_0}))\|_{L^\infty_\psi}
\leq\|\phi\|_{L^\infty_\psi}\leq C\|\phi\|_{L^2_\psi}^{\f12}\|\phi_\psi\|_{L^2_\psi}^{\f12},
	\deq
	which, together with \eqref{phiu2}, yields that
	\beq\label{pu21}
	\begin{aligned}
		\|\frac{\phi}{u^2}\|_{L^\infty_\psi}\leq Cf^{\f{17}{2}}(x)\big(\|\phi\|_{L^2_\psi}+\|\phi_\psi\|_{L^2_\psi}+\|\frac{\phi_x}{\sqrt{u}}\|_{L^2_\psi}\big).
	\end{aligned}
	\deq
	Recalling the definition of $\om$, we deduce from \eqref{pu21} that
	\beqq
	\begin{aligned}
		\|\f{\om}{u}\|_{L^\infty_{\psi}}=&\|\f{\phi}{u(u+\us)}\|_{L^\infty_{\psi}}=\|\f{\phi}{u^2}\|_{L^\infty_{\psi}}\|\f{u}{u+\us}\|_{L^\infty_{\psi}}\leq \|\f{\phi}{u^2}\|_{L^\infty_{\psi}}\\
		\leq&{Cf^{\f{17}{2}}(x)}\big(\|\phi\|_{L^2_\psi}+\|\phi_\psi\|_{L^2_\psi}+\|\frac{\phi_x}{\sqrt{u}}\|_{L^2_\psi}\big)
		\leq {\bar Cf^{\f{17}{2}}(x)}\e(x),
	\end{aligned}
	\deqq
	which, together with the definition of $\alpha(x)$, gives \eqref{alpha} immediately.
\end{proof}
\begin{lemm}\label{b}
Under the  assumption \eqref{priassum}, we have
	\begin{equation}\label{beta}
	f(x)\leq \gm,
	\end{equation}
	provided that $\sg \leq \min\Big{\{}\f{1}{2^{\f{21}{2}}{\gm}^{\f{17}{2}}\bar C},\f{\gm-1}{2^{\f{19}{2}}{\gm}^{\f{19}{2}}\bar{C}}\Big{\}}$.
\end{lemm}
\begin{proof}
	The estimate \eqref{alpha} and the assumption \eqref{priassum} lead us to obatin
	\begin{equation}\label{al}
	\alpha(x)\leq \bar{C}f^{\f{17}{2}}(x)\e(x)
     \leq 2^{\f{21}{2}}{\gm}^{\f{17}{2}}\bar{C}\sg \leq \f14,
	\end{equation}
	provided that $\sg\leq \f{1}{2^{\f{21}{2}}{\gm}^{\f{17}{2}}\bar C}$.
	On the one hand, it is easy to check that
	\begin{equation*}
	\begin{aligned}
	\|\f{u}{\us}\|_{L^\infty_{\psi}}=\|\f{\us+\om}{\us}\|_{L^\infty_{\psi}}\leq 1+\|\f{\om}{u}\|_{L^\infty_{\psi}}\|\f{u}{\us}\|_{L^\infty_{\psi}}
	\leq1+\alpha(x)\|\f{u}{\us}\|_{L^\infty_{\psi}},
	\end{aligned}
	\end{equation*}
	which implies that
	\beq\label{uu1}
	\|\f{u}{\us}\|_{L^\infty_{\psi}}\leq \f{1}{1-\alpha(x)}.
	\deq
	On the other hand, we find
	\begin{equation}\label{uu2}
	\|\f{\us}{u}\|_{L^\infty_{\psi}}\leq \|\f{\us}{\us+\om}\|_{L^\infty_{\psi}}\leq\|\f{\us}{\us-|\om|}\|_{L^\infty_{\psi}}\leq \|\f{1}{1-|\f{\om}{\us}|}\|_{L^\infty_{\psi}}\leq\|\f{1}{1-|\f{\om}{u}||\f{u}{\us}|}\|_{L^\infty_{\psi}}\leq \f{1-\alpha(x)}{1-2\alpha(x)}.
	\end{equation}
	Then, recalling the definition of $f(x)$ and combining the estimates \eqref{al}-\eqref{uu2}, we have
	$$
	f(x)\leq \f{1-\alpha(x)}{1-2\alpha(x)}\leq \f{1}{1-2\alpha(x)}\leq \f{1}{1-2^{\f{19}{2}}{\gm}^{\f{17}{2}}\bar{C}\sg}.
	$$
	This immediately implies that
	$f(x)\leq \gm,$
	provided that $$\sg \leq \f{\gm-1}{2^{\f{19}{2}}{\gm}^{\f{19}{2}}\bar{C}}.$$
	Thus, we have completed the proof of this lemma.
\end{proof}
From Lemmas \ref{a}-\ref{b}, we may deduce that
\begin{coro}\label{ab}
Under the  assumption \eqref{priassum}, we have
	\beq\label{alphaestimate}
	\alpha(x)\leq C\e(x),
	\deq
	where $C$ is a positive constant independent of $x$.
\end{coro}
From now on, with \eqref{beta}, \eqref{alphaestimate} and our bootstrap assumption \eqref{priassum} at hand, we will use the fact that $f(x)\lesssim 1$ and $\alpha(x)\lesssim1$ in forthcoming estimates.
Next, we will devote ourselves to showing standard energy estimate, quotient estimate and other estimate in $L^2$ level, that is Lemmas \ref{A1}-\ref{A3}. The first one is $L^2$ standard energy estimate.
\begin{lemm}\label{A1}
Under the  assumption \eqref{priassum}, we have
	\begin{equation}\label{A11}
	\begin{aligned}
	\sup_{x'\in I}\|\phi\|_{L^2_\psi}^2+\int_{I}\big(\|\sqrt{u}\phi_{\psi}\|_{L^2_\psi}^2+\|\frac{\phi}{u}\|_{L^2_\psi}^2\big)dx\leq \|\phi_0\|_{L^2_\psi}^2.
	\end{aligned}
	\end{equation}
	Moreover, we have the following exponential decay estimate:
	\begin{equation}\label{Ad1}
	\sup_{x'\in I}e^{x'}\|\phi\|_{L^2_\psi}^2+\int_{I}e^{x'}\|\phi\|_{L^2_\psi}^2dx'\leq \|\phi_0\|_{L^2_\psi}^2.
	\end{equation}
\end{lemm}
\begin{proof}
	Multiplying $\eqref{eq3}_1$ by $\phi$ and integrating with respect to $\psi$, then we compute that
	\begin{equation}\label{l1}
	\frac{1}{2}\frac{d}{dx}\|\phi\|_{L^2_\psi}^2+\big(\|\sqrt{\us}\phi_{\psi}\|_{L^2_\psi}^2+\|\frac{\phi}{\sqrt{\us(u+\us)}}\|_{L^2_\psi}^2\big)-\f12\int_0^\infty\us_{\psi\psi}|\phi|^2d\psi
	\leq \int_0^\infty \big|\omega\phi\phi_{\psi\psi}\big| d\psi.
	\end{equation}
	The last term on the left handside of \eqref{l1} is positive, since\[\us_{\psi\psi}=\us_{yy}-\us|\us_{\psi}|^2\leq 0. \]
	Now returning to estimate the term on the right handside of \eqref{l1} as follows:
	\beq\label{l2}
	\int_0^\infty \big|\omega\phi\phi_{\psi\psi}\big| d\psi \leq C\|u\omega\|_{L^\infty_\psi} \|\frac{\phi}{u^{\frac{3}{2}}}\|_{L^2_\psi}\|\sqrt{u}\phi_{\psi\psi}\|_{L^2_\psi}
	\deq
	We recall the definition of $\om$ and Sobolev inequality, to find
	\begin{equation}\label{uw}
	\begin{aligned}
	\|u\omega\|_{L^\infty_\psi}\leq C\|\phi\|_{L^\infty_\psi}\leq C\|\frac{\phi}{\sqrt{u}}\|_{L^2_\psi}^{\frac{1}{2}} \|\sqrt{u}\phi_{\psi}\|_{L^2_\psi}^{\frac{1}{2}}.
	\end{aligned}
	\end{equation}
	Combining the estimates \eqref{pu31} and \eqref{beta}, we find that
	\begin{equation}\label{pu3}
	\|\f{\phi}{u^{\f32}}\|_{L^2_\psi}\leq C\big(\|\phi\|_{L^2_\psi}+\|\sqrt{u}\phi_{\psi}\|_{L^2_\psi}\big).
	\end{equation}
	Next, we note from $\eqref{u2pp1}$ and \eqref{beta} that
	\begin{equation}\label{u2pp}
	\begin{aligned}
	\|\sqrt{u}\phi_{\psi\psi}\|_{L^2_\psi}\leq  C\big(\|\phi\|_{L^2_\psi}+\|\sqrt{u}\phi_{\psi}\|_{L^2_\psi}+\|\frac{\phi_x}{\sqrt{u}}\|_{L^2_\psi}\big).
	\end{aligned}
	\end{equation}
	Inserting the estimates \eqref{l2}-\eqref{u2pp} into \eqref{l1} yields the estimate
	\begin{equation*}
	\begin{split}
	\frac{1}{2}\frac{d}{dx}\|\phi\|_{L^2_\psi}^2+\big(\|\sqrt{\us}\phi_{\psi}\|_{L^2_\psi}^2+\|\frac{\phi}{\sqrt{\us(u+\us)}}\|_{L^2_\psi}^2\big)\leq& C\big(\|\phi\|_{L^2_\psi}+\|\sqrt{u}\phi_{\psi}\|_{L^2_\psi}+\|\frac{\phi_x}{\sqrt{u}}\|_{L^2_\psi}\big)\big(\|\sqrt{u}\phi_{\psi}\|_{L^2_\psi}^2+\|\frac{\phi}{u}\|_{L^2_\psi}^2\big)\\
	\leq&C\e(x)\big(\|\sqrt{u}\phi_{\psi}\|_{L^2_\psi}^2+\|\frac{\phi}{u}\|_{L^2_\psi}^2\big).
	\end{split}
	\end{equation*}
	It follows the definition of $f(x)$ and \eqref{beta} that $u\sim\us$.
	Since $\e(x)\leq \sg$, by choosing $\sg$ small enough, then we have
	\[\frac{d}{dx}\|\phi\|_{L^2_\psi}^2+\big(\|\sqrt{u}\phi_{\psi}\|_{L^2_\psi}^2+\|\frac{\phi}{u}\|_{L^2_\psi}^2\big)\leq 0, \]
	which, integrating over $I$, yields \eqref{A11} immediately. Since $u\lesssim1$, the estimate above implies that
	\[\frac{d}{dx}\|\phi\|_{L^2_\psi}^2+\|\phi\|_{L^2_\psi}^2\leq 0. \]
	Therefore, using Gronwall inequality, we can obtain the decay estimate \eqref{Ad1}.
\end{proof}
By using the method in \cite{Iyer-2020-ARMA}, we then derive the following quotient estimate.
\begin{lemm}
Under the  assumption \eqref{priassum}, we have
	\begin{equation}\label{A22}
	\begin{aligned}
	\sup_{x'\in I}\|\frac{\phi}{\sqrt{u}}\|_{L^2_\psi}^2+\int_{I}\Big(\|\phi_{\psi}\|_{L^2_\psi}^2+\|\frac{\phi}{u^{\frac{3}{2}}}\|_{L^2_\psi}^2 \Big)dx'\leq \|\frac{\phi_0}{\sqrt{u_0}}\|_{L^2_\psi}^2.
	\end{aligned}
	\end{equation}
	Moreover, we have the following exponential decay estimate:
	\begin{equation}\label{Ad2}
	\sup_{x'\in I}e^{x'}\|\f{\phi}{\sqrt{u}}\|_{L^2_\psi}^2+\int_{I}e^{x'}\|\f{\phi}{\sqrt{u}}\|_{L^2_\psi}^2dx'\leq \|\frac{\phi_0}{\sqrt{u_0}}\|_{L^2_\psi}^2.
	\end{equation}
\end{lemm}
\begin{proof}
	We multiply $\eqref{eq3}_1$ by $\frac{\phi}{u}$ and integrate with respect to $\psi$, to discover
	\begin{equation}\label{A2}
	\begin{aligned}
	\frac{1}{2}\frac{d}{dx}\|\frac{\phi}{\sqrt{u}}\|_{L^2_\psi}^2+\Big(\|\phi_{\psi}\|_{L^2_\psi}^2+\|\frac{\phi}{\sqrt{u\us(u+\us)}}\|_{L^2_\psi}^2 \Big)\leq C\int_0^\infty \frac{\big|\phi^2\omega_x\big|}{u^2} d\psi\leq C\|u\om_x\|_{L^\infty_\psi}\|\frac{\phi}{u^{\f32}}\|_{L^2_\psi}^2.
	\end{aligned}
	\end{equation}
	Now we write
	\begin{equation}\label{wxgs}
	\om_x=\big(\f{\phi}{u+\us}\big)_x=\f{\phi_x}{u+\us}-\f{\phi\om_x}{(u+\us)^2},
	\end{equation}
	which gives the estimate
	\begin{equation}\label{uomx}
	\begin{aligned}
	\|u\om_x\|_{L^\infty_\psi}\leq C\|\phi_x\|_{L^\infty_\psi}+C\|\frac{\phi}{u^2}\|_{L^\infty_\psi}\|u\om_x\|_{L^\infty_\psi}.
	\end{aligned}
	\end{equation}
	Using Sobolev inequality, we get
	\begin{equation}\label{px}
	\|\phi_x\|_{L^\infty_\psi}\leq C\|\phi_{x}\|_{L^2_\psi}^{\f12}\|\phi_{x\psi}\|_{L^2_\psi}^{\f12}.
	\end{equation}
	Invoking \eqref{pu21} and \eqref{beta}, we obtain
	\begin{equation}\label{pu2}
	\begin{aligned}
	\|\frac{\phi}{u^2}\|_{L^\infty_\psi}\leq C\big(\|\phi\|_{L^2_\psi}+\|\phi_\psi\|_{L^2_\psi}+\|\frac{\phi_x}{\sqrt{u}}\|_{L^2_\psi}\big)\leq C\e(x).
	\end{aligned}
	\end{equation}
	We substitute the estimates \eqref{px}-\eqref{pu2} into \eqref{uomx}, to discover
	\begin{equation*}
	\begin{aligned}
	\|u\om_x\|_{L^\infty_\psi}\leq C\|\phi_{x}\|_{L^2_\psi}^{\f12}\|\phi_{x\psi}\|_{L^2_\psi}^{\f12}+ C\e(x)\|u\om_x\|_{L^\infty_\psi}\leq C\|\phi_{x}\|_{L^2_\psi}^{\f12}\|\phi_{x\psi}\|_{L^2_\psi}^{\f12}+ C\sg\|u\om_x\|_{L^\infty_\psi}.
	\end{aligned}
	\end{equation*}
	We choose $\sg$ small enough to find
	\begin{equation}\label{uomx2}
	\|u\om_x\|_{L^\infty_\psi}\leq C\|\phi_{x}\|_{L^2_\psi}^{\f12}\|\phi_{x\psi}\|_{L^2_\psi}^{\f12}\leq C\e(x).
	\end{equation}
	Utilizing \eqref{uomx2} in \eqref{A2} and using the fact that $u\sim\us$, we then obtain
	\begin{equation*}
	\frac{d}{dx}\|\frac{\phi}{\sqrt{u}}\|_{L^2_\psi}^2+\Big(\|\phi_{\psi}\|_{L^2_\psi}^2+\|\frac{\phi}{u^{\frac{3}{2}}}\|_{L^2_\psi}^2 \Big)\leq C\e(x)\|\frac{\phi}{u^{\f32}}\|_{L^2_\psi}^2.
	\end{equation*}
	Since $\e(x)\leq \sg$, by choosing $\sg$ small enough, we can deduce that
	\begin{equation*}
	\frac{d}{dx}\|\frac{\phi}{\sqrt{u}}\|_{L^2_\psi}^2+\Big(\|\phi_{\psi}\|_{L^2_\psi}^2+\|\frac{\phi}{u^{\frac{3}{2}}}\|_{L^2_\psi}^2 \Big)\leq 0,
	\end{equation*}
	which implies that \eqref{A22}. According to the fact that $u\lesssim1$, we conclude
	\begin{equation*}
	\frac{d}{dx}\|\frac{\phi}{\sqrt{u}}\|_{L^2_\psi}^2+\|\frac{\phi}{\sqrt{u}}\|_{L^2_\psi}^2\leq 0.
	\end{equation*}
	Then using Gronwall inequality, we obatin the decay estimate \eqref{Ad2} immediately.
\end{proof}
Finally, we also need the following energy estimate.
\begin{lemm}\label{A3}
Under the  assumption \eqref{priassum}, we have
	\begin{equation}\label{A33}
	\begin{aligned}
	\sup_{x'\in I}\big(\|\phi_\psi\|_{L^2_\psi}^2+\|\f{\phi}{u^{\f32}}\|_{L^2_\psi}^2\big)+\int_{I}\|\f{\phi_{x'}}{\sqrt{u}}\|_{L^2_\psi}^2dx'\leq \e^2(0)+C\e(x)\D^{2}(x),
	\end{aligned}
	\end{equation}
	where $C$ is a positive constant independent of $x$.
\end{lemm}
\begin{proof}
	Multiplying $\eqref{eq3}_1$ by $\f{\phi_x}{u}$, using equivalent relation $u\sim\us$, and integrating by parts, we deduce
	\begin{equation}\label{otherestimate}
	\begin{aligned}
	\f12\f{d}{dx}\big(\|\phi_\psi\|_{L^2_\psi}^2+\|\f{\phi}{\sqrt{u\us(u+\us)}}\|_{L^2_\psi}^2\big)+\|\f{\phi_{x}}{\sqrt{u}}\|_{L^2_\psi}^2\leq C\int_0^\infty \f{|\phi|^2|\om_x|}{u^4}d\psi\leq C\|\f{\om_x}{u^{\f32}}\|_{L^2_\psi}\|\f{\phi}{u^{\f52}}\|_{L^2_\psi}\|\phi\|_{L^\infty_\psi}.
	\end{aligned}
	\end{equation}
	Recalling the definition of $\e(x)$, then combinining \eqref{wxgs} with \eqref{pu2}, we get
	\begin{equation*}
	\begin{aligned}
	\|\om_x\|_{L^\infty_\psi}\leq \|\f{\phi_x}{u}\|_{L^\infty_\psi}+\|\f{\phi}{u^2}\|_{L^\infty_\psi}\|\om_x\|_{L^\infty_\psi}\leq\|\f{\phi_x}{u}\|_{L^\infty_\psi}+ C\e(x)\|\om_x\|_{L^\infty_\psi}.
	\end{aligned}
	\end{equation*}
	Recalling assumption \eqref{priassum} that $\e(x)\leq \sg$, choosing $\sg$ small enough, and using Sobolev inequality, we then have
	\begin{equation}\label{omxwq}
	\begin{aligned}
	\|\om_x\|_{L^\infty_\psi}^2\leq&C \|\f{\phi_x}{u}\|_{L^\infty_\psi}^2\leq C\big(\|\f{\phi_x(x,\psi)}{u(x,\psi)}(1-\chi(\f{\psi}{\eta_0}))\|_{L^\infty_\psi}^2
+\|\f{\phi_x(x,\psi)}{u(x,\psi)}\chi(\f{\psi}{\eta_0})\|_{L^\infty_\psi}^2\big)\\
	\leq&C\big(\|\phi_x\|_{L^\infty_\psi}^2
+\|\psi^{\f12}\phi_{x\psi}(x,\psi)\chi(\f{\psi}{\eta_0})\|_{L^\infty_\psi}^2\big)\\ \leq&C\|\phi_{x}\|_{L^2_\psi}\|\phi_{x\psi}\|_{L^2_\psi}
+C\int_\psi^\infty(|\phi_{x\psi'}(x,\psi')|^2
+\psi'\phi_{x\psi'\psi'}(x,\psi')\phi_{x\psi'}(x,\psi')\chi(\f{\psi'}{\eta_0}))d\psi'\\ \leq&C\big(\|\phi_{x}\|_{L^2_\psi}^2+\|\phi_{x\psi}\|_{L^2_\psi}^2
+\|\sqrt{u}\phi_{x\psi}\|_{L^2_\psi}\|u^{\f32}\phi_{x\psi\psi}\|_{L^2_\psi}\big),
	\end{aligned}
	\end{equation}
where we have used that
\beqq
\begin{split}	
|\f{\phi_x(x,\psi)}{u(x,\psi)}\chi(\f{\psi}{\eta_0})|
\leq&\f{C}{\psi^{\f12}}\int_0^{\min\{\psi,2\eta_0\}} |\phi_{x\psi'}(x,\psi')| d\psi'
\leq\f{C}{\psi^{\f12}}\int_0^\psi |\phi_{x\psi'}(x,\psi')\chi(\f{\psi'}{\eta_0})| d\psi'\\
		\leq& \f{C}{\psi^{\f12}}\|\psi^{\f12}\phi_{x\psi}(x,\psi)\chi(\f{\psi}{\eta_0})\|_{L^\infty_\psi}
\int_0^\psi\psi'^{-\f12}d\psi'
\leq C\|\psi^{\f12}\phi_{x\psi}(x,\psi)\chi(\f{\psi}{\eta_0})\|_{L^\infty_\psi}.
\end{split}
\deqq
	Next differentiate $\eqref{eq3}_1$ with respect to $x$, then
	\begin{equation}\label{eq4}
	\begin{aligned}
	\phi_{xx}-u\phi_{x\psi\psi}+2\f{\phi_x}{\us(u+\us)}-\om_x\phi_{\psi\psi}-2\f{\phi\om_x}{\us(u+\us)^2}=0.
	\end{aligned}
	\end{equation}
	Combining this equation, \eqref{u521}, \eqref{beta}, \eqref{u2pp} with \eqref{omxwq}, for any small $\eps$ ($\eps$ will be determined later), it holds
	\begin{equation*}
	\begin{aligned}
	\|u^{\f32}\phi_{x\psi\psi}\|_{L^2_\psi}\leq&C\big(\|\sqrt{u}\phi_{xx}\|_{L^2_\psi}+\|\f{\phi_x}{u^{\f32}}\|_{L^2_\psi}+\|\sqrt{u}\phi_{\psi\psi}\|_{L^2_\psi}\|\om_x\|_{L^\infty_\psi}+\|\f{\phi}{u^{\f52}}\|_{L^2_\psi}\|\om_x\|_{L^\infty_\psi}\big)\\
	\leq&C\Big(\|\sqrt{u}\phi_{xx}\|_{L^2_\psi}+\|\f{\phi_x}{u^{\f32}}\|_{L^2_\psi}+\big(\|\phi\|_{L^2_\psi}+\|\sqrt{u}\phi_{\psi}\|_{L^2_\psi}+\|\frac{\phi_x}{\sqrt{u}}\|_{L^2_\psi}\big)\\
	&\quad\times\big(\|\phi_{x}\|_{L^2_\psi}+\|\phi_{x\psi}\|_{L^2_\psi}+\|\sqrt{u}\phi_{x\psi}\|_{L^2_\psi}^{\f12}\|u^{\f32}\phi_{x\psi\psi}\|_{L^2_\psi}^{\f12}\big)\Big)\\
	\leq&\eps\|u^{\f32}\phi_{x\psi\psi}\|_{L^2_\psi}+C_{\eps}\big(\|\phi\|_{L^2_\psi}+\|\sqrt{u}\phi_{\psi}\|_{L^2_\psi}+\|\f{\phi_x}{u^{\f32}}\|_{L^2_\psi}+\|\phi_{x\psi}\|_{L^2_\psi}+\|\sqrt{u}\phi_{xx}\|_{L^2_\psi}\big),
	\end{aligned}
	\end{equation*}
	where we have used the fact that $u\lesssim1$ and\beqq
	\|\phi\|_{L^2_\psi}+\|\sqrt{u}\phi_{\psi}\|_{L^2_\psi}
    +\|\frac{\phi_x}{\sqrt{u}}\|_{L^2_\psi}\leq C\e(x)\leq C\sg.
	\deqq
	We then choose $\eps$ small enough, to find
	\begin{equation}\label{pxpp}
	\|u^{\f32}\phi_{x\psi\psi}\|_{L^2_\psi}\leq C\big(\|\phi\|_{L^2_\psi}^2+\|\sqrt{u}\phi_{\psi}\|_{L^2_\psi}^2+\|\f{\phi_x}{u^{\f32}}\|_{L^2_\psi}+\|\phi_{x\psi}\|_{L^2_\psi}^2+\|\sqrt{u}\phi_{xx}\|_{L^2_\psi}\big).
	\end{equation}
	Recalling \eqref{wxgs} and \eqref{pu2}, we deduce
	\begin{equation}\label{omx32}
	\begin{aligned}
	\|\f{\om_x}{u^{\f32}}\|_{L^2_\psi}\leq& \|\f{\phi_x}{u^{\f52}}\|_{L^2_\psi}+\|\f{\phi}{u^2}\|_{L^\infty_\psi}\|\f{\om_x}{u^{\f32}}\|_{L^2_\psi}
	\leq\|\f{\phi_x}{u^{\f52}}\|_{L^2_\psi}+C\e(x)\|\f{\om_x}{u^{\f32}}\|_{L^2_\psi}.
\end{aligned}
\end{equation}
It then follows from a similar way as \eqref{u521} that
\begin{equation}\label{px52}
\begin{aligned} \|\f{\phi_x}{u^{\f52}}\|_{L^2_\psi}
\leq&\|\f{\phi_x(x,\psi)}{u^{\f52}(x,\psi)}(1-\chi(\f{\psi}{\eta_0}))\|_{L^2_\psi}
+\|\f{\phi_x(x,\psi)}{\psi^{\f54}}\chi(\f{\psi}{\eta_0})\|_{L^2_\psi}\\
\leq& C\big(\|\phi_x\|_{L^2_\psi}+\|\f{\phi_x(x,\psi)}{\psi^{\f14}}\chi'(\f{\psi}{\eta_0})\|_{L^2_\psi}
+\|\f{\phi_{x\psi}(x,\psi)}{\psi^{\f14}}\chi(\f{\psi}{\eta_0})\|_{L^2_\psi}\big)\\
\leq& C\big(\|\phi_x\|_{L^2_\psi}+\|\phi_x(x,\psi)\psi^{\f34}\chi''(\f{\psi}{\eta_0})\|_{L^2_\psi}
+\|\phi_{x\psi}(x,\psi)\psi^{\f34}\chi'(\f{\psi}{\eta_0})\|_{L^2_\psi}
+\|\phi_{x\psi\psi}(x,\psi)\psi^{\f34}\chi(\f{\psi}{\eta_0})\|_{L^2_\psi}\big)\\ \leq&C\big(\|\phi\|_{L^2_\psi}+\|u^{\f32}\phi_{x\psi}\|_{L^2_\psi}
+\|u^{\f32}\phi_{x\psi\psi}\|_{L^2_\psi}\big)\\ \leq&C\big(\|\phi\|_{L^2_\psi}+\|\sqrt{u}\phi_{\psi}\|_{L^2_\psi}
+\|\f{\phi_x}{u^{\f32}}\|_{L^2_\psi}+\|\phi_{x\psi}\|_{L^2_\psi}
+\|\sqrt{u}\phi_{xx}\|_{L^2_\psi} \big),
	\end{aligned}
	\end{equation}
	where we have used the estimate \eqref{pxpp} in the last inequality.
	Therefore, combining \eqref{px52} with \eqref{omx32}, and choosing $\sg$ small enough, we discover
	\begin{equation}\label{omx3}
	\begin{aligned}
	\|\f{\om_x}{u^{\f32}}\|_{L^2_\psi}\leq C\big(\|\phi\|_{L^2_\psi}+\|\sqrt{u}\phi_{\psi}\|_{L^2_\psi}
+\|\f{\phi_x}{u^{\f32}}\|_{L^2_\psi}+\|\phi_{x\psi}\|_{L^2_\psi}
+\|\sqrt{u}\phi_{xx}\|_{L^2_\psi} \big).
	\end{aligned}
	\end{equation}
	Using Sobolev inequality, we have\beq\label{phiinfty}
	\|\phi\|_{L^\infty_\psi}\leq \|\phi\|_{L^2_\psi}^{\f12}\|\phi_\psi\|_{L^2_\psi}^{\f12}.
	\deq
	We thereupon conclude from  \eqref{u521}, \eqref{beta}, \eqref{omx3} and \eqref{phiinfty} that
	\begin{equation*}
	\begin{aligned}
	\int_0^\infty \f{|\phi|^2|\om_x|}{u^4}d\psi
	\leq&C\big(\|\phi\|_{L^2_\psi}+\|\sqrt{u}\phi_{\psi}\|_{L^2_\psi}
+\|\f{\phi_x}{u^{\f32}}\|_{L^2_\psi}+\|\phi_{x\psi}\|_{L^2_\psi}
+\|\sqrt{u}\phi_{xx}\|_{L^2_\psi} \big)\\	&\quad\times\big(\|\phi\|_{L^2_\psi}+\|\sqrt{u}\phi_{\psi}\|_{L^2_\psi}
+\|\phi_x\|_{L^2_\psi}\big)\big(\|\phi\|_{L^2_\psi}+\|\phi_{\psi}\|_{L^2_\psi}\big).
	\end{aligned}
	\end{equation*}
	Substituting the estimate above into \eqref{otherestimate}, and then integrating over $I$, by using equivalent relation $u\sim\us$ and recalling the definition of $\e(x)$ and $\D(x)$, we obtain the estimate \eqref{A33} immediately.
\end{proof}
At the $H^1$ level, comparing equation \eqref{eq3} versus equation \eqref{eq4}, the nonlinear terms changes and thus requires a new treatment, see Lemmas \ref{l4}-\ref{l6} below.
Firstly, we derive the following standard energy estimate in $H^1$ level.
\begin{lemm}\label{l4}
Under the  assumption \eqref{priassum}, we have
	\begin{equation}\label{l44}
	\sup_{x'\in I}\|\phi_{x'}\|_{L^2_\psi}^2+\int_{I}\big(\|\sqrt{u}\phi_{x'\psi}\|_{L^2_\psi}^2+\|\frac{\phi_{x'}}{u}\|_{L^2_\psi}^2 \big)dx'\leq \e^2(0)+C\e(x)\D^{2}(x),
	\end{equation}
	where $C$ is a positive constant independent of $x$.
\end{lemm}
\begin{proof}
	Multiplying \eqref{eq4} by $\phi_x$ and integrating with respect to $\psi$, we deduce that
	\begin{equation}\label{h1energyestim}
	\begin{aligned}
	&\f12\f{d}{dx}\|\phi_{x}\|_{L^2_\psi}^2+\|\sqrt{u}\phi_{x\psi}\|_{L^2_\psi}^2+\|\f{\phi_{x}}{\sqrt{\us(u+\us)}}\|_{L^2_\psi}^2 -\f12\int_0^\infty\us_{\psi\psi}|\phi_x|^2d\psi\\
	\leq& \int_0^\infty\Big(\big|\om_x\phi_x\phi_{\psi\psi}\big|+\f{\big|\phi\phi_x\om_x\big|}{u^3}+\big|\om_{\psi}\phi_{x}\phi_{x\psi}\big|\Big)d\psi.
	\end{aligned}
	\end{equation}
	Since $\us_{\psi\psi}\leq0$, we have\beqq
	\f12\int_0^\infty\us_{\psi\psi}|\phi_x|^2d\psi\leq 0.
	\deqq
	We now proceed to estimate three terms on the right handside of \eqref{h1energyestim}.
	According to \eqref{u2pp} and \eqref{uomx2}, we find
	\begin{equation}\label{firterm}
	\begin{aligned}
	\int_0^\infty\big|\om_x\phi_x\phi_{\psi\psi}\big|d\psi\leq& C\|\sqrt{u}\phi_{\psi\psi}\|_{L^2_\psi}\|\f{\phi_{x}}{u^{\f32}}\|_{L^2_\psi}\|u\om_x\|_{L^\infty_\psi}\\
	\leq& C\big(\|\phi\|_{L^2_\psi}+\|\sqrt{u}\phi_{\psi}\|_{L^2_\psi}+\|\frac{\phi_x}{\sqrt{u}}\|_{L^2_\psi}\big)\|\f{\phi_{x}}{u^{\f32}}\|_{L^2_\psi}\|\phi_{x}\|_{L^2_\psi}^{\f12}\|\phi_{x\psi}\|_{L^2_\psi}^{\f12}.
	\end{aligned}
	\end{equation}
	Next, we deduce from \eqref{u521}, \eqref{uomx2} and \eqref{px52} that
	\begin{equation}\label{secterm}
	\begin{aligned}
	\int_0^\infty\f{\big|\phi\phi_x\om_x\big|}{u^3}d\psi\leq& C\|\f{\phi}{u^{\f52}}\|_{L^2_\psi}\|\f{\phi_x}{u^{\f52}}\|_{L^2_\psi}\|u\om_x\|_{L^\infty_\psi}\\
	\leq&C\big(\|\phi\|_{L^2_\psi}+\|\sqrt{u}\phi_{\psi}\|_{L^2_\psi}+\|\phi_x\|_{L^2_\psi}\big)\big(\|\phi\|_{L^2_\psi}+\|\sqrt{u}\phi_{\psi}\|_{L^2_\psi}+\|\f{\phi_{x}}{u^{\f32}}\|_{L^2_\psi}+\|\phi_{x\psi}\|_{L^2_\psi}\\
	&\quad+\|\sqrt{u}\phi_{xx}\|_{L^2_\psi}\big)\|\phi_{x}\|_{L^2_\psi}^{\f12}\|\phi_{x\psi}\|_{L^2_\psi}^{\f12}.
	\end{aligned}
	\end{equation}
	In order to estimate the last term on the right handside of \eqref{h1energyestim}, we first control the following term by using a similar way as \eqref{phiu2} and \eqref{phiu2chi}
	\begin{equation}\label{phiu52}
	\begin{aligned}
	\|\f{\phi}{u^{\f52}}\|_{L^\infty_\psi}\leq C\big(\|\phi\|_{L^\infty_\psi}+ \|\f{\phi_\psi(x,\psi)}{\sqrt{u(x,\psi)}}\chi(\f{\psi}{\eta_0})\|_{L^\infty_\psi}\big).
	\end{aligned}
	\end{equation}
	It then follows from a similar way as \eqref{pu3} that
	\begin{equation}\label{pup3}
	\|\f{\phi_\psi}{u^{\f32}}\|_{L^2_\psi}\leq C\big(\|\phi_\psi\|_{L^2_\psi}+\|\sqrt{u}\phi_{\psi\psi}\|_{L^2_\psi}\big).
	\end{equation}
	For the enhanced localized $\f{\phi_\psi}{\sqrt{u}}$ estimate, we deduce from \eqref{u2pp} and \eqref{pup3} that
	\begin{equation}\label{pp52}
	\begin{aligned}
	&\|\f{\phi_{\psi}(x,\psi)}{\sqrt{u(x,\psi)}}\chi(\f{\psi}{\eta_0})\|_{L^\infty_\psi}^2
	\leq \|\f{\phi_{\psi}(x,\psi)}{\psi^{\f14}}\chi(\f{\psi}{\eta_0})\|_{L^\infty_\psi}^2\\
	 \leq& C \int_{\psi}^{\infty}\Big|\f{\phi_{\psi'}(x,\psi')}{\psi'^{\f14}}\chi(\f{\psi'}{\eta_0}) \big(\f{\phi_{\psi'}(x,\psi')}{\psi'^{\f14}}\chi(\f{\psi'}{\eta_0})\big)_{\psi'}\Big| d{\psi'}\\
	=&C\int_\psi^\infty\Big|\f{\phi_{\psi'}(x,\psi')}{\psi'^{\f14}}\chi(\f{\psi'}{\eta_0}) \big(\f{\phi_{\psi'\psi'}(x,\psi')}{\psi'^{\f14}}\chi(\f{\psi'}{\eta_0})
+\f{\phi_{\psi'}(x,\psi')}{\psi'^{\f54}}\chi(\f{\psi'}{\eta_0})
+\f{\phi_{\psi'}(x,\psi')}{\psi'^{\f14}}\chi'(\f{\psi'}{\eta_0})\big)\Big|d\psi'\\ \leq&C\big(\|\f{\phi_{\psi}}{u^{\f32}}\|_{L^2_\psi}^2
+\|\sqrt{u}\phi_{\psi\psi}\|_{L^2_\psi}^2\big)
\leq C\big(\|\phi\|_{L^2_\psi}^2+\|\phi_{\psi}\|_{L^2_\psi}^2
+\|\frac{\phi_x}{\sqrt{u}}\|_{L^2_\psi}^2\big).
	\end{aligned}
	\end{equation}
	Substituting \eqref{phiinfty} and \eqref{pp52} into \eqref{phiu52}, we have
	\begin{equation}\label{ph52in}
	\begin{aligned}
	\|\f{\phi}{u^{\f52}}\|_{L^\infty_\psi}\leq C\big(\|\phi\|_{L^2_\psi}+\|\phi_{\psi}\|_{L^2_\psi}+\|\frac{\phi_x}{\sqrt{u}}\|_{L^2_\psi}\big).
	\end{aligned}
	\end{equation}
	We then estimate $\|\sqrt{u}\om_{\psi}\|_{L^\infty_\psi}$. Owing to
	\[2u\om_\psi=\phi_\psi-2\f{\us_y}{\us}\om, \]
	we then conclude from \eqref{pinty1}, \eqref{beta}, \eqref{pp52} and \eqref{ph52in} that
	\begin{equation}\label{uomp}
	\begin{aligned}
	\|\sqrt{u}\om_{\psi}\|_{L^\infty_\psi}\leq& C\big(\|\f{\phi_{\psi}}{\sqrt{u}}\|_{L^\infty_\psi}
+\|\us_y\|_{L^\infty_\psi}\|\f{\om}{u^{\f32}}\|_{L^\infty_\psi}\big)\\
	\leq& C\big(\|\f{\phi_{\psi}(x,\psi)}{\sqrt{u(x,\psi)}}(1-\chi(\f{\psi}{\eta_0}))\|_{L^\infty_\psi}
+ \|\f{\phi_{\psi}(x,\psi)}{\sqrt{u(x,\psi)}}\chi(\f{\psi}{\eta_0})\|_{L^\infty_\psi}
+\|\f{\phi}{u^{\f52}}\|_{L^\infty_\psi}\big)\\ \leq&C\big(\|\phi_{\psi}\|_{L^\infty_\psi}
+\|\f{\phi_{\psi}(x,\psi)}{\psi^{\f14}}\chi(\f{\psi}{\eta_0})\|_{L^\infty_\psi}
+\|\f{\phi}{u^{\f52}}\|_{L^\infty_\psi}\big)\\
	\leq& C\big(\|\phi\|_{L^2_\psi}+\|\phi_{\psi}\|_{L^2_\psi}
+\|\frac{\phi_x}{\sqrt{u}}\|_{L^2_\psi}\big).
	\end{aligned}
	\end{equation}
	Invoking Sobolev inequality and \eqref{uomp}, we obtain
	\begin{equation*}
	\begin{aligned}
	\int_0^\infty\big|\om_{\psi}\phi_{x}\phi_{x\psi}\big|d\psi\leq& C\|\sqrt{u}\om_{\psi}\|_{L^\infty_\psi}\|\f{\phi_{x}}{u}\|_{L^2_\psi}\|\sqrt{u}\phi_{x\psi}\|_{L^2_\psi}\\
	\leq&C\big(\|\phi\|_{L^2_\psi}+\|\phi_{\psi}\|_{L^2_\psi}+\|\f{\phi_{x}}{\sqrt{u}}\|_{L^2_\psi}\big)\|\f{\phi_{x}}{u}\|_{L^2_\psi}\|\sqrt{u}\phi_{x\psi}\|_{L^2_\psi}.
	\end{aligned}
	\end{equation*}
	This, combining \eqref{h1energyestim}-\eqref{secterm}, leads to the estimate \eqref{l44} by using equivalent relation $u\sim\us$ and recalling the definition of $\e(x)$ and $\D(x)$.
	Thus, we have completed the proof of this lemma.
\end{proof}
Next, we can also obtain the following quotient estimate in $H^1$ level.
\begin{lemm}\label{lastlemm}
	Under the  assumption \eqref{priassum}, we have
	\begin{equation}\label{last1}
	\sup_{x'\in I}\|\f{\phi_{x'}}{\sqrt{u}}\|_{L^2_\psi}^2+\int_{I}\big(\|\phi_{x'\psi}\|_{L^2_\psi}^2+\|\frac{\phi_{x'}}{u^{\f32}}\|_{L^2_\psi}^2 \big)dx'\leq \e^2(0)+C\e(x)\D^{2}(x),
	\end{equation}
	where $C$ is a positive constant independent of $x$.
\end{lemm}
\begin{proof}
	Multiplying \eqref{eq4} by $\f{\phi_x}{u}$ and integrating with respect to $\psi$, we find that
	\begin{equation*}
	\begin{aligned}
	\f12\f{d}{dx}\|\f{\phi_x}{\sqrt{u}}\|_{L^2_\psi}^2+\|\phi_{x\psi}\|_{L^2_\psi}^2+\|\frac{\phi_{x}}{\sqrt{u\us(u+\us)}}\|_{L^2_\psi}^2\leq C\int_0^\infty\Big(\f{\big|\om_x\phi_x\phi_{\psi\psi}\big|}{u}+\f{\big|\phi\phi_x\om_x\big|}{u^4}+\f{\big|\phi_x^2\om_x\big|}{u^2}\Big) d\psi.
	\end{aligned}
	\end{equation*}
	Utilizing \eqref{pxpp} into \eqref{omxwq}, we then calculate that
	\begin{equation}\label{omx}
	\begin{aligned}
	\|\om_x\|_{L^\infty_\psi}\leq C\big(\|\phi\|_{L^2_\psi}+\|\phi_{\psi}\|_{L^2_\psi}+\|\f{\phi_x}{u^{\f32}}\|_{L^2_\psi}+\|\phi_{x\psi}\|_{L^2_\psi}+\|\sqrt{u}\phi_{xx}\|_{L^2_\psi} \big).
	\end{aligned}
	\end{equation}
	Hence, we get from \eqref{u2pp} and \eqref{omx} that
	\begin{equation*}
	\begin{aligned}
	\int_0^\infty\f{\big|\om_x\phi_x\phi_{\psi\psi}\big|}{u}d\psi\leq&\|\om_x\|_{L^\infty_\psi}\|\f{\phi_x}{u^{\f32}}\|_{L^2_\psi}\|\sqrt{u}\phi_{\psi\psi}\|_{L^2_\psi}\\
	\leq&C\big(\|\phi\|_{L^2_\psi}+\|\phi_{\psi}\|_{L^2_\psi}+\|\f{\phi_x}{u^{\f32}}\|_{L^2_\psi}+\|\phi_{x\psi}\|_{L^2_\psi}+\|\sqrt{u}\phi_{xx}\|_{L^2_\psi} \big)\|\f{\phi_x}{u^{\f32}}\|_{L^2_\psi}\\
	&\quad\times\big(\|\phi\|_{L^2_\psi}+\|\sqrt{u}\phi_{\psi}\|_{L^2_\psi}+\|\frac{\phi_x}{\sqrt{u}}\|_{L^2_\psi}\big).
	\end{aligned}
	\end{equation*}
	Using \eqref{omx} once again, together with \eqref{u521} and \eqref{beta}, we compute
	\begin{equation*}
	\begin{aligned}
	\int_0^\infty\f{\big|\phi\phi_x\om_x\big|}{u^4}d\psi\leq&\|\om_x\|_{L^\infty_\psi}\|\f{\phi}{u^{\f52}}\|_{L^2_\psi}\|\f{\phi_x}{u^{\f32}}\|_{L^2_\psi}\\
	\leq&C\big(\|\phi\|_{L^2_\psi}+\|\phi_{\psi}\|_{L^2_\psi}+\|\f{\phi_x}{u^{\f32}}\|_{L^2_\psi}+\|\phi_{x\psi}\|_{L^2_\psi}+\|\sqrt{u}\phi_{xx}\|_{L^2_\psi} \big)\\
	&\quad\times\big(\|\phi\|_{L^2_\psi}+\|\sqrt{u}\phi_{\psi}\|_{L^2_\psi}+\|\phi_x\|_{L^2_\psi}\big)\|\f{\phi_x}{u^{\f32}}\|_{L^2_\psi}.
	\end{aligned}
	\end{equation*}
	Applying \eqref{omx} once more, we discover that
	\begin{equation*}
	\begin{aligned}
	\int_0^\infty\f{\big|\phi_x^2\om_x\big|}{u^2}d\psi\leq& \|\om_x\|_{L^\infty_\psi}\|\f{\phi_x}{u^{\f32}}\|_{L^2_\psi}\|\f{\phi_x}{\sqrt{u}}\|_{L^2_\psi}
	\leq C\big(\|\phi\|_{L^2_\psi}+\|\phi_{\psi}\|_{L^2_\psi}+\|\f{\phi_x}{u^{\f32}}\|_{L^2_\psi}+\|\phi_{x\psi}\|_{L^2_\psi}+\|\sqrt{u}\phi_{xx}\|_{L^2_\psi} \big)\|\f{\phi_x}{u^{\f32}}\|_{L^2_\psi}^2.
	\end{aligned}
	\end{equation*}
	Integrating and utilizing these estimates above, using the equivalent relation $u\sim\us$, then recalling the definition of $\e(x)$ and $\D(x)$, we deduce \eqref{last1} directly.
\end{proof}
\begin{lemm}\label{l6}
	Under the  assumption \eqref{priassum}, we have
	\begin{equation}\label{l66}
	\sup_{x'\in I}\big(\|\phi_{x'\psi}\|_{L^2_\psi}^2+\|\frac{\phi_{x'}}{u^{\f32}}\|_{L^2_\psi}^2\big)+\int_{I}\|\f{\phi_{x'x'}}{\sqrt{u}}\|_{L^2_\psi}^2 dx'\leq \e^2(0)+C\e(x)\D^{2}(x),
	\end{equation}
	where $C$ is a positive constant independent of $x$.
\end{lemm}
\begin{proof}
	Multiplying \eqref{eq4} by $\f{\phi_{xx}}{u}$ and integrating with respect to $\psi$, we find that
	\begin{equation}\label{l666}
	\begin{aligned}
	\f12\f{d}{dx}\big(\|\phi_{x\psi}\|_{L^2_\psi}^2+\|\f{\phi_{x}}{\sqrt{u\us(u+\us)}}\|_{L^2_\psi}^2\big)+\|\f{\phi_{xx}}{\sqrt{u}}\|_{L^2_\psi}^2\leq C\int_0^\infty\Big(\f{\big|\om_x\phi_{xx}\phi_{\psi\psi}\big|}{u}+\f{\big|\phi\phi_{xx}\om_x\big|}{u^4}+\f{\big|\phi_x^2\om_x\big|}{u^4}\Big) d\psi.
	\end{aligned}
	\end{equation}
	Recalling now the equation $\eqref{eq3}_1$, the estimates \eqref{px} and \eqref{pu2} yields
	\begin{equation}\label{uppwq}
	\begin{aligned}
	\|u\phi_{\psi\psi}\|_{L^\infty_\psi}\leq C\big(\|\phi_{x}\|_{L^\infty_\psi}+\|\f{\phi}{u^{\f32}}\|_{L^\infty_\psi}\big)\leq C\big(\|\phi\|_{L^2_\psi}+\|\phi_\psi\|_{L^2_\psi}+\|\frac{\phi_x}{\sqrt{u}}\|_{L^2_\psi}+\|\phi_{x\psi}\|_{L^2_\psi}\big).
	\end{aligned}
	\end{equation}
	We can calculate from \eqref{omx3} and \eqref{uppwq} that
	\begin{equation*}
	\begin{aligned}
	\int_0^\infty\f{\big|\om_x\phi_{xx}\phi_{\psi\psi}\big|}{u}d\psi\leq&C\|\f{\om_x}{u^{\f32}}\|_{L^2_\psi}\|\f{\phi_{xx}}{\sqrt{u}}\|_{L^2_\psi}\|u\phi_{\psi\psi}\|_{L^\infty_\psi}\\
	\leq&C\big(\|\phi\|_{L^2_\psi}+\|\sqrt{u}\phi_{\psi}\|_{L^2_\psi}+\|\f{\phi_x}{u^{\f32}}\|_{L^2_\psi}+\|\phi_{x\psi}\|_{L^2_\psi}+\|\sqrt{u}\phi_{xx}\|_{L^2_\psi} \big)\|\f{\phi_{xx}}{\sqrt{u}}\|_{L^2_\psi}\\
	&\quad\times\big(\|\phi\|_{L^2_\psi}+\|\phi_\psi\|_{L^2_\psi}+\|\frac{\phi_x}{\sqrt{u}}\|_{L^2_\psi}+\|\phi_{x\psi}\|_{L^2_\psi}\big).
	\end{aligned}
	\end{equation*}
	Then, utilizing \eqref{pu2} with \eqref{omx3}, we deduce
	\begin{equation*}
	\begin{aligned}
	\int_0^\infty\f{\big|\phi\phi_{xx}\om_x\big|}{u^4}d\psi\leq&C\|\f{\om_x}{u^{\f32}}\|_{L^2_\psi}\|\f{\phi_{xx}}{\sqrt{u}}\|_{L^2_\psi}\|\f{\phi}{u^2}\|_{L^\infty_\psi}\\
	\leq&C\big(\|\phi\|_{L^2_\psi}+\|\sqrt{u}\phi_{\psi}\|_{L^2_\psi}+\|\f{\phi_x}{u^{\f32}}\|_{L^2_\psi}+\|\phi_{x\psi}\|_{L^2_\psi}+\|\sqrt{u}\phi_{xx}\|_{L^2_\psi} \big)\|\f{\phi_{xx}}{\sqrt{u}}\|_{L^2_\psi}\\
	&\quad\times\big(\|\phi\|_{L^2_\psi}+\|\phi_\psi\|_{L^2_\psi}+\|\frac{\phi_x}{\sqrt{u}}\|_{L^2_\psi}\big).
	\end{aligned}
	\end{equation*}
	Making use of \eqref{px}, \eqref{px52} and \eqref{omx3}, we discover
	\begin{equation*}
	\begin{aligned}
	\int_0^\infty\f{\big|\phi_x^2\om_x\big|}{u^4}d\psi
	\leq&C\|\f{\om_x}{u^{\f32}}\|_{L^2_\psi}\|\f{\phi_{x}}{u^{\f52}}\|_{L^2_\psi}\|\phi_{x}\|_{L^\infty_\psi}\\
	\leq&C\big(\|\phi\|_{L^2_\psi}^2+\|\sqrt{u}\phi_{\psi}\|_{L^2_\psi}^2+\|\f{\phi_x}{u^{\f32}}\|_{L^2_\psi}^2+\|\phi_{x\psi}\|_{L^2_\psi}^2+\|\sqrt{u}\phi_{xx}\|_{L^2_\psi}^2 \big)\|\phi_{x}\|_{L^2_\psi}^{\f12}\|\phi_{x\psi}\|_{L^2_\psi}^{\f12}.
	\end{aligned}
	\end{equation*}
	Pluging these estimates above into the estimate \eqref{l666}, recalling the fact that $u\sim\bar u$ and the definition of $\e(x)$ and $\D(x)$, it is easy to deduce the estimate \eqref{l66}.
\end{proof}

Under the a priori assumption \eqref{priassum}, we collect all the estimates in
Lemmas \ref{A1}-\ref{l6}, to get that
	\begin{equation*}
	\D(x)\leq \e(0) +C^*\e^{\f12}(x)\D(x),
	\end{equation*}
	which, together with the a priori assumption \eqref{priassum}, leads to
	\[\D(x)\leq \e(0)+C^*\sg^{\f12}\D(x). \]
	If we take $\sg\leq \min\Big{\{}\f{1}{4(C^*)^2},\f{1}{2^{\f{21}{2}}{\gm}^{\f{17}{2}}\bar C},\f{\gm-1}{2^{\f{19}{2}}{\gm}^{\f{19}{2}}\bar{C}}\Big{\}}$, then we have
\begin{equation}\label{d1}
\e(x)\le \D(x)\leq 2\e(0).
\end{equation}
Taking $\gamma:=2f(0)$ and $\sigma:=4\e(0)$, one can apply the a priori estimate
\beq\label{priassum-01}
f(x)\leq 4f(0), \quad \e(x)\leq 4\e(0)
\deq
to obtain the following estimate
\begin{equation}\label{d1}
f(x)\leq 2f(0), \quad \e(x)\leq 2\e(0).
\end{equation}
Here we require the initial data is small enough to satisfy the condition
$$\e(0)\le \sigma_0=
\frac14  \min\Big{\{}\f{1}{4(C^*)^2},\f{1}{2^{\f{21}{2}}{\gamma}^{\f{17}{2}}\bar C},\f{\gamma-1}{2^{\f{19}{2}}{\gamma}^{\f{19}{2}}\bar{C}}\Big{\}}.$$
Therefore, we have established the closed estimate in this subsection.

\subsection{Global existence and exponential decay in von-Mises variable}\label{global}
In this section, our first aim is to establish the global existence of $\phi$. Next, we will use Lemmas \ref{A1}-\ref{l6} to show that the difference $u-\us$, as function of $(x,\psi)$, decays exponentially in $x$ variable, which proves the decay estimates \eqref{decay} in Proposition \ref{them1}.
This, together the proof of the global existence result, gives the proof of Proposition \ref{them1}.
\begin{proof}[\textbf{Proof of Proposition \ref{them1}}]
Suppose the assumptions in Proposition \ref{them1} hold on, it is easy to obtain the local existence of the solution $\phi$.
Next, we use standard continuity argument to show the global well-posedness.
From the local existence result in Proposition  \ref{local-result} and
initial conditions in Proposition \ref{them1}, it holds on
\[f(x)\leq 4f(0)~~\text{and}~~\e(x)\leq 4\e(0),~~\forall x\in[0,x_1).  \]
Set
\[x^*:= \sup_{x_1}\big\{x_1|f(x)\leq 4f(0)~~\text{and}~~\e(x)\leq 4\e(0),
~ \forall x\in[0,x_1)\big\}. \]
We claim that $x^*=+\infty$. Otherwise, applying the estimate \eqref{d1} and
the local existence result in Proposition \ref{local-result},
there exists a positive constant $x^{**}$, such that $x^{**}>x^*$,
it holds that for any $x^*<x<x^{**}$,
\[f(x)\leq 4f(0)~~\text{and}~~\e(x)\leq 4\e(0),~~\forall x\in[0,x_1).  \]
This contradicts the definition of $x^*$. Thus, we complete the proof of the global result in Proposition \ref{them1}.	
	Then we conclude from the estimates of Lemmas \ref{A1}-\ref{l6} that
	\begin{equation*}
	\begin{aligned}
	\f{d}{dx}\e(x)+\bar\D(x)\leq 0,
	\end{aligned}
	\end{equation*}
where $\bar\D(x):=\sum_{k=0,1}\Big(\|\sqrt{u}\p_x^k\p_{\psi}\phi\|_{L^2_\psi}
+\|\frac{\p_x^k\phi}{u}\|_{L^2_\psi}+\|\p_x^k\p_{\psi}\phi\|_{L^2_\psi}
+\|\frac{\p_x^k\phi}{u^{\f32}}\|_{L^2_\psi}+\|\f{\p_x^{k+1}\phi}{\sqrt{u}}\|_{L^2_\psi}\Big)$.
	Since $u\lesssim1$, by the definition of $\e(x)$ and $\bar\D(x)$, we note that there exists a positive constant independent of $x$ such that $\e(x)\leq C\bar\D(x)$. This, together with the above inequality directly, implies that
	\begin{equation*}
	\begin{aligned}
	\f{d}{dx}\e(x)+\e(x)\leq 0.
	\end{aligned}
	\end{equation*}
	We then employ Gronwall inequality, which yields the decay estimate \eqref{decay} immediately. Therefore, we finish the proof of Proposition \ref{them1}.
\end{proof}

\subsection{Proof of Theorem \ref{asym-beha}}\label{proofthem4}
With the help of Proposition \ref{them1}, we are ready to prove Theorem \ref{asym-beha}.
By virtue of the strong solution $\phi$ obtained in Proposition \ref{them1}, we have
\[u(x,\psi)=\big(\phi(x,\psi)+\us^2(\psi)\big)^{\f12}. \]
Then for a given value of $\psi$, let $y$ denotes the physical variable for the flows $u$, that is
\beqq
y=\int_0^\psi\f{d\psi'}{u(x,\psi')}.
\deqq
Indeed, this transformation is one-to-one between the regious $(x,\psi)\in[0,+\infty)\times[0,+\infty)$ and $(x,y)\in[0,+\infty)\times[0,+\infty)$.
Then for $0\leq x<+\infty$ and $0\leq y<+\infty$, we have
\begin{equation}\label{dy}
u(x,y)=u(x,\psi(x,y)),~~v(x,y)=\int_{0}^{y}u_{x}(x,y')dy'.
\end{equation}
And the pair $(u,v)$ definies a global strong solution to the magnetic Prandtl model \eqref{H-eq}.
Furthermore, we can obtain the strong nonlinear asymptotic stability of classical Hartmann boundary layer in the physical variable, see Theorem \ref{asym-beha}. This is coincide with pointwise perturbations $u(x,y)-\us(y)$ and $v(x,y)$.

More specifically, we have already shown the exponential decay estimate in $x$ variable of $u$ and $\us$ as functions of $(x,\psi)$ in Section \ref{global}.
In order to obtain that $u$ and $\us$, as functions of $(x,y)$, decay in same rate in $x$ variable, we need to investigate the relation between $y$ and $\psi$.
For any given value of $\psi$, we set\begin{equation}\label{y1y2-definion}
y_1:=\int_0^\psi\f{1} {u(x,\psi')}d\psi',\
y_2:=\int_0^\psi \f{1}{\us(\psi')}d\psi',\end{equation}
where $y_1$ and $y_2$ represent the physical variables for the flows $u$ and $\us$, respectively.
As a function of $(x,\psi)$ or $\psi$, it is easy to check that $y_1(x,\psi)\not\equiv y_2(\psi)$.
Clearly, for any given $y_1$, we have
\begin{equation}\label{uus}
\begin{aligned}
|u_{y}(x,y_1)-\us_y(y_1)|\leq |u_{y}(x,y_1)-\us_y(y_2)|+|\us_{y}(y_2)-\us_y(y_1)|.
\end{aligned}
\end{equation}
Recalling the definition \eqref{y1y2-definion} of $y_1$ and $y_2$, then using the fact that $u\lesssim 1$, \eqref{pinty1}, \eqref{beta} and the decay estimate \eqref{decay}, we deduce that
\begin{equation}\label{uusy1y1}
\begin{aligned}
|u_{y}(x,y_1)-\us_y(y_2)|=|uu_{\psi}-\us\us_{\psi}|(x,\psi)\leq C\|\phi_\psi\|_{L^\infty_\psi}\leq C\big(\|\phi\|_{L^2_\psi}+\|\phi_\psi\|_{L^2_\psi}+\|\frac{\phi_x}{\sqrt{u}}\|_{L^2_\psi}\big)\leq Ce^{-x}.
\end{aligned}
\end{equation}
Next return to estimate the second term on the right handside of \eqref{uus}. Without loss of generality, we assume that $y_1<y_2$. Employing Mean value theorem of differential, we conclude that there exists $y^*\in [y_1,y_2]$, such that
\begin{equation}\label{y1y2}
\begin{aligned}
|\us_{y}(y_2)-\us_y(y_1)|\leq C|\us_{yy}(y^*)(y_1-y_2)|\leq C|e^{-y^*}(y_1-y_2)|.
\end{aligned}
\end{equation}
We recall the definition \eqref{y1y2-definion} of $y_1$ and $y_2$, by using the estimates \eqref{u521}, \eqref{beta} and the decay estimate \eqref{decay}, to deduce that\beq\label{y1y22}
\begin{split}
	|e^{-y^*}(y_1-y_2)|\leq& \int_0^\psi\f{e^{-y_1}|u-\us|}{u\us}d\psi'
	\leq C\Big(\int_0^\psi\f{|\om|^2}{u^3}d\psi' \Big)^{\f12}\Big(\int_{0}^{\psi}\f{e^{-2y_1}}{u}d\psi'\Big)^{\f12}\\
	\leq&C \|\f{\om}{u^{\f32}}\|_{L^2_\psi}\Big(\int_0^{y_1}e^{-2y}dy\Big)^{\f12}
	\leq C\|\f{\om}{u^{\f32}}\|_{L^2_\psi}
	\leq C\|\f{\phi}{u^{\f52}}\|_{L^2_\psi}\\
	\leq& C\big(\|\phi\|_{L^2_\psi}+\|\sqrt{u}\phi_{\psi}\|_{L^2_\psi}+\|\phi_x\|_{L^2_\psi}\big)\leq Ce^{-x}.
\end{split}
\deq
Combining \eqref{y1y2} and \eqref{y1y22}, we obtain that
\beq\label{uy1y1}
|\us_{y}(y_2)-\us_y(y_1)|\leq Ce^{-x}.
\deq
It may be a wonder why we can obtain the decay estimate in $x$ variable of $\us$.
Although $\us$ is a function only depending on the space-like variable $y$, $y_1(x,\psi)$ is a function depending on $x$, see the definition \eqref{y1y2-definion}. Thus, we can get the exponential decay estimate \eqref{uy1y1} in $x$ variable of $\us$.
Insert \eqref{uusy1y1} and \eqref{uy1y1} into \eqref{uus}, we find
$$\|u_{y}(x,y)-\us_y(y)\|_{L^\infty_y}\leq Ce^{-x}.$$
It follows from the same way as \eqref{uus} that for any $k=0,1,2$,
\beq\label{uusl2}
\|\p_y^k(u-\us)(x,y_1)\|_{L^2_{y_1}}\leq \|\p_y^ku(x,y_1)-\p_y^k\us(y_2)\|_{L^2_{y_1}}+\|\p_y^k\us(y_2)-\p_y^k\us(y_1)\|_{L^2_{y_1}}.
\deq
Invoking \eqref{y1y2} and \eqref{y1y22}, it is easy to see that the second term on the right handside of \eqref{uusl2} can be controlled as follows:
\begin{equation}\label{uusl22}
\begin{aligned}
&\int_{0}^{y}|\p_y^k\us(y_2)-\p_y^k\us(y_1)|^2dy_1\leq C\int_{0}^{y}|\p_y^{k+1}\us(y^*)(y_1-y_2)|^2dy_1\\
\leq& C\int_{0}^{y}|e^{-y^*}(y_1-y_2)|^2dy_1\leq C \int_{0}^{y}e^{-y_1}\Big|\int_0^\psi\f{e^{-\f12y_1}|u-\us|}{u\us}d\psi'\Big|^2dy_1\\
\leq&C\sup_{y}\Big|\int_0^\psi\f{e^{-\f12y_1}|u-\us|}{u\us}d\psi'\Big|^2\int_{0}^{y}e^{-y_1}dy_1\\
\leq&C \int_0^\psi\f{|\om|^2}{u^3}d\psi'\int_{0}^{\psi}\f{e^{-\f12y_1}}{u}d\psi'\leq \|\f{\om}{u^{\f32}}\|_{L^2_\psi}^2\int_0^{y_1}e^{-\f12y}dy
\leq C\|\f{\phi}{u^{\f52}}\|_{L^2_\psi}^2\leq Ce^{-2x}.
\end{aligned}
\end{equation}
We now proceed to estimate the first term on the right handside of \eqref{uusl2}. We then deduce from the decay estimate \eqref{decay}, the estimates \eqref{u521} and \eqref{beta} that for $k=0$,
\begin{equation*}
\begin{aligned}
& \int_0^{y}|u(x,y_1)-\us(y_2)|^2dy_1
\leq \int_{0}^{\psi}\f{|u(x,\psi')-\us(\psi')|^2}{u}d\psi'
\leq \|\f{\om}{\sqrt{u}}\|_{L^2_\psi}^2+\|\f{\om}{u^{\f32}}\|_{L^2_\psi}^2
\leq C\|\f{\phi}{u^{\f52}}\|_{L^2_\psi}^2\leq Ce^{-2x}.
\end{aligned}
\end{equation*}
For $k=1$, we employ the decay estimate \eqref{decay}, the estimates \eqref{u521} and \eqref{beta} once again, and combine with \eqref{phs1}, to compute
\begin{equation*}
\begin{aligned}
& \int_0^{y}|u_{y}(x,y_1)-\us_{y}(y_2)|^2dy_1
\leq \int_{0}^{\psi}\f{|uu_{\psi'}(x,\psi')-\us\us_{\psi'}(\psi')|^2}{u}d\psi'\\
\leq& C\|\f{\phi_{\psi}}{\sqrt{u}}\|_{L^2_\psi}^2+C\|\f{\phi}{u^{\f52}}\|_{L^2_\psi}^2
\leq C\big(\|\phi\|_{L^2_\psi}^2+\|\phi_\psi\|_{L^2_\psi}^2+\|\phi_x\|_{L^2_\psi}^2\big)\leq Ce^{-2x}.
\end{aligned}
\end{equation*}
For $k=2$, in view of the equation \eqref{u2eq}, the decay estimate \eqref{decay}, the estimates \eqref{u521} and \eqref{beta}, we calculate that
\begin{equation*}
\begin{aligned}
& \int_0^{y}|u_{yy}(x,y_1)-\us_{yy}(y_2)|^2dy_1
\leq \int_{0}^{\psi}\f{|u(u^2)_{\psi'\psi'}-\us(\us)^2_{\psi'\psi'}|^2(x,\psi')}{u(x,\psi')}d\psi'
=\int_{0}^{\psi}\f{|\phi_x+2\om|^2}{u}d\psi'\\
\leq& C\big(\|\f{\phi}{u^{\f52}}\|_{L^2_\psi}^2+\|\f{\phi_x}{\sqrt{u}}\|_{L^2_\psi}^2\big)
\leq C\big(\|\phi\|_{L^2_\psi}^2+\|\phi_\psi\|_{L^2_\psi}^2+\|\f{\phi_x}{\sqrt{u}}\|_{L^2_\psi}^2\big)
\leq Ce^{-2x}.
\end{aligned}
\end{equation*}
By collecting three estimates above, we can get that for $k=0,1,2$,\beq\label{uusl21}
\|\p_y^k\us(y_2)-\p_y^k\us(y_1)\|_{L^2_{y_1}}\leq Ce^{-x}.
\deq
Substituting \eqref{uusl22} and \eqref{uusl21} into \eqref{uusl2}, we then obtain\beqq
\|\p_y^k(u-\us)(x,y_1)\|_{L^2_{y_1}}\leq Ce^{-x},\quad k=0,1,2.
\deqq
Therefore, we complete the proof of Theorem \ref{asym-beha}.

\begin{appendices}

\section{Some usefull estimates}\label{appendix-a}

First, we will state without proof two elementary Hardy type inequalities, refer to
Lemma B.1 in \cite{Masmoudi}.
\begin{lemm}
Let $f: \mathbb{R}^+ \rightarrow \mathbb{R}$ be some proper function. Then\\
$(i)$If $\lambda > -\f12$ and $\underset{y\rightarrow +\infty}{\lim}f(y)=0$, then
\beq\label{Hardy-one}
\|f(1+y)^\lambda f\|_{L^2_y(\mathbb{R}^+)}
\le \f{2}{2\lambda+1}\|f(1+y)^{\lambda+1} \py f\|_{L^2_y(\mathbb{R}^+)};
\deq
$(ii)$ If $\lambda < -\f12$, then
\beq\label{Hardy-two}
\|f(1+y)^\lambda f\|_{L^2_y(\mathbb{R}^+)}
\le  \sqrt{-\f1{2\lambda+1}}|f|_{y=0}
      -\f{1}{2\lambda+1}\|f(1+y)^{\lambda+1} \py f\|_{L^2_y(\mathbb{R}^+)}.
\deq
\end{lemm}

Second, we prove some estimates as follows.
\begin{prop}
Let $(\u, \v)$ be the smooth solution, defined on $[0, L^\eps]$,
to the approximated equations \eqref{app-q}, \eqref{bdyc} and \eqref{ic}.
Under the assumption conditions \eqref{equi-relation} and \eqref{lower-bound},
then it holds for integers $s\ge 0$ and $l \ge 1$ that
\begin{align}
\label{claim1}
&\|\f{\px^{s+1}\u}{\u}\|_{L^2_x L^\infty_y}
  \le C_{k_*,\delta_0}\|\px^s \py \v\|_{L^2_x H^2_y};\\
\label{claim21}
&\|\px^s \py \q \y^l \|_{L^2_x L^2_y}
\le o_L(1)C_{k_*}(\|\u \px^s \py \q \y^l \|_{L^\infty_x L^2_y}
           +\|\sqrt{\u}\px^s \py^2 \q\|_{L^2_x L^2_y});\\
\label{claim22}
&\|\px^s \py \q \y^l \|_{L^\infty_x L^2_y}
\le  \|\px^s \py \q \y^l|_{x=0}\|_{L^2_y}
     +o_L(1)C_{k_*,\delta_0}(\|\u \px^{s+1} \py \q \y^l \|_{L^\infty_x L^2_y}
           +\|\sqrt{\u}\px^{s+1} \py^2 \q\|_{L^2_x L^2_y});\\
\label{claim31}
&\|\px^s \q\|_{L^2_x L^\infty_y}
\le o_L(1)C_{k_*,l}(\|\u \px^s \py \q \y \|_{L^\infty_x L^2_y}
           +\|\sqrt{\u}\px^s \py^2 \q\|_{L^2_x L^2_y}); \\
\label{claim32}
&\|\px^s \q \|_{L^\infty_x L^\infty_y}
\le \|\px^s \py \q\y|_{x=0}\|_{L^2_y}
     +o_L(1)C_{k_*,l,\delta_0}(\|\u \px^{s+1} \py \q \y^l \|_{L^\infty_x L^2_y}
           +\|\sqrt{\u}\px^{s+1} \py^2 \q\|_{L^2_x L^2_y});\\
\label{claim4}
&\|\px^s \py^2 \q \y^l\|_{L^2_x L^2_y}
\le C_{k_*, \d_0}\|\sqrt{\u} \px^s \py^2 \q \y^l\|_{L^2_x L^2_y}
    +C_{k_*,k^*}\|\px^s \py^3 \v\|_{L^2_x L^2_y}.
\end{align}
Here we require $L$ satisfying $L\le \delta_0^2$ in estimates
\eqref{claim21} and \eqref{claim31}.
\end{prop}

\begin{proof}
Step 1:
For any integer $s\geq 0$, it holds for
any cut off function $\chi(\f{y}{\delta_0})$(see definition in \eqref{f-cutoff}):
\beq\label{a11}
\bal
\|\f{\px^{s+1} \u}{\u}\|_{L^\infty_y}
\le \|\f{\px^{s+1} \u}{\u}\chi(\f y {\delta_0})\|_{L^\infty_y}
    +\|\f{\px^{s+1} \u}{\u}[1-\chi(\f{y}{\delta_0})]\|_{L^\infty_y}.
\dal
\deq
Using the equivalent relation \eqref{equi-relation}, it is easy to check that
\beq\label{a12}
\bal
\|\f{\px^{s+1} \u}{\u}\chi(\f y {\delta_0})\|_{L^\infty_y}
&\le \|\f{\px^{s+1} \u}{k_*(y+\eps)}\chi(\f y {\delta_0})\|_{L^\infty_y}
\le \f{1}{k_*}\|\f{\px^{s+1} \u}{y}\chi(\f y {\delta_0})\|_{L^\infty_y}\\
&\le  C_{k_*}\|\px^{s+1}\py \u\chi(\f y \delta_0)\|_{L^\infty_y}
      +\f{C_{k_*}}{\delta_0}\|\px^{s+1} \u\chi'(\f y {\delta_0})\|_{L^\infty_y},
\dal
\deq
and
\beq\label{a13}
\|\f{\px^{s+1} \u}{\u}[1-\chi(\f y {\delta_0})]\|_{L^\infty_y}
\le \f{1}{k_*(\delta_0+\eps)}\|{\px^{s+1} \u}[1-\chi(\f y {\delta_0})]\|_{L^\infty_y}
\le \f{C_{k_*}}{\delta_0}\|{\px^{s+1} \u}[1-\chi(\f y {\delta_0})]\|_{L^\infty_y}.
\deq
Substituting estimates \eqref{a12} and \eqref{a13} into \eqref{a11},
and applying the Sobolev embedding inequality, it holds true
\beqq
\|\f{\px^{s+1}\u}{\u}\|_{L^\infty_y}
\le C_{k_*,\delta_0}(\|\px^s \py \v\|_{L^\infty_y}+\|\px^s \py^2 \v\|_{L^\infty_y})
\le C_{k_*,\delta_0}(\|\px^s \py \v\|_{L^2_y}
        +\|\px^s \py^2 \v\|_{L^2_y}+\|\px^k \py^3 \v\|_{L^2_y}),
\deqq
which implies directly
\beqq
\|\f{\px^{s+1}\u}{\u}\|_{L^2_x L^\infty_y}
\le C_{k_*,\delta_0}\|\px^s \py \v\|_{L^2_x H^2_y}.
\deqq
Therefore, we complete the proof of estimate \eqref{claim1}.

Step 2: For any integer $s\geq 0$ and $\delta \le \delta_0$,
we can apply the estimate \eqref{lower-bound} to get that
\beq\label{a14}
\|\px^s \py \q \y^l\|_{L^2_y}
\le \|\px^s \py \q \y^l \chi(\f{y}{\delta})\|_{L^2_y}
    +\|\px^s \py \q \y^l(1-\chi(\f{y}{\delta}))\|_{L^2_y},
\deq
and
\beq\label{a15}
\|\px^s \py \q \y^l(1-\chi(\f{y}{\delta}))\|_{L^2_y}
\le \f1{k_*(\d+\eps)}\|\u \px^s \py \q \y^l (1-\chi(\f{y}{\delta}))\|_{L^2_y}.
\deq
Integrating by part, we can get that
\beqq
\bal
&\|\px^s \py \q \y^l \chi(\f{y}{\delta})\|_{L^2_y}^2
\le \|\px^s \py \q \chi(\f{y}{\delta})\|_{L^2_y}^2
=\int_0^\infty |\px^s \py \q|^2 \chi(\f{y}{\delta})^2 \py\{y\}dy\\
&=-\int_0^\infty 2 \px^s \py \q \px^s \py^2 \q \chi(\f{y}{\delta})^2y dy
  -\int_0^\infty 2 |\px^s \py \q|^2 \chi (\f{y}{\delta})
  \chi'(\f{y}{\delta}) \f{y}{\delta}dy,
\dal
\deqq
which, together with the estimate \eqref{equi-relation}, yields directly
\beq\label{a16}
\bal
\|\px^s \py \q \y^l \chi\|_{L^2_y}
&\le C(\sqrt{\d} \|\sqrt{y}\px^s \py^2 \q \chi(\f{y}{\delta})\|_{L^2_y}
    +{\d}^{-1}\|y \px^s \py \q \chi(\f{y}{\delta})\|_{L^2_y})\\
&\le C_{k_*}({\d}^{-1}\|\u \px^s \py \q \chi(\f{y}{\delta})\|_{L^2_y}
             +\sqrt{\d} \|\sqrt{\u}\px^s \py^2 \q \chi(\f{y}{\delta})\|_{L^2_y}).
\dal
\deq
Substituting estimates \eqref{a16} and \eqref{a15} into \eqref{a14}, we have
\beq\label{a1}
\|\px^s \py \q \y^l \|_{L^2_y}
\le C_{k_*}(\d^{-1}\|\u\px^s \py \q \y^l \|_{L^2_y}
            +\sqrt{\d} \|\sqrt{\u} \px^s \py^2 \q\|_{L^2_y}),
\deq
which, integrating over $[0, L]$, implies that
\beqq
\|\px^s \py \q \y^l \|_{L^2_x L^2_y}
\le C_{k_*}(\d^{-1}L\|\u \px^s \py \q \y^l \|_{L^\infty_x L^2_y}
    +\sqrt{\d} \|\sqrt{\u}\px^s \py^2 \q\|_{L^2_x L^2_y}).
\deqq
Choosing $\d=\sqrt{L}\le \delta_0$ in the above inequality,
then we have
\beqq
\|\px^s \py \q \y^l \|_{L^2_x L^2_y}
\le o_L(1)C_{k_*}(\|\u \px^s \py \q \y^l \|_{L^\infty_x L^2_y}
           +\|\sqrt{\u}\px^s \py^2 \q\|_{L^2_x L^2_y}).
\deqq
On the other hand, it is easy to check that
\beqq
\|\px^s \py \q \y^l\|_{L^\infty_x L^2_y}
\le \|\px^s \py \q_0 \y^l\|_{L^2_y}
    +o_L(1)\|\px^{s+1} \py \q \y^l\|_{L^2_x L^2_y}.
\deqq
This, together with the estimate \eqref{a1} and $\d=\d_0$, yields directly
\beqq
\|\px^s \py \q \y^l \|_{L^\infty_x L^2_y}
\le  \|\px^s \py \q_0 \y^l\|_{L^2_y}
     +o_L(1)C_{k_*,\delta_0}(\|\u \px^{s+1} \py \q \y^l \|_{L^\infty_x L^2_y}
           +\|\sqrt{\u}\px^{s+1} \py^2 \q\|_{L^2_x L^2_y}).
\deqq
Thus, we complete the proof of claimed estimates \eqref{claim21} and \eqref{claim22}.

Step 3: It is easy to check that
\beqq
\|\px^s \q\|_{L^\infty_y}^2
\le
2\iy |\px^s \q||\px^s\py \q|dy
\le
 2\|\px^s \q\y^{-1} \|_{L^2_y}\|\px^s \py \q\y\|_{L^2_y}.
\deqq
Using the Hardy inequality \eqref{Hardy-two}, it holds true
\beqq
\|\px^s \q\y^{-1} \|_{L^2_y} \le C\|\px^s \py \q \|_{L^2_y}.
\deqq
Thus, we get that
\beqq
\|\px^s \q\|_{L^\infty_y} \le C_l\|\px^s \py \q\y \|_{L^2_y}.
\deqq
This together with the estimates \eqref{claim21} and \eqref{claim22} gives respectively
\beqq
\|\px^s \q\|_{L^\infty_x L^\infty_y}
\le \|\px^s \py \q \y|_{x=0}\|_{L^2_y}
     +o_L(1)C_{k_*,l,\delta_0}(\|\u \px^{s+1} \py \q \y\|_{L^\infty_x L^2_y}
           +\|\sqrt{\u}\px^{s+1} \py^2 \q\|_{L^2_x L^2_y}),
\deqq
and
\beqq
\|\px^s \q \|_{L^2_x L^\infty_y}
\le o_L(1)C_{k_*, l}(\|\u \px^s \py \q \y \|_{L^\infty_x L^2_y}
           +\|\sqrt{\u}\px^s \py^2 \q\|_{L^2_x L^2_y}),
\deqq
for all $L\le \delta_0^2$.
Therefore, we complete the proof the claimed estimates \eqref{claim31} and \eqref{claim32}.

Step 4: For any inteter $s\ge 0$, it holds
\beqq
\px^s \py^2 \q \y^l
=\px^s \py^2 \q \y^l\chi(\f{y}{\delta})
 +\px^s \py^2 \q \y^l (1-\chi(\f{y}{\delta})).
\deqq
Using the assumptions \eqref{equi-relation} and \eqref{lower-bound},
it is easy to check that
\beqq
\|\px^s \py^2 \q \y^l (1-\chi(\f{y}{\delta}))\|_{L^2_y}
\le \f{1}{\{k_*(\delta+\eps)\}^{\f12}}\|\sqrt{u}
\px^s \py^2 \q \y^l (1-\chi(\f{y}{\delta}))\|_{L^2_y},
\deqq
and
\beqq
\|\px^s\py^2 \q \y^l\chi(\f{y}{\delta})\|_{L^2_y}
\le \|\px^s \py^2 \q \chi(\f{y}{\delta})\|_{L^2_y}
\le C_{k_*,k^*}\|\px^s \py^3 \v \|_{L^2_y}.
\deqq
Therefore, we can get that
\beqq
\|\px^s \py^2 \q \y^l\|_{L^2_y}
\le \f{C_{k_*}}{\d^{\f12}}
     \|\sqrt{u} \px^s \py^2 \q \y^l (1-\chi(\f{y}{\delta}))\|_{L^2_y}
    +C_{k_*,k^*}\|\px^s \py^3 \v\|_{L^2_y},
\deqq
which, integrating over $[0, L]$ and choosing $\d=\d_0$, yields directly
\beqq
\|\px^s \py^2 \q \y^l\|_{L^2_x L^2_y}
\le {C_{k_*, \d_0}}\|\sqrt{\u} \px^s \py^2 \q \y^l\|_{L^2_x L^2_y}
    +C_{k_*,k^*}\|\px^s \py^3 \v\|_{L^2_x L^2_y},
\deqq
which implies the claimed estimate \eqref{claim4}.
Therefore, we complete the proof of this lemma.
\end{proof}

Let us define
\beq\label{disp-v}
V^k_l(x):=\|\px^{\la k \ra}\py \v \y^l\|_{L^2_x L^2_y}
              +\|\px^{\la k \ra}\py^2 \v \y^l\|_{L^2_x L^2_y}
              +\|\px^{\la k-1\ra}\py^3 \v \y^l\|_{L^2_x L^2_y},
\deq
and
\beq
\widehat{V}^k_l(0):=\|\py^{\la 2k-2\ra}\py^2 \u \y^l |_{x=0}\|_{L^2_y}
         +\sum_{\beta \ge 1, 2\alpha +\beta \le 2(k-1)}
              \|\px^\alpha \py^{\beta+2} \v \y^l|_{x=0}\|_{L^2_y}.
\deq

\begin{prop}\label{p3}
For any smooth solution $(\u, \v)$ of equation \eqref{app-q},
then the following estimates hold true
\begin{align}
\label{237}&\|\px^{\la k-1 \ra}\py^4 \v \y^l\|_{L^2_x L^2_y}
           \le C_{k,l}(1+\b^k_l(0)^2+V^k_l(x)^2);\\
\label{238}&\|\px^\alpha \py^{\beta+2}\v\y^l\|_{L^2_x L^2_y}
           \le C_{k,l}\mathcal{P}(1+\b^k_l(0)^2+\widehat{V}^k_l(0)^2+V^k_l(x)^2),
\end{align}
where any integers $\alpha$ and $\beta$ satisfying $0\le 2\alpha+\beta \le 2k$.
\end{prop}

\begin{proof}

Step 1: Applying the differential operator $\px^{k-1}$
to the equation \eqref{app-q}, we have
\beqq
\px^{k-1}\py^4 \v=\px^{k-1} \py^2 \v+\px^k(\u \py^2 \v)-\px^k(\v \py^2 \u):=J_1+J_2+J_3.
\deqq
Using the divergence-free condition, it is easy to check that
\beqq
\bal
\|J_2 \y^l\|_{L^2_x L^2_y}
\le &\|\u\px^k \py^2 \v \y^l\|_{L^2_x L^2_y}
      +\sum_{j=1}^k C^k_j \|\px^{j-1}\py \v \px^{k-j} \py^2 \v \y^l\|_{L^2_x L^2_y}\\
\le &(1+\|\py \u\y|_{x=0}\|_{L^2_y}+\|\py^2 \v \y\|_{L^2_x L^2_y})
      \|\px^k \py^2 \v \y^l\|_{L^2_x L^2_y}\\
    &+C_k \|\px^{\la [\f{k-1}{2}]\ra} \py \v\|_{L^\infty_x L^\infty_y}
           \|\px^{\la k-1 \ra}\py^2 \v \y^l\|_{L^2_x L^2_y}\\
    &+C_k \|\px^{\la k-1 \ra}\py \v \|_{L^2_x L^\infty_y}
            \|\px^{\la [\f{k-1}{2}]\ra} \py^2 \v \y^l\|_{L^\infty_x L^2_y}\\
\le &C_k(1+\|\py \u\y|_{x=0}\|_{L^2_y}^2+\|\px^{\la k-2 \ra}\py^2 \v \y|_{x=0}\|_{L^2_y}^2)
     +C_k\|\px^{\la k \ra} \py^2 \v \y^l\|_{L^2_x L^2_y}^2,
\dal
\deqq
and
\beqq
\bal
\|J_3 \y^l\|_{L^2_x L^2_y}
\le &\|\py^2 \u \y^l \|_{L^\infty_x L^2_y}\|\px^k \v\|_{L^2_x L^\infty_y}
     +\sum_{j=0}^{k-1}C^k_j \|\px^j \v \px^{k-1-j} \py^3 \v \y^l\|_{L^2_x L^2_y}\\
\le &C_l(\|\py^2 \u \y^l|_{x=0}\|_{L^2_y}+\|\py^3 \v \y^l\|_{L^2_x L^2_y})
        \|\px^k \py \v \y\|_{L^2_x L^2_y}\\
    &+C_k \|\px^{\la [\f{k-1}{2}]\ra}\v\|_{L^\infty_x L^\infty_y}
          \|\px^{\la k-1 \ra}\py^3 \v \y^l\|_{L^2_x L^2_y}\\
    &+C_k \|\px^{\la k-1 \ra} \v \|_{L^2_x L^\infty_y}
          \|\px^{\la [\f{k-1}{2}]\ra}\py^3 \v \y^l\|_{L^\infty_x L^2_y}\\
\le &C_{k,l}(1+\|\py^2 \u \y^l|_{x=0}\|_{L^2_y}^2
              +\|\px^{\la k-2\ra}\py \v \y^l|_{x=0}\|_{L^2_y}^2
              +\|\px^{\la k-2\ra}\py^3 \v\y^l|_{x=0}\|_{L^2_y}^2)\\
    &+C_{k,l}(\|\px^{\la k \ra}\py \v \y\|_{L^2_x L^2_y}^2
              +\|\px^{\la k-1\ra}\py^3 \v \y^l\|_{L^2_x L^2_y}^2).
\dal
\deqq
Therefore, we can get that for all $l \ge 1$
\beqq
\bal
&\|\px^{k-1}\py^4 \v \y^l\|_{L^2_x L^2_y}\\
\le &C_{k,l}(1+\|\py \u\y^l|_{x=0}\|_{L^2_y}^2
             +\|\py^2 \u \y^l|_{x=0}\|_{L^2_y}^2)\\
    &+C_{k,l}( \|\px^{\la k-2\ra}\py \v\y^l|_{x=0}\|_{L^2_y}^2
              +\|\px^{\la k-2 \ra}\py^2 \v \y^l|_{x=0}\|_{L^2_y}^2
              +\|\px^{\la k-2\ra}\py^3 \v\y^l|_{x=0}\|_{L^2_y}^2)\\
    &+C_{k,l}(\|\px^{\la k \ra}\py \v \y^l\|_{L^2_x L^2_y}^2
              +\|\px^{\la k \ra}\py^2 \v \y^l\|_{L^2_x L^2_y}^2
              +\|\px^{\la k-1\ra}\py^3 \v \y^l\|_{L^2_x L^2_y}^2).
\dal
\deqq
This implies the estimate \eqref{237}.

Step 2: We will give the proof of estimate \eqref{238} by induction.
Indeed, the estimate \eqref{237} implies that \eqref{238} holds true for
the case of $\beta=0, 1, 2$.
Now we suppose the estimate \eqref{238} holds true for the case
$\beta \le 2n$, i.e.,
\beq\label{est-m}
\|\px^\alpha \py^{\beta+2}\v\y^l\|_{L^2_x L^2_y}
           \le C_{k,l}\mathcal{P}(1+\b^k_l(0)^2+V^k_l(x)^2+\widehat{V}^k_l(0)^2),
\deq
for $2 \le 2n \le 2(k-1)$.
This and the relation $2\alpha+2n \le 2k$ implies the integer $\alpha \in [0, k-1]$.
We need to verify that \eqref{238} holds on for case of $\beta \le 2n+1$
and $\beta \le 2n+2$, i.e.
\beq\label{est-m-1}
\|\px^{\alpha-1} \py^{\beta+2}\v\y^l\|_{L^2_x L^2_y}
           \le C_{k,l}\mathcal{P}(1+\b^k_l(0)^2+V^k_l(x)^2+\widehat{V}^k_l(0)^2),
\deq
Indeed, the estimate \eqref{est-m} implies that \eqref{est-m-1}
holds true for the case $\beta \le 2n$.
we just need to prove the estimate \eqref{est-m-1} holds on
for the case $\beta=2n+1$ and $\beta=2n+2$.
Applying differential operator $\px^{\alpha-1} \py^{2n}$ to
the equation \eqref{app-q}, then we have
\beqq
\px^{\alpha-1} \py^{2n+3} \v=\px^{\alpha-1} \py^{2n+1} \v
+\px^{\alpha} \py^{2n}(\u \py \v)
-\px^{\alpha} \py^{2n}(\v \py \u)
:=K_1+K_2+K_3.
\deqq
Using the H\"{o}lder inequality and divergence-free condition, then we have
\beqq
\bal
\|K_2 \y^l \|_{L^2_x L^2_y}
\le & C_k \sum_{j=0}^{2n}
       \|\py^j \u \px^{\alpha} \py^{2n-j} \py \v \y^l \|_{L^2_x L^2_y}
       +C_k \sum_{i=1}^\alpha \sum_{j=0}^{2n}
          \|\px^i \py^j \u \px^{\alpha-i} \py^{2n-j} \py \v \y^l \|_{L^2_x L^2_y}\\
\le &  C_k \|\py^{2n} \u \y^l \|_{L^\infty_x L^2_y}
            \|\px^{\alpha} \py^{2n+1} \v \|_{L^2_x L^\infty_y}
       +C_k \|\px^{\alpha-1} \py^{2n+1} \v \y^l \|_{L^\infty_x L^2_y}
          \|\px^{\alpha-1} \py^{2n+1} \v \|_{L^2_x L^\infty_y}\\
\le &  C_k (\|\py^{2n} \u \y^l|_{x=0}\|_{L^2_y}
            +\|\py^{2n+1} \v \y^l \|_{L^2_x L^2_y})
            \|\px^{\alpha} \py^{2n+2} \v \y\|_{L^2_x L^2_y}\\
    &  +C_k (\|\px^{\alpha-1} \py^{2n+1} \v \y^l|_{x=0}\|_{L^2_y}
             +\|\px^{\alpha} \py^{2n+1} \v \y^l\|_{L^2_x L^2_y})
             \|\px^{\alpha-1} \py^{2n+2} \v \y \|_{L^2_x L^2_y},
\dal
\deqq
and
\beqq
\bal
\|K_3 \y^l \|_{L^2_x L^2_y}
\le &C_k \sum_{j=0}^{2n}\|\px^\alpha \py^j \v  \py^{2n-j}\py \u \y^l\|_{L^2_x L^2_y}
          +C_k\sum_{i=0}^{\alpha-1}\sum_{j=0}^{2n}
          \|\px^i \py^j \v \px^{\alpha-i} \py^{2n-j}\py \u \y^l\|_{L^2_x L^2_y}\\
\le &C_k  \|\px^\alpha \py^{2n}\v  \|_{L^2_x L^\infty_y}
          \|\py^{2n+1} \u \y^l\|_{L^\infty_x L^2_y}
          +C_k \|\px^{\alpha-1} \py^{2n} \v \|_{L^\infty_x L^\infty_y}
          \|\px^{\alpha-1}\py^{2n+2}\v \y^l\|_{L^2_x L^2_y}\\
\le &C_k  \|\px^\alpha \py^{2n+1}\v  \y\|_{L^2_x L^2_y}
          (\|\py^{2n+1} \u \y^l|_{x=0}\|_{L^2_y}+\|\py^{2n+2} \v \y^l\|_{L^2_x L^2_y})\\
    &    +C_k (\|\px^{\alpha-1} \py^{2n+1} \v \y|_{x=0}\|_{L^2_y}
                +\|\px^{\alpha} \py^{2n+1} \v \y\|_{L^2_x L^2_y})
               \|\px^{\alpha-1}\py^{2n+2}\v \y^l\|_{L^2_x L^2_y}.
\dal
\deqq
Then, we can get that
\beq\label{a21}
\bal
&\|\px^{\alpha-1} \py^{2n+3} \v \y^l\|_{L^2_x L^2_y}\\
\le &C_{k,l}(1+\|\py^{2n+1} \u \y^l|_{x=0}\|_{L^2_y}^2
             +\|\px^{\alpha-1} \py^{2n+1} \v \y^l|_{x=0}\|_{L^2_y}^2)
    +C_{k, l}\|\px^{\alpha} \py^{\la 2n+2 \ra} \v \y^l \|_{L^2_x L^2_y}^2.
\dal
\deq
Applying differential operator $\px^{\alpha-1} \py^{2n+1}$  to
the equation \eqref{app-q}, then we have
\beqq
\px^{\alpha-1} \py^{2n+4} \v=\px^{\alpha-1} \py^{2n+3} \v
+\px^{\alpha} \py^{2n+1}(\u \py \v)
-\px^{\alpha} \py^{2n+1}(\v \py \u).
\deqq
Then, similar to the estimates of terms $K_2$ and $K_3$,
it is easy to check that
\beq\label{a22}
\bal
\|\px^{\alpha-1} \py^{2n+4} \v \y^l\|_{L^2_x L^2_y}
\le &C_{k,l}(1+\|\py^{2n+2} \u \y^l|_{x=0}\|_{L^2_y}^2
             +\|\px^{\alpha-1} \py^{2n+2} \v \y^l|_{x=0}\|_{L^2_y}^2)\\
    &+C_{k, l}(\|\px^{\alpha} \py^{\la 2n+2 \ra} \v \y^l \|_{L^2_x L^2_y}^2
               +\|\px^{\alpha-1} \py^{2n+3} \v \y^l\|_{L^2_x L^2_y}^2).
\dal
\deq
Thus, the estimates \eqref{a21} and \eqref{a22} imply estimate \eqref{est-m-1}
holds true for $\beta=2n+1$ and $\beta=2n+2$.
Therefore, we complete the proof of lemma.
\end{proof}

\section{Compatibility of initial data}\label{appendix-b}

In this section, we will introduce the definition of compatibility
of initial data at the corner $(0, 0)$. This can explain the reason
that the initial data required in Theorem \ref{high-regu}.
First of all, by evaluating the equation \eqref{H-eq} at $y=0$,
we obtain the boundary condition:
\beqq
(-\py^2 u-1)|_{y=0}=0.
\deqq
Thus, we require the initial data $u_0(y)$ satisfies the compatibility condition
\beq\label{b1}
(-\py^2 u_0(y)-1)|_{y=0}=0.
\deq
Then, we use the equation \eqref{H-eq} to deduce that
$u \py(\f{v}{u})|_{x=0}\in L^2(\mathbb{R}^+)$.
Taking $\py$ operator to the equation \eqref{H-eq}
and applying the divergence-free condition, we have
\beq\label{b5}
u\px \py u+ v \py^2 u-\py^3 u+\py u=0.
\deq
Evaluating the above equation at $y=0$, then we have
\beqq
(-\py^3 u+ \py u)|_{y=0}=0.
\deqq
Thus, we require the initial data $u_0(y)$ satisfies the compatibility condition
\beq\label{b2}
(-\py^3 u_0(y)+ \py u_0(y))|_{y=0}=0.
\deq
Furthermore, we also have that $\py(\f{v}{u})|_{x=0}\in L^2(\mathbb{R}^+)$.
The above two compatibilities on initial data is devoted to
ensuring that the quantities $u\py(\f{v}{u})|_{x=0}$ and $\py(\f{v}{u})|_{x=0}$
is not singularity at the point $y=0$.

Secondly, using the equation \eqref{H-eq}, we have
\beqq
v(x, y)=u(x, y)\int_0^y \f{-\py^2 u(x, y')+u(x, y')-1}{[u(x, y')]^2}d y',
\deqq
which implies
\beq
v_0(y):=v(0, y)=u_0(y)\int_0^y \f{-\py^2 u_0(y')+u_0(y')-1}{[u_0(y')]^2}d y'.
\deq
By taking $\px$ differential operator to the equation
\eqref{H-eq} and using the divergence-free condition, we obtain
\beq\label{b6}
\px(u \px u+ v \py u)+\py^3 v-\py v=0,
\deq
which, evaluating at $y=0$, yields directly
\beqq
(\py^3 v-\py v)|_{y=0}=0.
\deqq
The first order compatibility condition that arises from this
entails matching now the initial data evaluated at $y=0$ with
boundary data evaluated at $x=0$ in the standard manner for
initial boundary value problem:
\beq\label{b3}
\{\py^3 v_0(y)-\py v_0(y)\}|_{y=0}=0.
\deq
On the other hand, using the divergence-free condition, we can
rewrite the equation as the form
\beqq
u^2 \partial_{xy} \left(\f{v}{u}\right)
=2 u \py v \py \left(\f{v}{u} \right)+\py^3 v-\py v.
\deqq
It is clear that all quantities on the righthand side of equation \eqref{b6}
are vanishing at $y=0$.
Then, the first compatibility on initial data implies directly
\beqq
u \partial_{xy} \left(\f{v}{u}\right)|_{x=0}\in L^2(\mathbb{R}^+).
\deqq
Taking $\px$ operator of equation \eqref{b5} and using the
divergence-free condition, then we have
\beqq
\px(u\px \py u+ v \py^2 u)+\py^4 v-\py^2 v=0.
\deqq
which, evaluating at $y=0$, yields directly
\beqq
(\py^4 v-\py^2 v)|_{y=0}=0.
\deqq
This then gives our second order compatibility condition:
\beq\label{b4}
\{\py^4 v_0(y)-\py^2 v_0(y)\}|_{y=0}=0.
\deq
Taking $\py$ operator to the equation, we have
\beq\label{b7}
\py\left\{u^2 \partial_{xy} \left(\f{v}{u}\right)\right\}
=\py \left\{ 2 u \py v \py \left(\f{v}{u} \right)\right\}+\py^4 v-\py^2 v.
\deq
It is clear that all quantities on the righthand side of equation \eqref{b7}
are vanishing at $y=0$.
Then, the second compatibility on initial data implies directly
\beqq
\partial_{xy} \left(\f{v}{u}\right)|_{x=0}\in L^2(\mathbb{R}^+).
\deqq
Therefore, we can define the higher order compatibility conditions at the
corner $(0, 0)$ in the same manner, and can be stated as follows
\beq\label{comp-k1}
\left\{\py^3 (\px^{m}v|_{x=0})-\py (\px^{m}v|_{x=0})\right\}|_{y=0}=0,
\deq
and
\beq\label{comp-k2}
\left\{\py^4 (\px^{m}v|_{x=0})-\py^2 (\px^{m}v|_{x=0})\right\}|_{y=0}=0,
\deq
for all $m \ge 0$.
It is worth noticing that the higher order $x-$derivative of $v$
at $x=0$ can be controlled by the lower one.
The compatibility conditions \eqref{comp-k1} and \eqref{comp-k2}
on the initial data are called the $(2m+1)-$th and $(2m+2)-$th order
compatibility conditions at the corner $(0, 0)$.
According to the initial norms ${\mathcal{B}}^k_l(0)$
and ${\mathcal{E}}^k_l(0)$(see the definitions in \eqref{energy-e}
and \eqref{initial-data}), we require the generic compatibility
conditions at the corner $(0, 0)$ up to order $2k-1$ in Theorem \ref{high-regu}.
Finally, we explain the initial data $u_0$ satisfying the
compatibility condition. Obviously, we have
\beqq
\py v=\py u \int_0^y \f{-\py^2 u+u-1}{(u)^2}d\tau
        + \f{-\py^2 u+u-1}{u},
\deqq
and hence, it holds
\beqq
\py v_0|_{y=0}
=\left\{\py u_0 \int_0^y \f{-\py^2 u_0+u_0-1}{(u_0)^2}d\tau
        + \f{-\py^2 u_0+u_0-1}{u_0}\right\}_{y=0}=0,
\deqq
where we have used the compatibility conditions \eqref{b1} and \eqref{b2}.
This and the compatibility condition \eqref{b3} implies
$\py^3 v_0|_{y=0}=0$. Recall that
\beqq
\py^3 v_0
=\py^3 u_0 \int_0^y \f{-\py^2 u_0+u_0-1}{u_0^2}d\tau
  +\py^2 u_0\f{-\py^2 u_0+u_0-1}{u_0^2}
  -\py u_0 \f{-\py^3 u_0+\py u_0}{u_0^2}
  +\f{-\py^4 u_0+\py^2 u_0}{u_0},
\deqq
then we require the initial data of $u_0$ to satisfy
\beqq
(-\py^4 u_0+\py^2 u_0)|_{y=0}=(-\py^5 u_0+\py^3 u_0)|_{y=0}=0.
\deqq
Thus, this ensures $\py^3 v_0|_{y=0}=0$ that will satisfy the compatibility condition.
Similarly, we require the initial data $u_0$ satisfying
\beqq
(-\py^6 u_0+\py^4 u_0)|_{y=0}=(-\py^7 u_0+\py^5 u_0)|_{y=0}=0,
\deqq
to ensure the compatibility \eqref{b4}.
Similarly, we can add conditions to the initial data $u_0$ to satisfy
the compatibility conditions \eqref{comp-k1} and \eqref{comp-k2}.

\section{Control of initial data}\label{appendix-c}

In this section, we will prove that the initial condition
$\b^k_l(0)+\e^k_l(0)$ can be controlled be the initial data
$u_0$, and hence is independent of $\eps$.
That is to give the proof for the claim estimate as follows
(see \eqref{initial-control})
\beq\label{initial-control-01}
\e^k_l(0)+\b^k_l(0)\le C(u_0),
\deq
where the constant $C(u_0)$ only depends on the initial data $u_0$.
\begin{proof}
First of all, it is easy to check that
\beqq
|\py(\f{\v}{\u})|_{x=0}
=|\f{-\py^2 \u+\u-1-\eps}{(\u)^2}|_{x=0}
=|\f{-\py^2 u_0+u_0-1}{(u_0+\eps)^2}|
\le |\f{-\py^2 u_0+u_0-1}{u_0^2}|.
\deqq
This, together with the compatibility conditions
\eqref{b1} and \eqref{b2}, yields directly
\beq\label{ac1}
\|\py(\f{\v}{\u})\y^l|_{x=0}\|_{L^2_y}\le C(u_0).
\deq

Secondly, according to the equation \eqref{app-sym}, we have
\beqq
\v=\u \int_0^y \f{-\py^2 \u+\u-1-\eps}{(\u)^2}d\tau,
\deqq
which implies directly
\beqq
\py \v=\py \u \int_0^y \f{-\py^2 \u+\u-1-\eps}{(\u)^2}d\tau
        + \f{-\py^2 \u+\u-1-\eps}{\u}.
\deqq
Then, it holds
\beqq
\bal
|\py \v|_{x=0}
&\le |\py u_0| \int_0^y |\f{-\py^2 u_0+u_0-1}{(u_0+\eps)^2}|d\tau
        + |\f{-\py^2 u_0+u_0-1}{u_0+\eps}|\\
&\le |\py u_0| \int_0^y |\f{-\py^2 u_0+u_0-1}{(u_0)^2}|d\tau
        + |\f{-\py^2 u_0+u_0-1}{u_0}|,
\dal
\deqq
which, together with compatibility condition \eqref{b2}, yields directly
\beq\label{ac2}
\|\py \v|_{x=0}\|_{L^2_y}\le C(u_0).
\deq
Similarly, it is easy to deduce that
\beq\label{ac3}
\|\py^2 \v|_{x=0}\|_{L^2_y}+\|\py^3 \v|_{x=0}\|_{L^2_y}\le C(u_0).
\deq
It is easy to check that
\beqq
\bal
\partial_{xy}\v
=&-\py^2 \v \int_0^y \f{-\py^2 \u+\u-1-\eps}{(\u)^2}d\tau
  +\f{(\py^3 \v - \py \v)\u+(-\py^2 \u+\u-\eps-1)\py \v}{(\u)^2}\\
 &+\py \u \int_0^y \f{(\py^3 \v -\py \v)(\u)^2+2 \u (-\py^2 \u+\u-1-\eps)\py \v}{(\u)^4}d\tau.
\dal
\deqq
This, together with the estimates \eqref{ac1}-\eqref{ac3}
and compatibility condition \eqref{b4}, yields directly
\beq\label{ac4}
\|\partial_{xy} \v|_{x=0}\|_{L^2_y}\le C(u_0).
\deq
Similarly, it is easy to deduce that
\beq\label{ac5}
\|\px \py^2 \v|_{x=0}\|_{L^2_y}+\|\px \py^3 \v|_{x=0}\|_{L^2_y}\le C(u_0).
\deq

Finally, using the equation \eqref{app-sym}, it is easy to check that
\beqq
\bal
|\u \partial_{xy}(\f{\v}{\u})|_{x=0}
&\le |\f{2\u \py(\f{\v}{\u})\py \v+\py^3 \v-\py \v}{\u}|_{x=0}\\
&\le 2|\py(\f{\v}{\u})|_{x=0}|\py \v|_{x=0}
      +|\f{(\py^3 \v-\py \v)}{u_0+\eps}|_{x=0},
\dal
\deqq
This, together with the estimates \eqref{ac1}-\eqref{ac3}
and compatibility condition \eqref{b4}, yields directly
\beqq
\|\u \partial_{xy}(\f{\v}{\u})|_{x=0}\|_{L^2} \le C(u_0).
\deqq
Similarly, we can also obtain
\beqq
\|\partial_{xy}(\f{\v}{\u})|_{x=0}\|_{L^2} \le C(u_0).
\deqq
The other terms in $\e^k_l(0)$ and $\b^k_l(0)$
can be controlled by the similar way, for the sake of
simplicity, we refrain from writing. Therefore, we complete the proof of
the claim estimate \eqref{initial-control-01}.
\end{proof}
\end{appendices}

\section*{Acknowledgements}

Jincheng Gao's research was partially supported by NNSF of China (11801586),
Guangzhou Science and technology project(202102020769),
and Guangdong Basic and Applied Basic Research Foundation(2020A1515110942).
Zheng-an Yao's research was partially supported by NNSF of China (11971496, 12026244),
and National Key Research and Development Program of China(2020YFA0712500).

\phantomsection

\end{document}